\documentclass[11pt]{amsart} \usepackage{pb-diagram,amssymb}
\usepackage{graphics,epsfig,psfrag}
\usepackage{hyperref} 
\hypersetup{colorlinks}
\usepackage{color}
\definecolor{darkred}{rgb}{0.5,0,0}
\definecolor{darkgreen}{rgb}{0,0.5,0}
\definecolor{darkblue}{rgb}{0,0,0.5}
\hypersetup{ colorlinks,
linkcolor=darkblue,
filecolor=darkgreen,
urlcolor=darkred,
citecolor=darkblue }
\newtheorem{theorem}{Theorem}[section]

\newtheorem{corollary}[theorem]{Corollary}

\newtheorem{proposition}[theorem]{Proposition}
\newtheorem{lemma}[theorem]{Lemma}
\newtheorem{lem}[theorem]{}
\theoremstyle{definition}
\newtheorem{definition}[theorem]{Definition}
\theoremstyle{remark}
\newtheorem{remark}[theorem]{Remark}
\newtheorem{example}[theorem]{Example}
\newcommand{\blem}{\begin{lem} \rm}
\newcommand{\elem}{\end{lem}}

\newcommand\M{\mathcal{M}}

\renewcommand\M{\mathcal{M}}

\renewcommand\S{\mathcal{S}}

\newcommand{\T}{\mathcal{T}}
\newcommand{\J}{\mathcal{J}}

\newcommand{\U}{\mathcal{U}}

\newcommand{\N}{\mathbb{N}}

\newcommand{\R}{\mathbb{R}}
\renewcommand{\H}{\mathbb{H}}

\newcommand{\C}{\mathbb{C}}
\newcommand{\cC}{\mathcal{C}}

\newcommand{\Z}{\mathbb{Z}}
\newcommand{\Q}{\mathbb{Q}}

\newcommand{\ddt}{\frac{d}{dt}}

\newcommand{\dds}{\frac{d}{ds}}

\newcommand{\on}{\operatorname}

\newcommand{\ainfty}{{$A_\infty$\ }}

\newcommand{\Def}{\on{Def}}
\newcommand{\dist}{\on{dist}}

\newcommand{\dual}{\vee}

\newcommand{\Fix}{\on{Fix}}

\renewcommand{\top}{{\on{top}}}

\newcommand{\Edge}{\on{Edge}}
\newcommand{\graph}{\on{graph}}

\newcommand{\Lag}{\on{Lag}}

\newcommand{\Ver}{\on{Vert}}
\newcommand{\Ve}{\on{Vert}}

\newcommand{\End}{\on{End}}

\newcommand{\Aut}{ \on{Aut} } 
 
\newcommand{\aut}{ \on{aut} }

\renewcommand{\ker}{ \on{ker}}

\newcommand\dirac{/\kern-1.2ex\partial} 
\newcommand\qu{/\kern-.7ex/} 
\newcommand\lqu{\backslash \kern-.7ex \backslash} 

\newcommand\dr{r_+ \kern-.7ex - \kern-.7ex r_-}
 
\def\cprime{$'$}

\newcommand{\labell}\label

\renewcommand{\d}{{\on{d}}}
\newcommand{\ol}{\overline}
\newcommand{\olp}{\ol{\partial}}

\newcommand\eps{\epsilon}

\newcommand{\ti}{\tilde}

\newcommand\cE{\mathcal{E}}

\newcommand\cG{\mathcal{G}}
\newcommand\cT{\mathcal{T}}
\newcommand\cF{\mathcal{F}}
\newcommand\cI{\mathcal{I}}

\newcommand\cJ{\mathcal{J}}

\newcommand\mE{\mathcal{E}}

\newcommand\curv{\on{curv}}

\newcommand\Map{\on{Map}}

\newcommand\Vect{\on{Vect}}
\newcommand\ul{\underline}

\newcommand\E{\mathcal{E}}
\newcommand\B{\mathcal{B}}

\newcommand\grad{\on{grad}}
\newcommand\reg{{\on{reg}}}

\newcommand\thin{{\on{thin}}}

\newcommand\bra[1]{ < \kern-.7ex {#1} \kern-.7ex >} 
\newcommand\bdefn{\begin{definition}}
\newcommand\edefn{\end{definition}}
\newcommand\bea{\begin{eqnarray*}}
\newcommand\eea{\end{eqnarray*}}
\newcommand\bcv{\left[ \begin{array}{r} }
\newcommand\ecv{\end{array} \right] }

\newcommand\bma{\left[ \begin{array}{l} }
\newcommand\ema{\end{array} \right]}
\newcommand\ben{\begin{enumerate}}
\newcommand\een{\end{enumerate}}
\newcommand\beq{\begin{equation}}
\newcommand\eeq{\end{equation}}
\newcommand\bex{\begin{example}}
\newcommand\bsj{\left\{ \begin{array}{rrr} }
\newcommand\esj{\end{array} \right\}}

\newcommand\univ{{\on{univ}}}

\newcommand\eex{\end{example}}
\newcommand\crit{{\on{crit}}}

\newcommand\sx{*\kern-.5ex_X}

\newcommand{\Tor}{\on{Tor}}

\newcommand{\thick}{{\on{thick}}}

\def\mathunderaccent#1{\let\theaccent#1\mathpalette\putaccentunder}
\def\putaccentunder#1#2{\oalign{$#1#2$\crcr\hidewidth \vbox
to.2ex{\hbox{$#1\theaccent{}$}\vss}\hidewidth}}

\begin{document}

\title{Floer trajectories and stabilizing divisors}

\author{Fran\c{c}ois Charest and Chris Woodward}

\address{Department of Mathematics, Barnard College - Columbia University,
MC 4433, 2990 Broadway, 
New York, NY 10027}
\email {charest@math.columbia.edu}

\address{Mathematics-Hill Center,
Rutgers University, 110 Frelinghuysen Road, Piscataway, NJ 08854-8019,
U.S.A.}  \email{ctw@math.rutgers.edu}

\thanks{
This work was partially supported
  by NSF grant DMS 1207194 and FRQNT grant B3 176181.  }

\begin{abstract} 
  We incorporate pearly Floer trajectories into the transversality
  scheme for pseudoholomorphic maps introduced by Cieliebak-Mohnke
  \cite{cm:trans}.  By choosing generic domain-dependent almost
  complex structures we obtain zero and one-dimensional moduli spaces
  with the structure of cell complexes with rational fundamental
  classes.  Integrating over these moduli spaces gives a definition of
  Floer cohomology over Novikov rings via stabilizing divisors for
  rational Lagrangians that are fixed point sets of anti-symplectic
  involutions satisfying certain Maslov index conditions as well as
  Hamiltonian Floer cohomology of compact rational symplectic
  manifolds.  A subsequent paper \cite{fuk} extends the techniques to
  regularization of moduli spaces needed for the definition of Fukaya
  algebras for more general Lagrangians.
\end{abstract} 


\maketitle

\parskip 0in \tableofcontents
\parskip .1in
\itemsep .1in

\section{Introduction} 

The Floer cohomology associated to a generic time-dependent
Hamiltonian on a compact symplectic manifold is a version of Morse
cohomology for the symplectic action functional on the space of paths
between two Lagrangians \cite{floer:lag}, \cite{floer:symp}.  The
cochains in this theory are formal combinations of Hamiltonian
trajectories connecting the Lagrangians while the coboundary operator
counts Hamiltonian-perturbed pseudoholomorphic strips.  When
well-defined, Floer cohomology is independent of the choice of
Hamiltonian and can be used to estimate the number of intersection
points as in the conjecture of Arnol\cprime d.

In order to construct the theory one must compactify the moduli spaces
of Floer trajectories by allowing pseudoholomorphic disks and spheres,
and unfortunately because of the multiple cover problem these moduli
spaces can not be made regular in general using a fixed almost
complex structure.  In algebraic geometry there is now an elegant
approach to this counting issue carried out in a sequence of papers by
Behrend-Manin \cite{bm:gw}, Behrend-Fantechi \cite{bf:in}, Kresch
\cite{kresch:can}, and Behrend \cite{be:gw}, based on suggestions of
Li-Tian \cite{litian:vir}.

In symplectic geometry several methods for solving these
transversality issues have been described.  An approach using
Kuranishi structures is developed in Fukaya-Ono \cite{fo:arnold} and a
related approach in Liu-Tian \cite{tian:floer}.  Technical details
have been further explained in Fukaya-Oh-Ohta-Ono \cite{fooo:tech} and
McDuff-Wehrheim \cite{mcduff:kur}.  A polyfold approach has been
developed by Hofer-Wysocki-Zehnder \cite{ho:sc} and is being applied
to Floer cohomology by Albers-Fish-Wehrheim \cite{afw}.  See also
Pardon \cite{pardon:alg} for an algebraic approach to constructing
virtual fundamental chains.  These approaches have in common that the
desired perturbed moduli space should exist {\em for abstract
  reasons}.

On the other hand, it is very useful to have perturbations {\em with
  geometric meaning} whenever possible.  For example, in the case of
toric varieties, Fukaya et al \cite{fooo:toric1} have pointed out that
the {\em general structure} of the invariants is not enough to see
that the Floer cohomology of Lagrangian torus orbits is well-defined,
and one has to choose a perturbation system adapted to the geometric
situation \cite{fooo:toric1}.  An approach to perturbing moduli spaces
of pseudoholomorphic curves based on {\em domain-dependent almost
  complex structures} was introduced by Cieliebak-Mohnke
\cite{cm:trans} for symplectic manifolds with rational symplectic
classes and further developed in Ionel-Parker \cite{ionel:vfc} and
Gerstenberger \cite{gersten:trans}.  It was extended by Wendl
\cite{wendl:hyp} in genus zero to allow insertions of Deligne-Mumford
classes.  Domain-dependent almost complex structures can be made
suitably generic only if one does not require invariance under
automorphisms of the domain.  Thus in order to obtain a perturbed
moduli space, one must kill the automorphisms.  This can be
accomplished by choosing a {\em stabilizing divisor}: a codimension
two almost complex submanifold meeting any pseudoholomorphic curve in
a sufficient number of points.  {\em Approximately} almost complex
submanifolds of codimension two exist by a result of Donaldson
\cite{don:symp}; the original almost complex structure can be
perturbed so that the Donaldson submanifolds are {\em exactly} almost
complex \cite{cm:trans}.  If the symplectic manifold admits a
compatible complex structure which makes it a smooth projective
variety, then Donaldson's results are not necessary.  Indeed, in this
case the existence of suitable divisors follows from results of
Bertini and Clemens \cite{cl:curves}.

We work under the following assumptions.  Let $X$ be a compact
connected symplectic manifold with symplectic form $\omega$ with
$[\omega] \in H^2(X,\Q)$ rational.  A {\em Lagrangian brane} is a
compact Lagrangian equipped with grading and relative spin structure,
see Section \ref{floertraj} below.  Results of Borthwick-Paul-Uribe
\cite[Theorem 3.12]{bo:le} in the K\"ahler case, and
Auroux-Gayet-Mohsen \cite{auroux:complement} more generally, imply the
existence of a stabilizing divisor in the complement of each
Lagrangian such that each holomorphic disk meets the divisor at least
once; the additional markings produced by the intersections allow us
to achieve transversality using domain-dependent almost complex
structures.  To define Floer cohomology we will also need further
conditions to rule out disk bubbling, which for lack of a better term
we call {\em admissibility}; this definition is local to this paper.

\begin{definition} \label{admissible} A Lagrangian $L \subset X$ is {\em admissible} 
if either
\begin{enumerate} 
\item the Lagrangian 
$$L = \on{\graph}(\psi) \subset X = Y^- \times Y$$
  is the graph of Hamiltonian diffeomorphism $\psi: Y \to Y$; here
  $Y^-$ denotes $Y$ with symplectic form reversed; or
\item the Lagrangian 
$$L = \{ x \in X \ | \ \iota(x) = x \} \subset X$$
is the fixed point set of an anti-symplectic involution
$$\iota:X \to X, \quad \iota^* \omega = - \omega,$$ 
has minimal Maslov number divisible by $4$, and is equipped with an
$\iota$-relative spin structure in the sense of \cite{fooo:anti}.
\end{enumerate} 
\end{definition} 
\noindent 
 Floer cohomology of more general Lagrangians is constructed,
following Fukaya-Oh-Ohta-Ono \cite{fooo} who used Kuranishi
structures, via stabilizing divisors in a follow-up paper \cite{fuk}
as a complex of bundles over the space of weak solutions to the
Maurer-Cartan equation associated to the Fukaya algebra.

The main results of this paper are the following existence and
invariance result for Floer cohomology groups. For $L_0,L_1 \subset X$
admissible Lagrangians, let $CF(L_0,L_1)$ denote the group of Floer
cochains with rational Novikov coefficients.  Denote by
$$\partial: CF(L_0,L_1) \to CF(L_0,L_1)$$
the coboundary operator defined by a rationally-weighted count of
Floer trajectories in the zero-dimensional component of the moduli
space.

\begin{theorem} \label{introthm} Let $L_0,L_1$ be admissible
  Lagrangian branes.  There exists a comeager subset of a set of
  domain-dependent almost complex structures such that the moduli
  space of perturbed Floer trajectories of expected dimension at most
  one $\ol{\M}_{\leq 1}(L_0,L_1)$ has the structure of a cell complex
  with fundamental class in relative homology.  The boundary of the
  locus of expected dimension one is the union of broken Floer
  trajectories consisting of pairs of Floer trajectories of expected
  dimension zero:
$$  \partial [\ol{\M}_{1}(L_0,L_1)] = [{\M}_{0}(L_0,L_1)
\times_{\cI(L_0,L_1)} {\M}_{0}(L_0,L_1)] .$$
The Floer coboundary operator $\partial$ is well-defined and satisfies
$\partial^2 = 0$.  The resulting Floer cohomology
$$HF(L_0,L_1) = \frac{\on{ker}(\partial)}{\on{im}(\partial)} $$ 
is independent of all choices and invariant under Hamiltonian
perturbation of either Lagrangian.  In the case $L_0 = L_1 = \Delta$
is equal to the diagonal Lagrangian $\Delta$ in a product
$X = Y^- \times Y$, the Floer cohomology $HF(\Delta,\Delta)$
isomorphic to the singular cohomology of $Y$ with coefficients in the
Novikov field.
\end{theorem} 

Versions of this theorem, which includes a weak version of the
Arnol\cprime d conjecture, appear in \cite{fooo:anti} and in Liu-Tian
\cite{tian:floer}, and another version announced as \cite{afw}.
However, abstract results such as the weak Arnol\cprime d conjecture
are not the motivation for the approach presented here; rather we have
various applications in mind in which the geometric meaning of the
trajectories plays a role.  Note that the proof of the weak
Arnol\cprime d conjecture given here does not require the use of
orbifolds (whose notion of morphism is very involved) nor virtual
fundamental classes of any type, nor the transversality results for
Floer trajectories in Floer-Hofer-Salamon \cite{fhs:tr}.

We thank Mohammed Abouzaid, Eleny Ionel, Tom Parker, Nick Sheridan,
Sushmita Venugopalan, and Chris Wendl for helpful discussions.

\section{Floer trajectories} 
\label{bm} 

This section constructs a compactification of the moduli space of
holomorphic strips with interior markings which allows the formation
of {\em bubbles} on the boundary. The domains of these maps are nodal
disks with two distinguished points.  There are morphisms between
moduli spaces of these disks of different combinatorial type, first
studied by Knudsen \cite{kn:proj2} and Behrend-Manin \cite{bm:gw}.
These morphisms will play a key role in the construction of the
perturbations.

\subsection{Stable strips} 

We recall the definition of stable marked curves introduced by
Mumford, Grothendieck and Knudsen \cite{kn:proj2}. By a {\em nodal
  curve} $C$ we mean a compact complex curve with only nodal
singularities denoted $w_1,\ldots, w_l \in C$ for some integer $l \ge
0$.  A {\em non-singular point} of $C$ is a non-nodal point.  For an
integer $n \ge 0$, an {\em $n$-marking} of a nodal curve $C$ is a
collection $ \ul{z} \in C^n$ of distinct, non-singular points. A {\em
  special point} is a marking or node.  An {\em isomorphism of marked
  curves} $(C_0,\ul{z}_0)$ to $(C_1,\ul{z}_1)$ is an isomorphism $C_0
\to C_1$ mapping $\ul{z}_0$ to $\ul{z}_1$.  Let $\Aut(C) =
\Aut(C,\ul{z})$ be the group of automorphisms of $(C,\ul{z})$.  A
genus zero marked curve $(C,\ul{z})$ is {\em stable} if the group
$\Aut(C)$ is trivial, or equivalently, any irreducible component of
$C$ has at least three special points.  The {\em combinatorial type}
of a nodal marked curve $C$ is the graph $\Gamma$ whose vertices
$\Ver(\Gamma)$ correspond to components of $C$ and edges
$\Edge(\Gamma)$ consisting of {\em finite edges}
$\Edge_{<\infty}(\Gamma)$ corresponding to nodes $w(e) \in C, e \in
\Edge_{< \infty}(\Gamma)$ and {\em semi-infinite edges}
$\Edge_\infty(\Gamma)$ corresponding to markings; the set of
semi-infinite edges with a labelling $\Edge_\infty(\Gamma) \to \{ 1,
\ldots, n \}$ so that the corresponding markings are $z_1,\ldots, z_n
\in C$.  A connected curve $C$ has genus zero if and only if the graph
$\Gamma$ is a tree and each irreducible component of $C$ has genus
zero, that is, is isomorphic to the projective line.  Knudsen
\cite{kn:proj2} shows that the moduli space of stable genus zero
$n$-marked connected curves has the structure of a smooth projective
variety, and in particular, a smooth compact manifold.  Let
$\ol{\cC}_{n}$ denote the moduli space of isomorphism classes of
connected stable $n$-marked genus zero curves, with topology induced
by Knudsen's theorem \cite{kn:proj2}.  For a combinatorial type
$\Gamma$ let $\cC_{\Gamma}$ the space of isomorphism classes of stable
marked curves of type $\Gamma$.  If $\Gamma$ is connected with $n$
semi-infinite edges we denote by $\ol{\cC}_\Gamma$ the closure of
$\cC_\Gamma$ in $\ol{\cC}_n$.  We also allow disconnected curves
$\Gamma$ whose components are genus zero, in which case we order the
markings on each component and the combinatorial type $\Gamma$ is a
{\em forest} (disjoint union of trees.)  If $\Gamma = \sqcup_i
\Gamma_i$ is disconnected then $\cC_\Gamma = \prod_i \cC_{\Gamma_i} $
is the product of the moduli spaces $\cC_{\Gamma_i}$ for the component
trees.  Over $ \ol{\cC}_{n}$ there is a {\em universal stable marked
  curve} whose fiber over an isomorphism class of stable marked curve
$[C,\ul{z}]$ is isomorphic to $C$.

The smooth structures on the moduli spaces can be described by
deformation theory.  Recall that a {\em deformation} of a nodal marked
curve $(C,\ul{z})$ is a germ of a family $C_B \to B$ of nodal curves
over a pointed scheme, say, $(B,b)$ together with sections $\ul{z}_B:
B \to C_B^n$ corresponding to the markings and an isomorphism of the
given marked curve $(C,\ul{z})$ with the fiber over $b \in B$.  We
omit the markings to simplify the notation.  Any stable curve has a
{\em universal deformation $C_B \to B$}, unique up to isomorphism,
with the property that any other deformation is obtained via pullback
by a unique map to $B$.  The {\em space of infinitesimal deformations}
$\Def(C)$ is the tangent space of $B$ at $b$.  Similarly the {\em
  tangential deformation space} of a nodal marked curve $C$ of type
$\Gamma$ is the space $\Def_\Gamma(C)$ of infinitesimal deformations
of the complex structure fixing the combinatorial type.  The
deformation spaces sit in an exact sequence
$$ 0 \to \Def_\Gamma(C) \to \Def(C) \to \cG_\Gamma(C) \to 0 $$
where
\begin{equation} \label{gluing0} \cG_\Gamma(C) = \bigoplus_{e \in
    \Edge_{< \infty}(\Gamma)} T_{w(e)^+} ^\dual C \otimes
  T_{w(e)^-}^\dual C \end{equation}
is the {\em normal deformation space}, see for example
Arbarello-Cornalba-Griffiths \cite[Chapter 11]{ar:gac2}, and
$T_{w(e)^\pm} C$ are the tangent spaces on either side of the
corresponding nodes $w(e) \in C$.  Elements of the normal deformation
space are known as {\em gluing parameters} in symplectic geometry.
Given a universal deformation of stable curves of fixed type $\Gamma$,
the {\em gluing construction} produces a universal deformation by
removing small disks around each node in a local coordinate $z$ and
gluing the components together using maps $z \mapsto \delta_e/z$ where
$\delta_e$ is the parameter corresponding to the $e$-th node.  In
genus zero there are several canonical schemes for choosing such local
coordinates, for example, by using three special points on a component
to fix an isomorphism with the projective line.

The moduli space of {\em stable marked disks} is a smooth manifold
with corners as described in Fukaya-Oh-Ohta-Ono \cite{fooo}.  For
integers $l,n \ge 0$, an $(l,n)$-marked disk is a stable $l+2n$-marked
stable sphere $\hat{C}$ equipped with an anti-holomorphic involution
$\iota_{\hat{C}}: \hat{C} \to \hat{C}$ such that the first $l$
markings are those on the fixed locus of the involution $\hat{C}^*$,
and the quotient $C = \hat{C}/\iota_{\hat{C}}$ of the curve by the
involution is a union of {\em disk components} (arising from
components preserved by the involution and with non-empty fixed locus)
and {\em sphere components} (arising from components interchanged by
the involution).  The markings fixed by the involution are {\em
  boundary markings} of $C$ while the remaining conjugated pairs of
markings are {\em interior markings} of $C$.  The deformation theory
of stable disks is similar to that of stable spheres, except that the
gluing parameters associated to the boundary markings will be assumed
to be real non-negative:
\begin{multline} \label{gluing}  \cG_\Gamma(C) := 
\bigoplus_{e \in \Edge_d(\Gamma)} (T_{w(e)^+}^\dual (\partial C)
\otimes T_{w(e)^-}^\dual (\partial C) )_{\ge 0} \\ \oplus \bigoplus_{e
  \in \Edge_s(\Gamma)} T_{w(e)^+} ^\dual C \otimes T_{w(e)^-}^\dual C
\end{multline}
where $\Edge_d(\Gamma)$ (resp. $\Edge_s(\Gamma)$) is the set of edges
corresponding to {\em boundary nodes} connecting disk components
(resp. {\em interior nodes} connecting sphere components to sphere or
disk components), and
$(T_{w(e)^+}^\dual (\partial C) \otimes T_{w(e)^-}^\dual (\partial C)
)_{\ge 0} \subset T_{w(e)^+}^\dual (\partial C) \otimes
T_{w(e)^-}^\dual (\partial C) \cong \R $
is the non-negative part; the last isomorphism is induced from the
natural orientations on the components of the boundary $\partial C$.

We introduce notation for moduli spaces of marked strips as follows.
A {\em marked strip} is a marked disk with two boundary markings.  Let
$C$ be a connected $(2,n)$-marked nodal disk with markings $\ul{z}$.
We write $\ul{z} = (z_-, z_+, z_1,\ldots, z_n)$ where $z_\pm$ are the
boundary markings and $z_1,\ldots,z_n$ are the interior markings.  We
call $z_-$ (resp. $z_+$) the {\em incoming} (resp. {\em outgoing})
marking. Let $C_1,\ldots, C_m$ denote the ordered {\em strip
  components} of $C$ connecting $z_-$ to $z_+$; the remaining
components are either {\em disk components}, if they have boundary, or
{\em sphere components}, otherwise.
Let $\ti{w}_i := C_i \cap C_{i+1}$ denote the {\em intermediate node}
connecting $C_i$ to $C_{i+1}$ for $i = 1,\ldots, m-1$.  Let
$\ti{w}_0 = z_-$ and $\ti{w}_m = z_+$ denote the incoming and outgoing
markings.  Let $C^\times := C - \{ \ti{w}_0,\ldots, \ti{w}_m \}$
denote the curve obtained by removing the nodes connecting strip
components and the incoming and outgoing markings.  Each strip
component may be equipped with coordinates
$$ \phi_i : C_i^\times := C_i - \{ \ti{w}_i,\ti{w}_{i-1} \} \to \R \times [0,1]
, \quad i =1,\ldots, m $$
satisfying the conditions that if $j$ is the standard complex
structure on $\R \times [0,1]$ and $j_i$ is the complex structure on
$C_i$ then 
$$\phi_i^* j = j_i, \quad  \lim_{z \to \ti{w}_i} \pi_1
\circ \phi_i(z) = -\infty, \quad \displaystyle \lim_{z \to \ti{w}_{i+1}}
\pi_1 \circ \phi_i(z) = \infty$$ 
where $\pi_i$ denotes the projection on the $i^{th}$ factor.  We
denote by 
\begin{equation} \label{timecoord}
f: C^\times \to [0,1], \quad z \mapsto \pi_2 \circ
\phi_i(z) \end{equation} 
the continuous map induced by the {\em time} coordinate on the
strip components. The time coordinate is extended to nodal marked
strips by requiring constancy on every connected component of
$C^\times - \bigcup_i C^\times_i$.  The boundary of any marked strip
$C$ is partitioned as follows.  For $b \in \{ 0, 1 \}$ denote
\begin{equation} \label{Cb}
(\partial C)_b := f^{-1}(b) \cap \partial C \end{equation}
so that $\partial C^\times = (\partial C)_0 \cup (\partial C)_1$.
That is, $(\partial C)_b$ is the part of the boundary from $z_-$ to
$z_+$, for $b = 0$, and from $z_+$ to $z_-$ for $b = 1$.  An example
of a stable strip is shown in Figure \ref{stabstrip}. 

\begin{figure}
\includegraphics[width=3in]{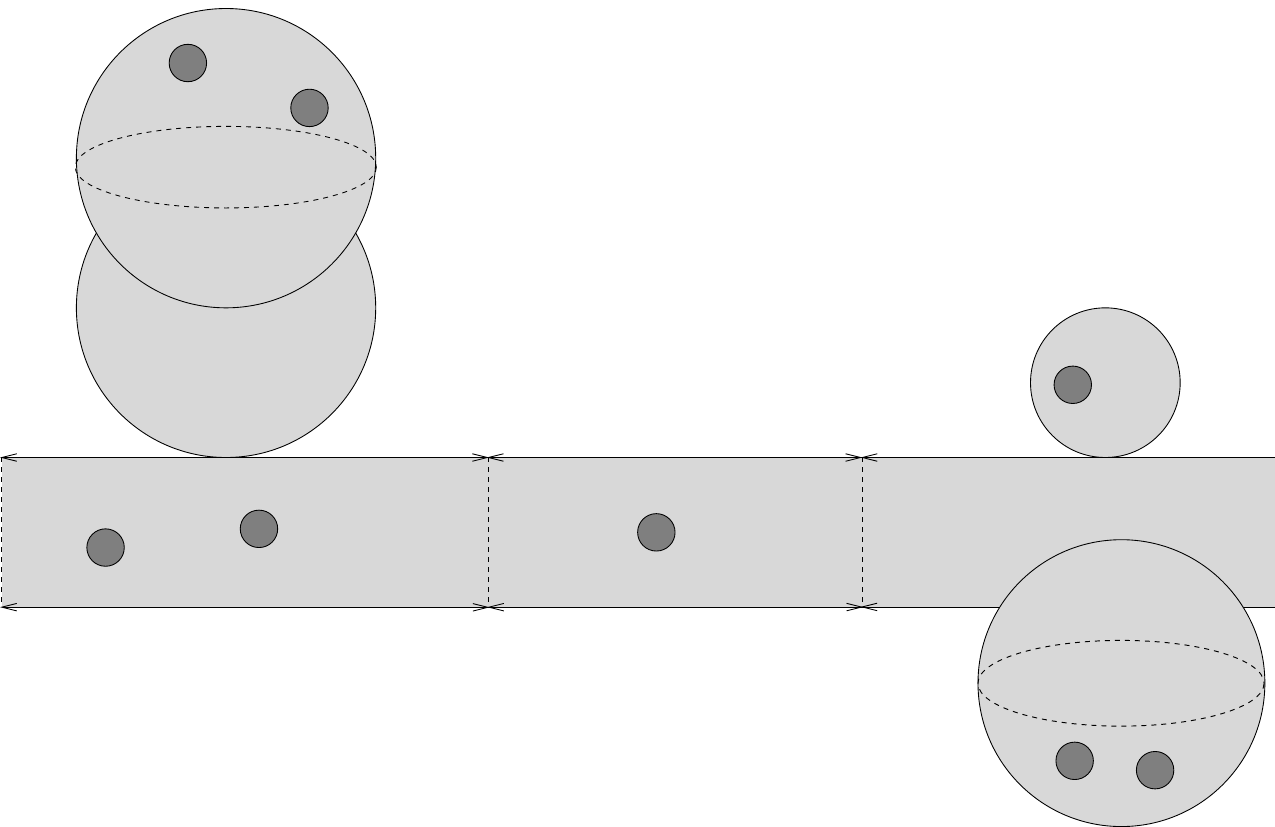}
\caption{A stable marked strip} 
\label{stabstrip} 
\end{figure} 

A variation on the definition of moduli spaces of stable disks or
strips involves {\em metric trees}.  A {\em treed strip} of type
$\Gamma$ is a marked strip $(C,\ul{z})$ with a {\em metric}
$$\ell:\Edge_{<\infty,d}(\Gamma) \to [0,\infty]$$ 
assigning lengths to the finite edges of the subgraph of $\Gamma$
corresponding to disk components.  The {\em combinatorial type} is
defined as before, except that the subset of edges with infinite, zero
or $]0,\infty[$ lengths is recorded as part of the data. Combinatorial
types of strips naturally define combinatorial types of treed strips
by adding zero metrics on their edges corresponding to boundary nodes.
An {\em isomorphism} of treed marked strips is an isomorphism of
marked strips having the same metric.  A {\em stable} treed strip is
one that has a stable underlying strip.  Constructions of this type
appear in, for example, Oh \cite{oh:fl1}, Cornea-Lalonde
\cite{cl:clusters}, Biran-Cornea \cite{bc:ql}, and Seidel
\cite{seidel:genustwo} and stable treed strips can be seen as special
cases of the domain spaces used in \cite{charest:clust}.  There is a
natural notion of convergence of treed stable strips, in which
degeneration to a nodal disk assigns length zero to the node that
appears.  Let $\ol{\M}_n$ denote the moduli space of isomorphism
classes of connected stable treed strips with $n$ interior markings in
addition to the incoming and outgoing markings. For $\Gamma$ a
connected type we denote by ${\M}_{\Gamma} \subset \ol{\M}_n$ the
moduli space of stable strips of combinatorial type $\Gamma$ and
$\ol{\M}_\Gamma$ its closure.  Each $\ol{\M}_\Gamma$ is naturally a
manifold with corners, with local charts obtained by a standard gluing
construction.  Generally $\ol{\M}_n$ is the union of several
top-dimensional strata.  In Figure \ref{onetwo} the locus in
$\ol{\M}_{1}$ where the first marking has time coordinate $1/2$ is
depicted (the stratum where the two interior markings come together to
form a sphere bubble is not shown.)
\begin{figure}[h!]
\includegraphics[width=3in]{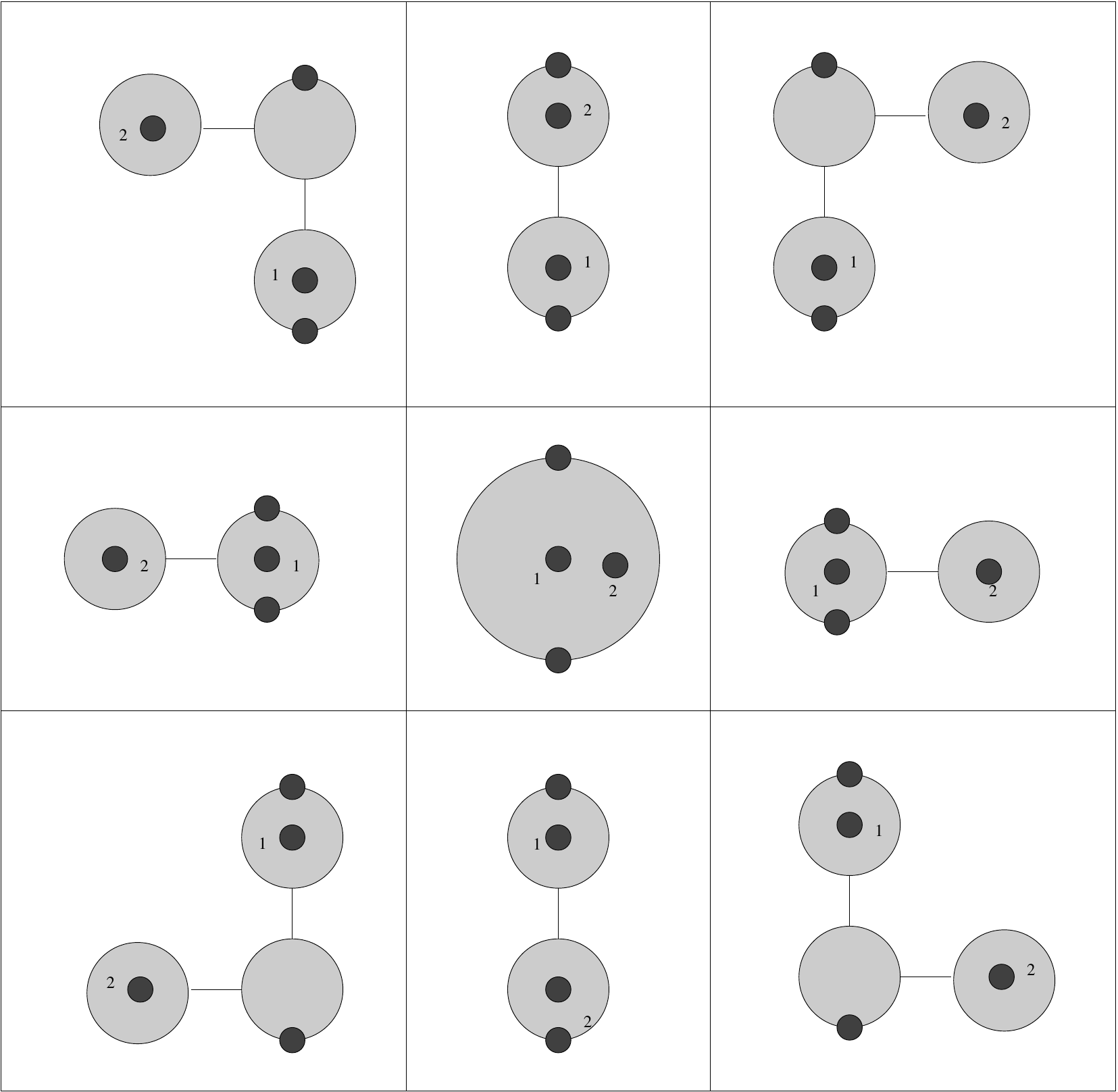}
\caption{Stable treed disks}
\label{onetwo} 
\end{figure}
\noindent For $\Gamma$ disconnected, $\ol{\M}_\Gamma$ is the product $\prod_i
\ol{\M}_{\Gamma_i}$ of moduli spaces for the connected components
$\Gamma_i$ of $\Gamma$.

Orientations on the strata in the moduli space of treed disks may be
constructed as follows (see Charest \cite{charest:clust} for a more
general procedure).  Each stratum is a product of moduli spaces for
the disk components and the intervals corresponding to the length
parameters.  For the case of a single disk, $m+2$ markings on the
boundary and $n$ points in the interior, the moduli space may be
identified with the subset of distinct tuples in
\begin{equation} \label{orientmod} \{ (w_1,\ldots,w_m,z_1,\ldots,z_n) \in \R \backslash (\R \times \{
0,1 \})^m \times ( \R \times (0,1))^n \} .\end{equation}
As a result, it inherits an orientation from the orientation on
$\R \times [0,1]$ and the orientation on $\R$.  In particular, note
that the configuration with $m = 1,n = 0$ and the points $w_1$ lying
in $\R \times \{ 0 \}$ resp. $\R \times \{ 1 \}$ has orientation given
by $-1$ resp. $+1$ times the orientation of a point.  If the disk has
a single special point on the boundary then its interior may be
identified with the positive half-space $\H \subset \C$.  The moduli
space $\M_\Gamma$ of disks with a single marking on the boundary and
$n$ points in the interior is the subset of distinct tuples in
$$ \M_\Gamma \cong 
\Aut(\H) \backslash \{ (z_1,\ldots,z_n) \in \H^n \}
 $$
where $\Aut(\H)$ is the Lie group of automorphisms of the upper
half-plane $\H$.  Now $\Aut(\H)$ is generated by translations and
dilations, and this identifies $\aut(\H)$ with $\R^2$ and so gives
$\Aut(\H)$ an orientation.  The moduli space is oriented by the
orientation on $\H^n$ and $\Aut(\H)$.  In general, any stratum
$\M_\Gamma$ of $\ol{\M}_n$ is oriented by taking the product of the
orientations for the disk components and strip components and the
product of the intervals corresponding to lengths of the edges.  The
resulting orientation depends on the ordering of components and edges.
However, the constructions below involve only the case of one disk
component with only one boundary special point (in which case the
moduli space has even dimension and the ordering is irrelevant) and at
most two strip components (in which case the one containing the
incoming marking is ordered first, and at most one interval
corresponding to an edge with non-zero length).  In particular, if
$\Gamma$ is a type with a single disk bubble attached to
$\R \times \{ 0 \}$ resp.  $\R \times \{1 \}$ by an edge of zero
length and $\Gamma'$ is the type obtained by collapsing an edge then
the map $\M_\Gamma \to \ol{\M}_{\Gamma'}$ is orientation preserving
resp.  reversing.

Our perturbations will be maps defined on certain {\em universal
  curves} over the moduli spaces above.  The moduli space of treed
strips of each fixed type admits a universal strip
$$\ol{\U}_\Gamma \to \ol{\M}_\Gamma $$ 
whose fiber over an element $[C,\ul{z},\ell] \in \ol{\M}_\Gamma$ is
the underlying stable marked disk $S= (C,\ul{z})$ without the metric.
The universal strip is closely related to the moduli space
$\ol{\M}_{\Gamma'}$ of strips of the type $\Gamma'$ obtained from
$\Gamma$ by adding an additional interior marking; the two spaces are
homeomorphic away from the boundary nodes (where the latter has real
blowups). In particular, $\ol{\U}_\Gamma$ is a manifold with corners
away from the boundary nodes.  Later we will use local trivializations
$\ol{\U}_\Gamma$ of the universal strip on each stratum $\M_\Gamma$
giving rise to families of complex structures.  For a combinatorial
type $\Gamma$ and a treed disk $C$ of type $\Gamma$ let
\begin{equation} \label{coll} {{\U}}_{\Gamma}^{i} \to {{\M}}_{\Gamma}^i 
\times S , \quad i = 1,\ldots, l \end{equation}
be a collection of local trivializations of the universal strip
identifying each nearby fiber with $S$ in a way that the markings are
constant.  Let $\J(S)$ denote the space of complex structures on the
smooth curve underlying $C$, or equivalently, the space of complex
structures on its {\em normalization} of $C$ obtained by blowing up
each node (so each node gets replaced with a pair of points).  The
complex structures on the fibers induce a map
\begin{equation} \label{localtriv} {\M}_{\Gamma}^i \to \J(S), \  m \mapsto j(m).
 \end{equation}
Since the universal curve is locally holomorphically trivial in a
neighborhood of the nodes and markings (see for example \cite[Chapter
  11]{ar:gac2} for the case of curves without boundary) we may assume
that $j(m)$ is independent of $m$ on neighborhoods of the nodes and
markings of $S$.

The combinatorial structure of the moduli spaces of stable treed
strips involves identifications between moduli spaces of different
combinatorial types introduced in Knudsen \cite{kn:proj2} and
Behrend-Manin \cite{bm:gw} in the case of stable curves.
Families of perturbations will later
be chosen to be coherent with respect to those identifications.  In
Behrend-Manin \cite{bm:gw}, these morphisms were associated to
morphisms of graphs called {\em extended isogenies}.  Here we call
them Behrend-Manin morphisms.

\begin{definition} \label{bmgraphs} {\rm (Behrend-Manin morphisms of
    graphs)} A {\em morphism} of graphs $\Upsilon: \Gamma \to \Gamma'$
  is a surjective map of the set of vertices
  $\Ve(\Gamma) \to \Ve(\Gamma')$ obtained by combining the following
  {\em elementary morphisms}:
\begin{enumerate} 
\item {\rm (Cutting edges)} $\Upsilon$ {\em cuts an edge}
  $e \in \Edge_{< \infty,d}(\Gamma)$ with infinite length or an edge
  $e \in \Edge_{< \infty,s}(\Gamma)$ if $\Upsilon$ is a bijection but the edge sets are related by
$$\Edge(\Gamma') \cong \Edge(\Gamma) - \{ e \} + \{ e_+,e_- \} $$
where $e_\pm$ are semi-infinite edges attached to the vertices
contained in $e$.  Since our curves have genus zero, $\Gamma'$ is
disconnected with pieces $\Gamma_-,\Gamma_+$.  If the edge corresponds
to a node connecting disk components then $\Gamma_-,\Gamma_+$ are
types of stable disks, while if the edge corresponds to a node
connecting a disk or sphere to a sphere component then one type, say
$\Gamma_-$ is the type of a stable disk while $\Gamma_+$ is the type
of a stable sphere. The ordering on $\Edge_{ \infty,s}(\Gamma)$
induces one on $\Edge_{ \infty,s}(\Gamma_\pm)$ by using on $e_\pm$ the
lowest value label on $\Edge_{\infty,s}(\Gamma_\mp)$.

\item {\rm (Collapsing edges)} $\Upsilon$ {\em collapses an edge} if
  the map on vertices $\Upsilon: \Ve(\Gamma) \to \Ve(\Gamma')$ is a
  bijection except for a single vertex $v' \in \Ve(\Gamma')$ which has
  two pre-images connected by an edge in $\Edge(\Gamma)$ of length
  zero, and
$$\Edge(\Gamma') \cong \Edge(\Gamma) - \{ e \} .$$
\item {\rm (Making an edge finite or non-zero)} $\Upsilon$ {\em makes
  an edge finite or non-zero} if $\Gamma$ is the same graph as
  $\Gamma'$ and the lengths of the edges $\ell(e), e \in
  \Edge(\Gamma')$ are the same except for a single edge for which
  $\ell(e) = \infty $ resp. $0$ and the length $\ell'(e)$ in $\Gamma'$ is
  in $(0,\infty)$.
\item {\rm (Forgetting tails)} $\Upsilon: \Gamma \to \Gamma'$ {\em
  forgets a tail} (semi-infinite edge) and collapses edges that become
  unstable. The ordering on $\Edge_{ \infty,s}(\Gamma)$ then naturally
  defines one on $\Edge_{ \infty,s}(\Gamma')$.
\end{enumerate} 

\end{definition} 

Each of the above operations on graphs corresponds to a map of
moduli spaces of stable marked disks.

\begin{definition} {\rm (Behrend-Manin maps of moduli spaces)} 
\begin{enumerate} 
\item {\rm (Cutting edges)} Suppose that $\Gamma'$ is obtained from
  $\Gamma$ by cutting an edge.  A diffeomorphism
  $\ol{\M}_{\Gamma'} \to \ol{\M}_\Gamma$ obtained by identifying the
  two markings corresponding to the cut edge and choosing the ordering
  of the markings to correspond to the non-identified markings of the
  original curve.
\item {\rm (Collapsing edges)} Suppose that $\Gamma'$ is obtained from
  $\Gamma$ by collapsing an edge. There is an embedding
  $\ol{\M}_{\Gamma} \to \ol{\M}_{\Gamma'}$ with normal bundle having
  fiber at $[C]$ isomorphic to the space
  $\cG_{\Gamma'}(C)/\cG_\Gamma(C)$, see \eqref{gluing}. The image of
  such an embedding is a 1-codimensional corner of
  $\ol{\M}_{\Gamma'}$.  
\item {\rm (Making an edge finite resp. non-zero )} If $\Gamma'$ is
  obtained from $\Gamma$ by making an edge finite resp. non-zero then
  $\ol{\M}_{\Gamma}$ also embeds in $\ol{\M}_{\Gamma'}$ as the
  1-codimensional corner where $e$ reaches infinite resp. zero length,
  with trivial normal bundle.
\item {\rm (Forgetting tails)} Suppose that $\Gamma'$ is obtained from
  $\Gamma$ by forgetting the $i$-th tail.  Forgetting the $i$-th
  marking and collapsing the unstable components and sum distances for
  the glued edges defines a map $\ol{\M}_\Gamma \to
  \ol{\M}_{\Gamma'}$.
\end{enumerate} 
\end{definition} 

Most of these maps were already considered by Knudsen \cite{kn:proj2}
and might also be called Knudsen maps.  Each of the maps involved in
the operations {\rm (Collapsing edges), (Making edges finite or
  non-zero), (Forgetting tails), (Cutting edges)} extends to a smooth
map of universal treed strips.  
In the
case that $\Gamma$ is disconnected, say the disjoint union of
$\Gamma_1$ and $\Gamma_2$, then
$\ol{\M}_\Gamma \cong \ol{\M}_{\Gamma_1} \times \ol{\M}_{\Gamma_2} $.
In this case the universal treed strip $\ol{\U}_{\Gamma}$ is the
disjoint union of the pullbacks of the universal treed strip
$\ol{\U}_{\Gamma_1}$ and $\ol{\U}_{\Gamma_2}$.

We recall the definition of the morphisms in (Forgetting a tail).  Let
$(C,z_1,\ldots,z_n,\ell)$ be a stable treed marked strip with $n$
interior markings $z_1,\ldots, z_n$.  Given an integer
$i \in \{ 1,\ldots, n \}$, we obtain an $n-1$-marked treed strip
$(C,z_1,\ldots, \hat{z_i}, z_n,\ell)$ by forgetting the $i$-th
marking.  The unstable components can be collapsed as follows:
\begin{enumerate} 
\item If a sphere component $C_k \subset C$ becomes unstable after
  forgetting the $i$-th marking, having only two remaining special
  points $w', w'' \in C_k$, collapse the component, identifying the
  remaining special points $w' \sim w''$;
  
\item If a disk or strip component $C_k \subset C$ becomes unstable
  after forgetting the $i$-th marking, starting from the component
  containing the $i$-th marking, iteratively collapse the unstable
  disk components $C_{i_1},\ldots, C_{i_l}$, identifying the boundary
  special points if there are more than one or forgetting it whenever
  there is only one. At every step, if the collapsed disk $C_{i_l}$
  has only one node $w \in C_{i_l}$ with a metric, that metric should
  be forgotten (the node itself is forgotten), while if it has two
  nodes $w',w''$ with a metric, those metrics should be summed and the
  two node identified to obtain a node $w \sim w' \sim w''$ with
  length $\ell(w) = \ell(w') + \ell(w'')$.
\end{enumerate} 

The universal strip is equipped with the following maps.  On the
universal strip, the time coordinates \eqref{timecoord} on the strip
components extend to a map
\begin{equation} \label{allbutone} f: \ol{\U}_\Gamma \to [0,1
  ] \end{equation}
On the subset with time coordinate equal to zero or one, we have
additional maps measuring the distance to the strip components
given by summing the lengths of the connecting edges:
$$ \ell_b: f^{-1}(b) \to [0,\infty], 
\quad z \mapsto \sum_{e \in \Edge(z)} \ell(e), \quad b \in \{ 0, 1
\} $$
where $\Edge(z)$ is the set of edges corresponding to nodes between
$z$ and the strip components.  Thus any point on the universal
strip $z \in \ol{\U}_\Gamma$ which lies on a disk component has
$f(z) \in \{ 0, 1 \}$ and so $\ell_b(z)$ is well-defined for either
$b = 0 $ or $b = 1$.

\subsection{Floer trajectories} 
\label{floertraj}

We now turn to Floer theory. First we describe the space of Floer
cochains.  Recall that $(X,\omega)$ is a compact symplectic manifold.
Let Lagrangians $L_0,L_1 \subset X$ be admissible Lagrangian branes.
Let 
$H \in C^\infty([0,1] \times X)$ be a time-dependent Hamiltonian and 
let 
$$\widehat{H}\in \Map([0,1], \Vect(X)), \quad \iota(\hat{H}) \omega = \d
H $$
denote its Hamiltonian vector field with time $t$ flow
$\phi_t: X \to X$. Denote by
$$ \cI(L_0,L_1) := \left\{ x: [0,1] \to X \ | \ \ddt x = \widehat{H}
\circ x, \quad x(b) \in L_b, b \in \{ 0 ,1 \} \right\} \cong \phi_1(L_0) \cap
L_1 $$ 
the set of perturbed intersection points, assumed to be transverse.

We are particularly interested in the special case of the diagonal.
That is, suppose that $Y$ is a compact symplectic manifold, $Y^-$ is
the symplectic manifold obtained by reversing the symplectic form and
$X = Y^- \times Y$.  Let 
$$ \Delta = \{ (y,y) \ | \ y \in Y \} \subset X $$
denote the diagonal Lagrangian.  In the case that $L_0=L_1 = \Delta$
there is a bijection between perturbed intersection points and the
space of one-periodic orbits,
$$ \cI(L_0,L_1) = \left\{ x: S^1 \to X \ | \ \ddt x = \widehat{H}
\circ x \right\} \cong \{ x \in X \ | \ \phi_1(x) = x \} .$$

A degree map on the set of intersection points is induced from
gradings on the Lagrangians.  Suppose that $X$ is equipped with an
$N$-fold Maslov cover $\Lag^N(X) \to \Lag(X)$ for some positive
integer $N$; recall from Seidel \cite[Lemma 2.6]{se:gr} that
$(X,\omega)$ admits an $N$-fold Maslov covering if and only if $N$ divides twice
the minimal Chern number.  A {\em grading} of a Lagrangian $L$ is a
lift of the canonical section $L \to \Lag(X)$ to $\Lag^N(X)$.  Given
gradings on $L_0,L_1$, the intersection set $\cI(L_0,L_1)$ is equipped
with a {\em degree map}
$$ \cI(L_0,L_1) \to \Z_N, \quad x \mapsto |x| .$$
Denote by $\cI(L_0,L_1)_k = \{ x \in \cI(L_0,L_1) \ |\  | x| = k \}$ 
the subset of degree $k$ so that
$$ \cI(L_0,L_1) = \bigcup_{k \in \Z_N} \cI(L_0,L_1)_k .$$
The Floer cochains form a cyclically-graded group generated by the
time-one periodic trajectories.  Let $\Lambda$ be the universal
Novikov field in a formal parameter $q$,
$$ \Lambda = \left\{ \sum_{\rho \in \R} a(\rho) q^\rho , \quad \forall
\eps, \on{supp}(a) \cap (-\infty,\eps) < \infty, \quad a(\R) \subset
\Q \right\} .$$
Products in $\Lambda$ are defined by
$q^{\rho_1} q^{\rho_2} = q^{\rho_1 + \rho_2}$ for
$\rho_1,\rho_2 \in \R$, and extended to linear combinations.  Let
$CF(L_0,L_1)$ denote the free $\Lambda$-module generated by $\cI(L_0,L_1)$,
$$ CF(L_0,L_1) = \bigoplus_{k \in \Z_N} CF^k(L_0,L_1), \quad 
CF^k(L_0,L_1) = \bigoplus_{x \in \cI(L_0,L_1)_k} \Lambda \bra{x} .$$

The coboundary operator in Floer cohomology is defined by counting
Floer trajectories. To describe Floer's equation we introduce
notations for almost complex structures and associated decompositions.
We have the following notation for almost complex structures on
surfaces.  Let $C$ be an oriented surface, possibly with non-empty
boundary $\partial C$.  Denote by $\J(C)$ the space of complex
structures $j: C \to C$ compatible with the orientation.  Let
$E \to C$ be a complex vector bundle and let $\Omega^1(C,E)$ denote
the space of one-forms with values in $E$, that is, sections of
$T^\dual C \otimes_\R E$.  Given $\alpha \in \Omega^1(C,E)$ and a
complex structure $j \in \J(C)$, denote by
$\alpha^{0,1} \in \Omega^{0,1}(C,E)$ its $0,1$-part with respect to
the decomposition
$$ T^\dual C \otimes_\R E = (T^\dual C 
  \otimes_\R E)^{1,0} \oplus (T^\dual C \otimes_\R E)^{0,1} .$$
  We introduce the following terminology for almost complex
  structures.  An almost complex structure $J : TX \to TX$ is
  $\omega$-{\em tamed} if $\omega( \cdot, J \cdot)$ is positive, and
  $\omega$-{\em compatible} if $\omega( \cdot, J \cdot)$ is a
  Riemannian metric.  Denote by $\J_{\tau}(X,\omega)$ the space of
  $\omega$-tamed almost complex structures and by $\J(X,\omega)$ the
  space of $\omega$-compatible almost complex structures.  In the
  following paragraphs, let $J \in \J_\tau(X,\omega)$.

  Next we introduce notations for Hamiltonian perturbations and
  associated perturbed Cauchy-Riemann operators.  Let
  $K \in \Omega^1(C,\partial C; C^\infty(X))$ be a one-form with
  values in smooth functions vanishing on the tangent space to the
  boundary, that is, a smooth map $TC \times X \to \R$ linear on each
  fiber of $TC$, equal to zero on $TC | \partial C$.  Denote by
  $\widehat{K} \in \Omega^1(C,\partial C; \Vect(X))$ the corresponding
  one form with values in Hamiltonian vector fields.  The {\em
    curvature} of the perturbation is
\begin{equation} \label{RK}
 R_K = \d K + \{ K, K \}/2 \in \Omega^2(C, C^\infty(X)) \end{equation} 
where $ \{ K, K \} \in \Omega^2(C, C^\infty(X))$ is the two-form
obtained by combining wedge product and Poisson bracket, see
McDuff-Salamon \cite[Lemma 8.1.6]{ms:jh} whose sign convention is
slightly different.  Given a map $u : C \to X$, define
\begin{equation} \label{cr} 
\olp_{J,K} u := (\d u + \widehat{K} )^{0,1} \in \Omega^{0,1}(C, u^*
TX).\end{equation}
The map $u$ is {\em (J,K)-holomorphic} if $\olp_{J,K} u = 0 $.
Suppose that $C$ is equipped with a compatible metric and $X$ is
equipped with a tamed almost complex structure and perturbation $K$.
The {\em $K$-energy} of a map $u: C \to X$ is
$$ E_K(u) := (1/2) \int_C | \d u + \widehat{K}(u) |_J^2 $$
where the integral is taken with respect to the measure determined by
the metric on $C$ and the integrand is defined as in \cite[Lemma
  2.2.1]{ms:jh}.  If $\olp_{J,K}u = 0$, then the $K$-energy differs
from the symplectic area $A(u) := \int_C u^* \omega$ by a term
involving the curvature from \eqref{RK}:
\begin{equation} \label{energyarea}
 E_K(u) = A(u) + \int_C R_K(u) .\end{equation}
In particular, if the curvature vanishes then the area is
non-negative.

The notion of Floer trajectory is obtained from the notion of
perturbed pseudoholomorphic map by specializing to the case of a strip
$$C = \R \times [0,1] = \{ (s,t) \ | \ s \in \R, t \in [0,1] \}.$$
Given $H \in C^\infty([0,1] \times X)$ let $K$ denote the perturbation
one-form $K = - H \d t$ and let $E_H := E_K$.  If $u: \R \times [0,1]
\to X$ has limits $x_\pm : [0,1] \to X$ as $s \to \pm \infty$ then the
energy-area relation \eqref{energyarea} becomes
$$ E_H(u) = A(u) - \int_{[0,1]} (x_+^* H - x_-^* H) \d t .$$
Let $\olp_{J,H} = \olp_{J,K}$ be the corresponding perturbed
Cauchy-Riemann operator from \eqref{cr}.  A map $u: C \to X$ is {\em
  $(J,H)$-holomorphic} if $\olp_{J,H} u =0$.  A {\em Floer trajectory}
for Lagrangians $L_0,L_1$ is a finite energy $(J,H)$-holomorphic map
$u: \R \times [0,1] \to X$ with $\R \times \{ b \} \subset L_b, b
=0,1$.  An {\em isomorphism} of Floer trajectories $u_0,u_1: \R \times
[0,1] \to X$ is a translation $\psi: \R \times [0,1] \to \R \times
[0,1]$ in the $\R$-direction such that $\psi^* u_1 = u_0$.  Denote by
$$\M(L_0,L_1) := \left\{ u: \R \times [0,1] \to X \ | \begin{array}{c}
  \olp_{J,H} u = 0 \\ u(\R \times \{b \}) \subset L_b, b \in \{ 0,1 \}
  \\ E_H(u) < \infty \end{array} \right\} / \R $$
the moduli space of isomorphism classes of Floer trajectories of
finite energy, with its quotient topology.

In order to apply the theory of Donaldson hypersurfaces we will need
our Hamiltonian perturbations to vanish, and in order to achieve
recall the correspondence between Floer trajectories with Hamiltonian
perturbation and trajectories without perturbation with different
boundary condition.  Let $H \in C^\infty([0,1] \times \R, X)$ be a
time-dependent Hamiltonian and let
$J \in \Map([0,1],\J_{\tau}(X,\omega))$ be a time-dependent almost
complex structure.  Suppose that $L_0,L_1$ are Lagrangians such that
$\phi_1(L_0) \cap L_1$ is transversal.  There is a bijection between
$(J_t,H_t)$-holomorphic Floer trajectories with boundary conditions
$L_0,L_1$ and $(\phi_{1-t}^{-1})^* J_t$-holomorphic Floer trajectories
with boundary conditions $\phi_1(L_0),L_1$ obtained by mapping each
$(L_0,L_1)$ trajectory $(s,t) \mapsto u(s,t)$ to the
$(\phi_1(L_0),L_1)$-trajectory given by
$(s,t) \mapsto \phi_{1-t}(u(s,t))$ \cite[Discussion after
(7)]{fhs:tr}.

A compactification of the space of Floer trajectories is obtained by
allowing sphere and disk bubbling.  A {\em nodal Floer trajectory}
consists of a nodal strip $C$ together with a map $u: C \to X$ 
satisfying the following conditions: 
\begin{enumerate}
\item 
$u( (\partial C)_b) \subset L_b$ for $b \in \{ 0,1 \}$ (see
\eqref{Cb}) 
\item $u$ is $(J_t,H_t)$-holomorphic in strip coordinates on each
  strip component $C_i - \{ w_i, w_{i+1} \}$ and
\item $u$ is $J_t$-holomorphic on each sphere and disk component on
  which $f$ is constant and equal to $t$; see \eqref{Cb} for the
  definition of $(\partial C)_b$.
\end{enumerate} 
An {\em isomorphism of marked nodal Floer trajectories} is an
isomorphism of marked nodal curves with disk structures intertwining
the markings and the maps to $X$. A marked nodal Floer trajectory $u:
C \to X$ is {\em stable} if it has finitely many automorphisms, that
is, $\# \Aut(u) < \infty$.  A component $C_v, v \in \Ver(\Gamma)$ of
$C$ is a {\em ghost component} if the restriction of $u$ to $C_v$ has
zero energy, that is, $E_H(u | C_v) = 0$.  A nodal marked Floer
trajectory $u: C \to X$ is stable if and only if any sphere ghost
component $C_v \subset C$ has at least three special points $z_i \in
C_v, w_k \in C_v$ and any disk ghost component has either (a) at least
three boundary special points or (b) one interior special point and
one boundary special point.  Let
$$\ol{\M}(L_0,L_1) = \left\{ u: C \to X \ | \ \begin{array}{c}
\olp_{J,H} u = 0 , \ E_H(u) < \infty \\ \ u( (\partial C)_b) \subset
L_b, b \in \{ 0,1 \} \end{array} \right\}/ \sim $$
denote the moduli space of isomorphism classes of stable Floer
trajectories of finite energy.  $\ol{\M}(L_0,L_1)$ may be equipped
with a topology similar to that of pseudoholomorphic maps, see for
example McDuff-Salamon \cite[Section 5.6]{ms:jh}: Gromov-Floer
convergence defines a topology in which convergence is Gromov-Floer
convergence, by proving the existence of suitable local distance
functions.  Given $x_\pm \in \cI(L_0,L_1)$, denote by
$$\ol{\M}(L_0,L_1,x_+,x_-) := \left\{ [u] \in \ol{\M}(L_0,L_1) \ |
\ \lim_{s \to \pm \infty} u(s,\cdot) = x_\pm \right\} $$
the moduli space with fixed limits along the strip-like ends near $z_\pm
\in \partial C$.   

In preparation for the transversality arguments later, we review the
Fredholm theory for the moduli space of Floer trajectories.  Let $C$
be a treed strip of type $\Gamma$.  We denote by $C^\times$ the curve
with strip-like ends obtained by removing the nodes connecting strip
components and incoming and outgoing markings.  For integers $p > 2$
and $k$ sufficiently large (for the moment, $k > 2/p$ suffices) we
denote by $\Map^{k,p}(C^\times,X,\ul{L})$ the space of continuous maps
of Sobolev class $W^{k,p}$ on each component
$C_v \subset C^\times$ and taking the given Lagrangian boundary
conditions:
$$ \Map^{k,p}(C^\times,X) := \left\{ u: C^\times \to X, \  \forall m, \ \Vert u
| C_v \Vert_{k,p} < \infty,\  u((\partial C_b) \subset L_b, b \in
\{ 0 , 1 \} \right\} .$$
Here the Sobolev $k,p$-norm $\Vert \cdot \Vert_{k,p}$ is defined using
a covariant derivative on $X$ and $C^\times$ of standard form on the
strip-like ends as in, for example, Schwarz \cite{sch:coh}.  Denote by
$ \Omega^0(C^\times, u^* (T X, TL) )_{k,p}$ the space of continuous
sections of Sobolev class $k,p$ on each component, taking values in
the pull-back $u^* TX$ of the cotangent bundle, and with boundary
values in $ (u |
 \partial C^\times_b)^* T L_b$ on $\partial C^\times_b \subset
 \partial C$
 on $f^{-1}(b) \cap \partial C$.  Given time-dependent metrics that so
 that $L_b$ are totally geodesic for time $b \in \{ 0,1 \}$, we have
 geodesic exponentiation maps
\begin{equation} \label{mapcharts}
\Omega^0(C^\times, u^*(T X, T\ul{L}) )_{k,p} \to \B_{k,p} :=
\Map^{k,p}(C^\times,X,\ul{L}), \quad \xi \mapsto
\exp_u(\xi) \end{equation}
which provide charts for $\B_{k,p}$.  With these charts $\B_{k,p}$
admits the structure of a smooth Banach manifold by standard Sobolev
estimates.  Define a fiber bundle $\E_{k,p}$ over $\B_{k,p}$ whose
fiber at $u$ is
\begin{equation} (\E_{k,p})_u := \Omega^{0,1}(C^\times,
u^* TX)_{k-1,p}, \end{equation} 
that is, a product of $(0,1)$-forms over the smooth components of
$C^\times$.  Local trivializations of $\E_{k,p}$ are defined by geodesic
exponentiation from $u$ and parallel transport using the Hermitian
connection defined by the almost complex structure 
$$ \Phi_\xi: \Omega^{0,1}(C^\times, u^* TX)_{k-1,p} \to 
\Omega^{0,1}(C^\times, \exp_u(\xi)^*TX)_{k-1,p} $$
see for example \cite[p. 48]{ms:jh}.  The Cauchy-Riemann operator
defines a section
\begin{equation} \label{here0} 
\olp_{J,H} : \ \B_{k,p} \to \E_{k,p}, \quad u \mapsto \olp_{J,H} u
.\end{equation}
Define a non-linear map on Banach spaces using the local
trivializations
$$ \cF_{u}: \Omega^0(C^\times, u^* (TX,T\ul{L}))_{k,p} \to
\Omega^{0,1}(C^\times, u^* (TX,T\ul{L}))_{k-1,p}, \quad \xi \mapsto
\Phi_\xi^{-1} \olp_{J,H} \exp_{u}(\xi) .$$
The quotient of the zero level set $\cF_u^{-1}(0)/\Aut(C^\times)$ is
naturally identified with a subset of $\M(L_0,L_1)$ and gives a local
description of the moduli space.

The local smoothness of the moduli space can be guaranteed by the
surjectivity of the linearized operator by the implicit function
theorem for maps between Banach spaces.  Let $D_u$ denote the
linearization of $\cF_u$ (see Floer-Hofer-Salamon \cite[Section
5]{fhs:tr}).  Standard arguments show that this operator is Fredholm,
see Lockhart-McOwen \cite{loc:ell} and Donaldson \cite[Section
3.4]{don:floer}.  A stable non-nodal Floer trajectory
$u: C^\times \to X$ is {\em regular} if $D_u$ is surjective.  If
$C = \R \times [0,1]$ and $u: C^\times \to X$ is a regular Floer
trajectory then $\olp_{J,H}^{-1}(0)$ is a smooth, finite dimensional
manifold near $u$.  The action of $\Aut(C)$ is free and proper, and
the quotient is Hausdorff by exponential decay at the ends, equivalent
to finiteness of the energy.  It follows that $\M(L_0,L_1)$ is a
smooth manifold near any regular $[u]$.  Floer-Hofer-Salamon
\cite[Proof of Theorem 5.1]{fhs:tr} show, at least in the case of
periodic Floer cohomology, that there exist time-dependent almost
complex structures such that every Floer trajectory is regular.  We
avoid the use of \cite{fhs:tr} by using stabilizing divisors to
construct coherent perturbations in the next section.

\section{Coherent perturbations} 

In this section we describe how to regularize the moduli space of
Floer trajectories using stabilizing divisors.  We assume the reader
is at least modestly familiar with the results of Cieliebak-Mohnke
\cite{cm:trans}, which show how to use stabilizing divisors to kill
the automorphisms of spheres appearing in the compactification of
moduli spaces of pseudoholomorphic maps.

\subsection{Stabilizing divisors} \label{stab}

As mentioned in the introduction, our goal is to kill automorphism
groups of disks appearing in the compactification by adding markings
corresponding to intersection points with a codimension two symplectic
submanifold produced by Donaldson's construction.  Because any such
disk has at least one special point on the boundary, it suffices to
choose divisors so that any non-constant holomorphic disk intersects
the divisor at least once.

We introduce the following terminology.  By a {\em divisor} we mean a
closed codimension two symplectic submanifold $D \subset X$.  An
almost complex structure $J: TX \to TX$ is {\em adapted} to a divisor
$D$ if $D$ is an almost complex submanifold of $(X,J)$.

\begin{definition} \label{defi_stab}
A divisor $D \subset X$ is {\em stabilizing} for a Lagrangian
submanifold $L \subset X$ if and only if
\begin{enumerate}
\item the divisor $D$ is disjoint from the Lagrangian $L$, that is, 
$L \cap D = \emptyset$; and 
\item any positive area disk $u: (C,\partial C) \rightarrow (X,L)$
  intersects the divisor $D$ in at least one point:
$\omega([u])>0 \implies u^{-1}(D) \neq \emptyset .$
\label{cond_b}
\end{enumerate} 
\end{definition} 

\begin{example} \label{exact} Suppose that the Lagrangian $L$ is {\em
    exact} in $(X-D,\omega)$, that is, for some one-form
  $\alpha \in \Omega^1(X - D)$ and function $ \phi \in \Omega^0(L)$ we
  have $\d \alpha = \omega$ and $\d \phi = \alpha |_L$.  In this case
  $D$ satisfies condition \eqref{cond_b}. Indeed, since
  $ \omega = d\alpha $ on $X - D$ and $\alpha | L = \d \phi$ for some
  $\phi: L \to \R$, the integral of $\omega$ over any disk
  $u: (C,\partial C) \rightarrow (X-D,L)$ vanishes:
$$ \int_C u^* \omega= \int_{\partial C} \alpha = 0 .$$
Thus in particular any disk with positive area must intersect the divisor. \end{example}
 
Stabilizing divisors for Lagrangians exist under suitable rationality
assumptions.  We say that $X$ is {\em rational} if
$[\omega] \in H^2(X,\Q)$.  Let $X$ be rational and denote by $m_0$ the
{\em integrality constant} given by the minimal positive integer such
that $m_0 [\omega] \in H^2(X,\Z)$.  We say that a divisor $D$ has {\em
  degree $k>0$} if $[D]$ is Poincar\'{e} dual to $m_0k[\omega]$.  For
later use we also define {\em torsion constant}
$$t_0 =
|\Tor(H_1(L))|=|\Tor(H^2(L))| .$$ 

\begin{lemma} 
  If $X$ is rational and $L \subset X$ is a Lagrangian submanifold
  then there exists a codimension two submanifold $D \subset X$ representing a
  positive multiple of $[\omega]$ in the complement of $L$.
\end{lemma} 

\begin{proof} By the rationality assumption, there exists a complex
  line bundle $\ti{X} \to X$ whose first Chern class $c_1(\ti{X})$ is
  equal to $m_0 [\omega]$. Since $[\omega|_{L}]=0\in H^2(L,\R)$, the
  first Chern class of the restriction $c_1(\ti{X}|_L)$ is in the
  torsion submodule $\Tor(H^2(L))$ of $H^2(L)$.  Note that the tensor
  product $\ti{X}^{\otimes t_0}|_L$ is topologically trivial.  Choose
  a section $s: X \to \ti{X}^{\otimes t_0}$ that is non-vanishing on
  $L$.  A generic perturbation $s + \sigma$ of $s$ is transverse to
  the zero section (for instance, by Sard-Smale applied to sections of
  some differentiability class.)  The zero set
  $D = (s + \sigma)^{-1}(0)$ is then smooth by the implicit function
  theorem and disjoint from $L$.\end{proof}

The following is an example of a Lagrangian disjoint from a Donaldson
hypersurface but not exact in the complement.  Consider a circle in
the symplectic two-sphere, with Donaldson hypersurface given by a
point.  One of the disks bounding the circle meets the hypersurface,
but the other does not.  In this case one sees that in general there
is no relation between the area and the intersection number of a disk
with the hypersurface.  However, in the exact case discussed in
Example \ref{exact} we have the following relation between area of
disks with boundary in the Lagrangian and their intersection number
with codimension two submanifolds as in the previous paragraph.

\begin{lemma} 
  \label{rema_exac} Suppose that $u: (C,\partial C) \to (X,L)$ is a
  disk with boundary in $L$ and $D \subset X - L$ is an oriented
  codimension two submanifold given as the zero set of a section
  $s: X \to \ti{X}$ such that $L$ is exact in $X - D$.  Then the
  intersection number is proportional to the area:
$$ u_*([C])\cdot [D] = m_0 \int_C u^* \omega .$$
\end{lemma} 

\begin{proof} The proof is an integration-by-parts computation.  Fix
  on $\ti{X}$ a Hermitian metric and a unitary connection with
  curvature $-2\pi i m_0 \omega$.  Let $\ti{X}_1 \subset \ti{X}$
  denote the unit circle bundle, and suppose that
$$\alpha_s \in \Omega^1(\ti{X}_1 |X- D),  \quad \pi^* (-2 \pi i m_0 \omega) = \d
\alpha_s $$
is the connection one-form of $\nabla$ with respect to the
trivialization defined by the section $s$.  For any curve
$u: (C,\partial C) \to (X,L)$, the intersection multiplicity
$\mu(u,z)$ of any intersection point $x \in u^{-1}(D)$ is the residue
of $\alpha_s$,
$$ \mu(u,z) = \lim_{ \eps \to 0} \int_{\partial B_\eps(z)} u^* \alpha_s
.$$
The total intersection number is the sum of the local intersection
multiplicities
$$ u_*([C]) \cdot [D] = 
\sum_{z \in u^{-1}(D)} \mu(u,z) .$$
Since $L$ is exact in $X - D$, there exists a function
$\phi_s: L \to \R$ such that $\d \phi_s = \alpha_s$.  By Stokes'
theorem the area is related to the intersection multiplicity by
\begin{eqnarray} \label{area}
 \int_C u^* \omega &=& (1/m_0) \lim_{\eps \to 0} \sum_{ z \in u^{-1}(D)}
\int_{\partial B_\eps(z)} u^* \alpha_s + (1/m_0) \int_{\partial C} u^*
\alpha_s \\ &=& (1/m_0) \sum_{ z \in u^{-1}(D)} \mu(u,z) + \int_{\partial
  C} u^* \d \phi_s \\ & =& (1/m_0) u_*([C])\cdot [D]
 .\end{eqnarray}
\end{proof} 

There are various notions of what it means for a Lagrangian to be
rational.  For example, one could require simply that the relative
cohomology class $[\omega]$ lies in the rational relative cohomology
group $H^2(X,L;\Q)$.  This is equivalent, by relative Poincar\'e
duality, to requiring that relative two-cycles having rational area.
However, a slightly stronger condition is convenient for the
construction of stabilizing divisors:

\begin{definition} 
  Given a rational symplectic manifold $X$, an immersed Lagrangian
  submanifold $L \subset X$ is {\em rational} or {\em rationally
    Bohr-Sommerfeld} if there exists a line bundle $\ti{X}$ with a
  connection $\alpha \in \Omega^1(\ti{X}_1)$ (here $\ti{X}_1$ is as
  above the unit circle bundle) whose curvature satisfies
$$\exists k \in \Z, \quad \curv(\alpha)  = 
-2\pi i k m_0 \omega$$
and that restricts to a trivial line-bundle-with-connection on $L$.
That is, $L$ is rational if and only if
$\ti{X} |_L \cong L \times \C$
as line bundles with connection.
\end{definition} 

Rationality can be phrased in terms of holonomies as follows.  Given a
line-bundle-with-connection $\ti{X} \to X$, the restriction of
$\ti{X}$ to $L$ is automatically a {\em flat} line bundle by the
Lagrangian condition.  Flatness implies that the parallel transports
give rise to a {\em holonomy representation}
$$h:  \pi_1(L,l) \to U(1)$$ 
at any base point $l \in L$.  The Lagrangian $L$ is rational for 
$\ti{X}$ if and only if the restriction $\ti{X} |_L $ is trivial if 
and only if its holonomy representation $h$ is trivial for all base 
points $l \in L$.  It suffices to check the condition for one base 
point in each connected component of $L$.  Moreover, using 
trivializations of $\tilde{X}$ over disks representing $\pi_2(X,L,l)$,
one gets a refined holonomy representation 
$h_2(\pi_2(X,L,l)) \overset{h_0}{\to} \R$ that is proportional to the 
integral of $\omega$ over the disks. Thus $L$ will be rational if and 
only if there exists $r >0$ such that the image of this representation 
has the form $\Z \cdot r$ for all base points $l \in L$. 

In particular the rationality condition is slightly stronger than the
condition of rationality for the relative symplectic class: there may
be loops in the Lagrangian which are not bounding, which nevertheless have
rational holonomies by the above condition.  On the other hand, the
rational Bohr-Sommerfeld condition implies rationality of the relative
symplectic class since the pairings with relative cycles may be
computed by Stokes' theorem.

In order to produce stabilizing divisors we apply a modification of
Donaldson's construction \cite{don:symp}, \cite{auroux:remark} introduced by
Auroux-Gayet-Mohsen \cite{auroux:complement}, with a minor improvement
in the exactness property for rational Lagrangians.  Donaldson's
construction proceeds by the construction of {\em approximately
  holomorphic} sequences of sections of line bundles.  These sections
are then carefully perturbed to obtain the approximately holomorphic
submanifolds.  Fix an $\omega$-compatible almost complex structure
$J\in \mathcal{J}(X,\omega)$ so that $\omega(\_,J\_)$ defines a
Riemannian metric $g$ on $X$.  Recall that for any $2m$-dimensional
submanifold $D \subset X$ the K\"ahler angle is a map measuring the
failure of $D$ to be an almost complex submanifold.  The K\"ahler
angle is defined by
$$\on{ang}_D: D \to [0, \pi], \quad 
x \mapsto \cos^{-1}\left( \frac{ \omega_x^m }{  m! \Omega_{T_xD}}
\right) $$
where $\Omega_{T_xD}$ is the volume form defined by the metric and
orientation \cite[p. 325]{cm:trans}.  A codimension two submanifold
$D \subset X$ is called {\em $\theta$-approximately $J$-holomorphic}
if its K\"ahler angle satisfies
$\on{ang}_D(x) < \theta,\forall x \in D$.

\begin{theorem} \label{bs_divi} Let $X,L \subset X$ be rational and
  $J \in \J_\tau(X,\omega)$.  There exists an integer $k_{m} >0$ such
  that for every $\theta>0$ there is an integer $k_\theta>0$ such that
  for every $k > k_\theta$ there exists a $\theta$-approximately
  $J$-holomorphic divisor $D$ of degree $t_0k_{m} k$ stabilizing for
  $L$ and such that $L$ is exact in $(X-D,\omega)$.
\end{theorem}

\begin{proof}
  The rationality assumption allows us to start, in Donaldson's
  construction in \cite{don:symp}, with a section that is covariant
  constant over the Lagrangian.  We suppose that $\omega$ is integral
  (that is, $m_0=1$) so that there is a line bundle $\ti{X}$ over $X$
  with Hermitian connection $\nabla$ whose curvature is
  $-2 \pi i \omega$ and such that $\ti{X}^{\otimes t_0}|_L$ is
  trivial.  Since $L$ is rational, there is an integer $k_{m}>0$ such
  that $\ti{X}^{\otimes t_0k_{m}}|_L$ together with the connection
  $\nabla^{\otimes t_0k_{m}}$ is trivial. There is therefore a choice
  of {\em exactly} flat unitary sections
  $\tau_{k}:L \rightarrow \ti{X}^{\otimes t_0k_{m}k}|_L$ over the
  Lagrangian at the beginning of the Auroux-Gayet-Mohsen construction
  \cite{auroux:complement}, that choice being unique up to the $S^1$
  action given by scalar multiplication on the fibers.  As explained
  in Auroux-Gayet-Mohsen \cite{auroux:complement}, Donaldson's
  perturbation scheme gives asymptotically holomorphic sections $s_k$
  which define codimension two symplectic submanifolds
  $D=s_k^{-1}(0) \subset X-L$ for $k$ sufficiently large.

  It remains to show that the Lagrangian $L$ is exact in the
  complement of the submanifolds $D$ produced in the previous
  paragraph.  The quotient $ s_k/ | s_k | $ gives a trivialization of
  the line bundle away from $D = s_k^{-1}(0)$.  The connection
  one-form $\beta_k \in \Omega^1(X - D)$ for
  $\nabla^{\otimes t_0k_{m}k}$ with respect to this trivialization
  satisfies $ \d \beta_k = \omega |_{X - D}$, which vanishes on $L$.
  The flat connection defined by the trivialization $s_{k} |_L $
  (which is only approximately equal to
  $\nabla^{\otimes t_0k_{m}k} |_L $) has the same trivial holonomies
  as the restriction of $\nabla^{\otimes t_0k_{m}k}$ to $L$.  Hence,
  the two connections with flat sections $s_{k}, \tau_{k}$ are gauge
  equivalent by some gauge transformation $g_{k}: L \to U(1)$.  Thus
  $\beta_{k} = g_{k}^{-1} \d g_{k}$.  We wish to show that $g_{k}$ is
  in the identity component of the group of gauge transformations.
  Lemma 4 of \cite{auroux:complement} states that the phase of the
  restriction of $s_{k}$ to $L$ with respect to $\tau_{k}$ is less
  than $\pi/4$ in absolute value.  It follows that $s_{k} $ is related
  to the trivial section $\tau_{k} $ by the exponential of an
  infinitesimal gauge transformation:
\begin{equation}  \label{infgauge}
 \exists \xi_k : L \to \R, \quad   s_{k} |_L = \exp(2 \pi i \xi_k)
\tau_k . \end{equation}
So $ \beta_{k} |_L = \d \xi_{k} $ is exact as claimed.
\end{proof}

Donaldson's construction is a symplectic imitation of well-known
results in algebraic geometry.  Many of the examples we have in mind
are actually smooth projective varieties, and the results of Donaldson
and Auroux-Gayet-Mohsen are not necessary.  Indeed smooth divisors in
smooth projective varieties exist by Bertini's theorem
\cite[II.8.18]{ha:al}.  Holomorphic sections concentrating near
rational Lagrangian submanifolds exist by e.g. results of
Borthwick-Paul-Uribe \cite{bo:le}.  Generic perturbations of the
divisors corresponding to these sections are smooth and intersect any
other divisor transversally, by Bertini again.

The following are well-known properties of the (rational)
Bohr-Sommerfeld condition:

\begin{example} 
\begin{enumerate} 
\item {\rm (Disjoint unions)} Let $X$ be a rational symplectic
  manifold and $\ti{X} \to X$ a linearization.  Let
  $L_0,L_1 \subset X$ be rational disjoint Lagrangians so that
  $\ti{X} |_{L_b}$ is trivial for $b =0,1$.  Then the restriction of
  $\ti{X}$ to the disjoint union $L_0 \cup L_1$ is again trivial, so
  $L_0 \cup L_1$ is also rational.
\item {\rm (Products)} If $X_i$ is a rational symplectic
  manifold with linearization $\ti{X}_i \to X_i$ with $c_1(\ti{X}_i)=
  k_i [\omega_i]$,  $i = 0,1$, then 
$ \ti{X}_0^{\otimes k_1} \boxtimes \ti{X}_1^{\otimes k_0} \to X_0
  \times X_1$
is a linearization for $X_0 \times X_1$.  If $L_i \subset X$ is
rational for $X_i$ for the line bundle $\ti{X}_i$, 
$i = 0,1$, then
$L_0 \times L_1 \subset
  X_0 \times X_1$
is a rational Lagrangian with line bundle $ \ti{X}_0^{\otimes k_1}
\boxtimes \ti{X}_1^{\otimes k_0}$.  
\item {\rm (Duals)} Let $(X,\omega)$, or $X$ for short, be a
  symplectic manifold.  The {\em dual} $(X,\omega)^- = (X,-\omega)$ of
  $(X,\omega)$ is obtained by changing the sign on the symplectic
  form.  A linearization of $X^-$ is given by the inverse line bundle
  $ \ti{X}^{-1}$.  Any rational Lagrangian $L \subset X$ is also
  rational in the dual $X^-$ since the trivialization of $\ti{X}$ on
  $L$ induces a trivialization of the inverse $\ti{X}^-$.
\item {\rm (Diagonals)} Let $X$ be a linearized symplectic manifold
  and $\Delta \subset X^- \times X$ the diagonal Lagrangian.  The
  restriction satisfies   $  ( \ti{X}^- \boxtimes \ti{X}) |_\Delta \cong \ti{X}^-
  \otimes \ti{X} \cong X \times \C $.  It follows that 
 $\Delta$ is rational.
\item {\rm (Hamiltonian perturbations)} Let
  $H \in C^\infty(\R \times X)$ be a time-dependent Hamiltonian with
  Hamiltonian vector field  $\widehat{H} \in \Map(\R, \Vect(X))$.
  If $L \subset X$ is rational, then so is $\phi_t(L)$ for any
  $t$. Indeed the flow $\phi_t$ lifts to an isomorphism
  $\ti{\phi_t}: \ti{X} \to \ti{X}$ of line bundles with connection: If
  $\alpha \in \Omega^1(\ti{X}_1)$ is the connection one-form and
  $\ti{H} \in \Vect(\ti{X})$ is the lift of $\widehat{H}\in \Vect(X)$
  with $\alpha(\ti{H}) = - H$ then by Cartan's formula
$$ L_{\ti{H}} \alpha = \iota( \ti{H} ) \d \alpha + \d \iota(\ti{H})
 \alpha = \iota(\hat{H}) \omega - \d H = \d H - \d H = 0 .$$
After restriction one obtains an isomorphism from $\ti{X} |_{\phi_t(L)}$
to $\ti{X} |_L$.
\item {\rm (Graphs of Hamiltonian diffeomorphisms)} In particular, if
$\phi: X \to X$ is a Hamiltonian diffeomorphism then $(1 \times \phi)
\Delta \subset X \times X$ is a rational Lagrangian.
\end{enumerate} 
\end{example} 

It follows from the last item above that transversally-intersecting
rational Lagrangians may always be perturbed so that the union is
rational.  Indeed choose a Hamiltonian perturbation
$H \in C^\infty(X)$ near any intersection point $p \in L_0 \cap L_1$
with $H(p)$ non-zero and is tangent to $L_0$ in a neighborhood of $p$.
Let $\phi_t:X \to X$ denote the flow of $H$ and
$p_t \in L_0 \cap \phi_t(L_1)$ (for $t$ small) the unique family of
intersection points with $p_0 = p$, given by flowing $p_t$ under
$\widehat{H}$.  The trivializing flat section of $\phi_t(L_1)$ is obtained
by flowing the trivializing flat section of $L_1$ under $\ti{H}$.  The
phase change at $p_t$ between the flat sections over $L_0$ and
$\phi_t(L_1)$ is the integral of $H(p_s), s \in [0,t]$, and can be
made rational by suitable choice of $H$.  After suitable tensor
products the phase changes become identities and so the flat sections
glue to a flat section over the union.

For a general Lagrangian submanifold, it is still possible to find
divisors that still intersect {\em holomorphic} non-constant disks for
a given holomorphic structure.  We introduce the following definition.

\begin{definition} \label{defi_weak}
A divisor $D \subset X$ is {\em weakly stabilizing} for a Lagrangian 
submanifold $L \subset X$ if and only if
\begin{enumerate}
\item[(a)\hspace{0.12cm }] the divisor $D$ is disjoint from the
  Lagrangian $L$, that is, $L \cap D = \emptyset$, and
\item[(b)'] there exists an almost-complex structure
  $J_D \in \mathcal{J}(X,\omega)$ adapted to $D$ such that any
  $J_D$-holomorphic disk $u: (C,\partial C) \rightarrow (X,L)$ with
  $\omega([u])>0$ intersects $D$ in at least one point, that is,
  $u^{-1}(D) \neq \emptyset$.
    \label{cond_b'}
\end{enumerate}
\end{definition}

\begin{lemma} \label{rati_divi2} Let $L$ be a Lagrangian, not
  necessarily rational, and $J \in \J_\tau(X,\omega)$.  There exists
  an integer $k_{m}>0$ such that for every $\theta>0$, there is an
  integer $k_\theta>0$ such that for every $k>k_\theta$ there exists a
  $\theta$-approximately $J$-holomorphic divisor $D$ of degree
  $t_0k_{m} k$ that is weakly stabilizing for $L$.
\end{lemma}

\begin{proof}
  We will choose the divisor as in Auroux-Gay-Mohsen
  \cite{auroux:complement}, and then show that the relationship
  \eqref{area} between intersection number and area holds up to a
  small error.  In order to carry this out we wish to choose the
  sections $\tau_k$ so that they satisfy a topological condition:
  Given a class $[u] \in h_2(\pi_2(X,L))$, the intersection number
  $[u]\cdot [D]$ is given by the degree of $\tau_k$ over $\partial u$,
  and is equal to the degree of $s_k$.  On the other hand, the
  holonomy (calculated in multiples $2\pi$) of
  $\nabla^{\otimes t_0 k}$ on $[\partial u] \in H_1(L)$ with respect
  to a trivialization over $[u]$ is $t_0 k \omega([u])$.  Indeed,
  denote by $\alpha_{\tau_k}$ the connection 1-form of
  $\nabla^{\otimes t_0k}|_L$ in the trivialization given by
  $\tau_k$. Note that since
  $\d (\alpha_{\tau_k}|_L) = (\d\alpha_{\tau_k})|_L = t_0 k \omega|_L
  \equiv 0$,
  $\alpha_{\tau_k}$ defines a class $[\alpha_{\tau_k}] \in H^1(L,\R)$
  which only depends on the homotopy class of the section
  $\tau_k$. Changing the class of $\tau_k$ amounts to changing the
  class of $\alpha_{\tau_k}$ by an element of $H^1(L,\Z)$. More
  precisely, if $\tau_k$ and $\tau_k'$ are two trivializing sections,
  then their connection $1$-forms satisfy
  $[\alpha_{\tau_k'}-\alpha_{\tau_k}] = [\alpha] \in H^1(L,\Z)$, so that
  changing the homotopy class of the trivializing section $\tau_k$
  corresponds to adding an integer class $[\alpha] \in H^1(L,\Z)$ to the
  class $[\alpha_{\tau_k}] \in H^1(L,\R)$.  As in \eqref{area} one has
  a relationship between the intersection number and area:
\begin{equation} \label{inteform}
[u]\cdot [D] = t_0k \int_{C} u^*\omega - \int_{\partial C} 
u^*\alpha_{\tau_k} 
\end{equation}
for every $[u] \in h_2(\pi_2(X,L))$. 

Next we obtain a bound of the homology class of a disk in terms of its
boundary class.  Choose bases $\{\beta_1, \ldots, \beta_p\}$ and
$\{ \alpha_1, \ldots, \alpha_q\}$ for (the torsion-free part of)
$H^2(X,L)$ and $H^1(L)$, respectively. Define
$$\| [u] \|= \underset{j}{\max}\{|\beta_j([u])|\}, \quad 
\| [\partial u] \|= \underset{j}{\max}\{|\alpha_j([\partial u])|\} .$$
Let $\| \partial \|$ be the norm of $\partial$ with respect to the
latter norms, so that
\[
\| [\partial u] \| \leq \| \partial \| \| [u] \|
\]
for every $[u] \in h_2(\pi_2(X,L))$.  By uniform boundedness of the
forms $\alpha_{\tau_k}$, there exists a constant $C_{\alpha} > 0$ such
that for all $[u]$,
\begin{equation}
|\alpha_{\tau_k}([\partial u])| \leq C_{\alpha} \| [\partial u]\|  \leq C_{\alpha} \|\partial \| \|[u]\| \label{equ_top}.
\end{equation} 

Secondly, we show that the homology class of a holomorphic disk can be
bounded in terms of its area: for every almost-complex structure
$J_D \in \J_\tau(X,\omega)$ such that $\| J_D - J \| < \theta'$, there
exists a constant $C_{\beta, \theta'}>0$
\begin{equation} \label{equ_hol}
\| [u] \| \leq C_{\beta, \theta'} \omega([u]) 
\end{equation}
for every $J_D$-holomorphic disk $u: (C,\partial C) \rightarrow
(X,L)$.  Equation \eqref{equ_hol} is mostly a relative version of
\cite[Proposition 4.1.5]{ms:jh}
or \cite[Lemma 8.19]{cm:trans}.  Its proof goes as follows.  For every
$2$-form $\hat{\beta} \in \Omega^2(X,L)$ on $X$ vanishing on $L$ we
have
\begin{eqnarray*} 
\hat{\beta}(v,J_{D}v) &=& \hat{\beta}(v,(J + (J_{D}-J)) v) \\
&\leq&  \|\hat{\beta}\|_{L^\infty} \|v\|^2 + \|\hat{\beta}\|_{L^\infty} \theta' \|v\|^2 \\
&=& (1+\theta') \|\hat{\beta}\|_{L^\infty} \|v\|^2 \end{eqnarray*}
where $\|J_D-J\| < \theta'$ and the norms are again taken with respect
to the metric $g(\_,\_) = \frac{1}{2}(\omega(\_,J\_) +
\omega(J\_,\_))=\omega(\_,J\_)$.  Similarly, $\omega(v,J_{D}v) \geq
(1-\theta') \|v\|^2$, so that
\[
\hat{\beta}(v,J_D v) \leq \frac{1+\theta'}{1-\theta'} \|\hat{\beta}\| 
\omega(v,J_{D}v).
\]
Therefore, if $u: (C,\partial C) \rightarrow (X,L)$ is a $J_D$-holomorphic disk,
\[
\hat{\beta}([u]) = \int_{C} \hat{\beta} \leq \frac{1+\theta'}{1-\theta'} 
\|\hat{\beta}\| \int_{C} \omega = \frac{1+\theta'}{1-\theta'} \|\hat{\beta}\| 
\omega([u])
\]
and 
\begin{eqnarray*} 
\|[u]\| &=& \underset{j}{\max}\{|\beta_j([u])|\} \leq  \frac{1+\theta'}{1-\theta'} \|\hat{\beta_j}\| \omega([u]) \\
& \leq&  \frac{1+\theta'}{1-\theta'} \underset{j}{\max}\{\|\hat{\beta_j}\|\}
\omega([u]) 
 =  C_{\beta,\theta'} \omega([u]) \end{eqnarray*} 
where $\hat{\beta_j}$ is a representative of $\beta_j$, $1 \leq j \leq p$, and 
$C_{\beta,\theta'} = \frac{1+\theta'}{1-\theta'}
\underset{j}{\max}\{\|\hat{\beta_j}\|\}$.

Combining inequalities \eqref{equ_top} and \eqref{equ_hol} gives a
uniform bound on the boundary term of \eqref{inteform} of any
holomorphic disk in terms of its area:
\[
|\alpha_{\tau_k}([\partial u])| \leq C_{\alpha} C_{\beta,\theta'}
\|\partial\|  \omega([u])
\]
for every $J_D$-holomorphic disk
$u: (C,\partial C) \rightarrow (X,L)$.  Then equation \eqref{inteform}
gives
\begin{eqnarray*}
[u]\cdot [D] &=& t_0k \int_{C} u^*\omega - \int_{\partial C}
u^*\alpha_{\tau_k} \geq t_0k u([\omega]) - C_{\alpha} \|\partial\|
C_{\beta,\theta'} \omega([u]) \\ &=& (t_0k - C_{\alpha} \|\partial\|
C_{\beta,\theta'} ) \omega([u])
\end{eqnarray*} 
where the constants $C_{\alpha}$ and $C_{\beta,\theta'}$ do not
depend on $k$.
Hence for large enough values of $k$, $[u]\cdot [D] > 0$ for 
every $J_D$-holomorphic disk $u: (C,\partial C) \rightarrow (X,L)$
with $\omega([u])>0$.
\end{proof}

\begin{remark}
\begin{enumerate} 
\item The sections used in the proof of Lemma \ref{rati_divi2} are the
  same that are used in \cite{auroux:complement}. The associated
  divisors become weakly stabilizing in the sense of \eqref{defi_weak}
  as $k$ increases.

\item In the case of a general Lagrangian $L$ as in Lemma
  \ref{rati_divi2}, the intersection numbers of the $J_D$-holomorphic
  disks with boundary in $L$ with the divisor $D$ are no longer
  necessarily proportional to their symplectic areas.

\item The reverse isoperimetric inequality of Groman and Solomon
  \cite{gs:iso} would allow for a more direct proof of the above
  result if $J_D$ could be chosen to be $C^2$-close to $J$ instead of
  $C^0$-close as we are given.  Indeed, \cite[Theorem 1.1]{gs:iso}
  states that there exists a constant $f>0$ such that
$
\on{length}(\partial u,g) \leq f \omega([u])
$
for every $J$-holomorphic disk $u: (C,\partial C) \rightarrow (X,L)$,
while \cite[Theorem 1.4]{gs:iso} extends that inequality to a
$C^2$-neighborhood of $J$. Then one could argue that since 
$|\nabla^{\otimes t_0 k} \tau_k| < C_L$ for every $k>0$,
\[
|\alpha_{\tau_k}([\partial u])| = \left|\int_{\partial u} \nabla^{\otimes t_0 k} \tau_k \right| 
\leq C_L \on{length}(\partial u,g) \leq C_L f \omega([u])
\]
uniformly for all $J_D$-holomorphic disks.
\end{enumerate}
\end{remark}

Next we turn to pairs of Lagrangians, needed to define Lagrangian
Floer theory.  Note that divisors that are stabilizing for each
$L_0,L_1$ individually, are not necessary stabilizing for the union.
For example, take $L_0$ to be a circle in the symplectic two-sphere,
and $L_1$ to be a small Hamiltonian perturbation.  Even if $D$ is
stabilizing for both $L_0$ and $L_1$, there may be non-constant
holomorphic strips that do not intersect $D$.  The following Lemma
provides divisors so that holomorphic strips with boundary in
$L_0,L_1$ intersect the divisor at least once.

\begin{lemma}
\label{auroux}
\begin{enumerate}
\item \label{lpair} {\rm (Existence for pairs of rational
Lagrangians)}  
  Suppose that $L_0,L_1$ are  rational Lagrangians intersecting
  transversally and $L_0 \cup L_1$ is  rational with line
  bundle $\ti{X}$.  Then there exists a divisor $D$ such that $L_0
  \cup L_1$ is exact in $X - D$ in the sense that there exists a
  one-form 
$$\alpha \in \Omega^1(X - D), \quad \d \alpha = \omega | (X
  - D)$$ 
  and functions $\phi_0: L_0 \to \R, \ \phi_1: L_1 \to \R$ such that
$$
d \phi_b =
\alpha | L_b, b \in \{0,1 \}, \quad \phi_0 |_{L_0 \cap L_1} = \phi_1
|_{L_0 \cap L_1} .$$

\item \label{lpair2} {\rm (Existence for pairs of Lagrangians)} Let
  $L_0,L_1$ be Lagrangians intersecting transversally, let
  $t_0 = |\Tor(H_1(L_0\cup L_1))|$ and let $J \in
  \J_\tau(X,\omega)$.
  There exists an integer $k_{m}>0$ such that for every $\theta>0$,
  there is an integer $k_\theta>0$ such that for every $k>k_\theta$
  there exists a $\theta$-approximately $J$-holomorphic divisor $D$ of
  degree $t_0k_{m} k$ that is weakly stabilizing for $L_0 \cup L_1$.

\item \label{uniq} {\rm (Uniqueness)} Suppose that
  $(J_t \in \J_\tau(X,\omega), t\in [0,1])$ is a time-dependent almost
  complex structure. Let $D_{t}$ be stabilizing divisors for $L$
  (resp. a pair $L_b, b = 0,1$) constructed as the zero set of
  approximately $J_t$-holomorphic sections
  $s_{k,t}: X \rightarrow \ti{X}^{\otimes t_0k_{m}k}$ built from
  homotopic unital sections
$$\tau_{k,t}: L \rightarrow \ti{X}^{\otimes t_0k_{m}k}|_L,
  \ k>k_{0,t}, \ t \in \{ 0, 1\}$$ 
(resp. sections $s_{b,k,t}$ and $\tau_{b,k,t}$ for $b = 0,1$.)
 
  Then there exists an integer $k_{m}>0$ such that for every
  $\theta>0$ there is an integer $k_\theta$ such that for
  $k>k_\theta$, there exists a smooth family $D_{t}$, $t \in [0,
    1]$, of $\theta$-approximately $J_t$-holomorphic divisors
  stabilizing for $L$ (resp. $L_0 \cup L_1$) of degree $t_0 kk_{m}$
  connecting $D_{0}$ and $D_{1}$. Moreover, there exists a
  symplectic isotopy $\{\phi_t \}_{0\leq t\leq 1}$ preserving $L$ such
  that $D_{t} = \phi_t(D_{0})$.
 \end{enumerate} 
\end{lemma}

\begin{proof}  
  For part (a), suppose $L_0 \cup L_1$ is rational with respect to
  some linearization $\ti{X}$.  Let $\tau_0$ resp. $\tau_1$ be
  trivializing sections of norm one of $\ti{X} |_{L_0}$ resp.  $\ti{X}
  |_{L_1}$ so that $\tau_0(p) = \tau_1(p)$ for each $p \in L_0 \cap
  L_1$.  Let $\sigma_{0,k},\sigma_{1,k}: X \to \ti{X}^{\otimes k}$ be
  a sequence of sections constructed from $\tau_0,\tau_1$ that are
  concentrated along $L_0,L_1$ as in \cite[p.7
  Remark]{auroux:complement}.  The sum $\sigma_{0,k} + \sigma_{1,k}: X
  \to \ti{X}$ is non-zero at each point in $L_0 \cup L_1$ for $k$
  sufficiently large, since for example on $L_0$ it is the sum of a
  section of norm one and a second of norm less than one.

  The remainder of the argument is the same as in the case of a single
  Lagrangian: A small perturbation of $\sigma_{0,k} + \sigma_{1,k}$ as
  in \cite{auroux:complement} then defines a divisor $D =
  (\sigma_{0,k} + \sigma_{1,k})^{-1}(0)$ disjoint from $L_0 \cup L_1$.
  The restriction of the perturbed section to $D$ may be assumed to
  differ from the trivializations $\tau_0,\tau_1$ by phase at most
  $\pi/4$.  The union $L_0 \cup L_1$ is exact in the complement by the
  previous argument around Equation \eqref{infgauge}.

For part (b), let $\tau_{k,0}$ resp. $\tau_{k,1}$ be trivializing
sections of $\ti{X}^{t_0 k_m k} |_{L_0}$ resp.
$\ti{X}^{\otimes t_0 k_m k}|_{L_1}$ with uniformly bounded
derivatives. Let $\{p_i\}_{1\leq I}$ be the (finite) set of
intersection points of $L_{0}$ and $L_{1}$.  In general,
$\tau_{k,0}(p_i) \neq \tau_{k,1}(p_i)$ for some subset of the indices
in $I$.  We would like to continuously deform one of those
trivializing sections, say $\tau_{k,1}$, to obtain trivializing
sections $\tau'_{k,0}=\tau_{k,0}$ and $\tau'_{k,1}$ in the same
homotopy classes (of non-vanishing sections) as $\tau_{k,0}$ and
$\tau_{k,1}$, respectively, such that
\begin{itemize}
\item 
{\rm (Matching condition)}  $\tau'_{k,0}(p_i) = \tau'_{k,1}(p_i)$ for all indices $i \in I$, and
\item {\rm ($C^2$-bound)}  $|\tau'_{k,0}(p_i)| + |\nabla \tau'_{k,0}(p_i)|_g + 
|\nabla \nabla \tau'_{k,0}(p_i)|_g < C$
and $|\tau'_{k,1}(p_i)|_g + |\nabla \tau'_{k,1}(p_i)|_g + 
|\nabla \nabla \tau'_{k,1}(p_i)|_g < C$
for some constant $C$ (independent of $k$).
\end{itemize}
\noindent To construct the sections $\tau'_{k,1}$, first define the numbers
$$\theta_{k,i} := \arg(\tau_{k,0}(p_i)) - \arg(\tau_{k,1}(p_i)) \in ]-\pi,\pi] .$$
Choose smooth functions 
$\theta_k: L_1 \to ]-\pi,\pi]$ with $\theta_k(p_i)=\theta_{k,i}, \forall i \in I $ 
satisfying the uniform bounds
$|\theta_k| + |\nabla \theta_k|_g + |\nabla \nabla \theta_k|_g <
C_\theta$.
Such functions exist since $|I|<\infty$, $L_1$ and $[-\pi,\pi]$ are
both compact and $g$ does not depend on $k$.  Define
$\tau'_{k,1}:=e^{2\pi i \theta_k} \cdot \tau_{k,1} .$
By definition, 
$$\arg(\tau'_{k,0}(p_i)) - \arg(\tau'_{k,1}(p_i)) =
\arg(\tau_{k,0}(p_i)) - \arg(\tau_{k,1}(p_i)) - \theta_{k,i} = 0, \quad 1\leq i \leq I $$ 
so the (Matching Condition) holds.  Furthermore, since there are
uniform bounds on the derivatives of both $\tau_{k,1}$ and $\theta_k$,
one can find a constant $C_{L_1}$ such that $|\tau'_{k,1}(p_i)| +
|\nabla \tau'_{k,1}(p_i)|_g + |\nabla \nabla \tau'_{k,1}(p_i)|_g <
C_{L_1}$ so the ($C^2$-bound) holds as well.

One can therefore take $\sigma_{0,k},\sigma_{1,k}: X \to 
\ti{X}^{\otimes t_0 k_m k}$ a sequence of sections constructed from 
$\tau_0,\tau_1$ that are concentrated along $L_0,L_1$ as in \cite[p.7
Remark]{auroux:complement}. The sum $\sigma_{0,k} + \sigma_{1,k}$ is 
bounded from below at each point in $L_0 \cup L_1$ and can be
used as an asymptotically holomorphic family of sections concentrated
along $L_0 \cup L_1$. The fact that the resulting divisor 
$D = s_k^{-1}(0) \subset X - (L_0 \cup L_1)$ is weakly stabilizing 
for a general pair $L_0 \cup L_1$
follows 
as in the proof of Lemma \ref{rati_divi2}. 

(c) The uniqueness statement is based on results of Auroux
\cite{auroux:asym} and Auroux-Gayet-Mohsen \cite{auroux:complement}
and Lemmas of the previous section: Any family of approximately
$J_{k,t}$-holomorphic sections $s_{k,t}$ built from homotopic unitary
sections $\tau_{k,t}$, $k>k_{0,t}$, $t \in \{ 0, 1\}$ can be modified,
for large enough $k$, to a family $s_{k,t}$ so that the zero sets
$D_{t} = s_{k,t}^{-1}(0)$ are approximately holomorphic and related
by a symplectic isotopy preserving $L$. In \cite{auroux:complement},
the case where the family $J_{k,t}$ is constant is considered.

If the almost complex structure is time-dependent, then the argument
of \cite{auroux:complement} can be modified as follows: The
almost-complex structure $J_{t}$ provides a way to identify a
Weinstein neighborhood $V_{t}$ of $L$, i.e. $V_{t}$ is
symplectomorphic to a neighborhood of the $0$-section in $T^*L$.
Choosing a trivialization of
$\tilde{X}|_L^{\otimes t_0 k_m} \rightarrow L$ and extending it
properly to $V_{t}$, let $p$ denote the local coordinate on the fibers
of $T^* L$.  Let $g_t$ denote the metric on the fibers of $T^*L$
induced through the identification by the metric $g_t|_L$ associated
with $\omega$ and $J_t$.  The norm $|p|_{g_t}$ is a globally defined
function on $T^*L$.  Define
$\sigma_{k,L,t} = e^{-\frac{1}{2} t_0 k_m k |p|^2_{g_t}}$.
Multiplying $\sigma_{k,L,t}$ by an appropriate
cutoff function, one verifies that the resulting sections $s_{k,L,t}$
are asymptotically $J_t$-holomorphic sections concentrated over $L$,
$t\in [0,1]$, from which the asymptotically $J_t$-holomorphic
$s_{k,t}$ can be built, $t\in \{0,1\}$.

One can then use a one-parameter perturbation argument from
\cite{auroux:asym} to obtain a family of divisors.  Consider sections
$s'_{k,t}$ that are equal to
$(1-3t)s_{k,0} + 3t s_{k,L,0}$ 
for $t \in [0,\frac{1}{3}]$, to 
$s_{k,L,3t-1}$ 
for $t \in [\frac{1}{3},\frac{2}{3}]$ and to
$(3-3t) s_{k,L,1} + (3t-2) s_{k,1}$ 
for $t \in [\frac{2}{3},1]$. The family $s_{k,t}$ is a one-parameter
family of asymptotically $J_{\min(1,\max(0,3t-1))}$-holomorphic
sections non-vanishing on over $L$.  One can then invoke \cite[Theorem
2]{auroux:asym} to get perturbed sections $s_{k,t}$ that are
transverse to the $0$-section, asymptotically $J_t$-holomorphic
sections and non-vanishing over $L$. Note that this requires raising
the degree of $D_{0}$ and $D_{1}$. The fact that the corresponding
isotopy between $(s_{k,0})^{-1}(0)$ (isotopic to $D_{0}$) and
$(s_{k,1})^{-1}(0)$ (isotopic to $D_{1}$) can be realized through a
symplectic isotopy preserving $L$ then follows from the argument of
\cite[Section 4]{auroux:asym}.  Regarding the compatibility with $L$
condition, note that the cohomology class $[\alpha_{\tau_k}]$ only
depends on the homotopy class of $\tau_k$.  In the case of a pair
$L_0,L_1$, the above argument is repeated with
$s_{k,L_0,t}+ s_{k,L_1,t}$ replacing $s_{k,L,t}$.
\end{proof} 

\subsection{Coherent perturbations} 

In this section we use the stabilizing divisors in the previous
section to allow the almost complex structures to vary over the
domains.  Existence of the gluing maps will require the
domain-dependent almost complex structures to satisfy coherence
conditions related to the Behrend-Manin morphisms introduced in
Section \ref{bm}.

We begin by fixing an almost complex structure that we wish to
perturb, and this almost complex structure should make the Donaldson
hypersurface almost complex so that the intersection multiplicities
with holomorphic curves are positive.  Given a divisor $D$, we denote
by $\J_{\tau}(X,D)$ the space of $\omega$-tamed almost complex
structures $J_D$ for which $D$ is an almost complex submanifold. If
$J \in \J_{\tau}(X,\omega)$ and $\theta > 0$, let
$\J_{\tau}(X,D,J,\theta)$ be the subset of elements $\theta$-close to
$J$ in $C^0$ norm.  The space $\J_{\tau}(X,D,J,\theta)$ is non-empty
by \cite[Section 8]{cm:trans}.

Domain-dependent almost complex structures are almost complex
structures depending on a point in the universal strip, and equal to
the fixed almost complex structure at the boundary and nodes.  If
$\Gamma$ is a type of stable strip with a single vertex, then
$\U_\Gamma$ is smooth and it makes sense to talk about class $C^l$
maps for some integer $l \ge 0$ from $\U_\Gamma$ to a target manifold.
More generally, if $\Gamma$ has interior edges then let $\Gamma'$ be
the graph obtained by cutting all edges.  By definition a map from
$\U_{\Gamma}$ of class $C^l$ is a map from $\U_{\Gamma'}$ of class
$C^l$ satisfying the matching condition at the markings identified
under the map $\U_{\Gamma'} \to \U_\Gamma$.  For any type $\Gamma$,
the compactified universal strip $\ol{\U}_\Gamma$ is a manifold with
corners away from the boundary nodes.  Given an almost complex
structure $J_D \in \J_\tau(X,D)$ we say that a map from
$\ol{\U}_\Gamma$ to $\J_\tau(X,D)$ agreeing with $J_D$ near the nodes
is of class $C^l$ if it is $C^l$ away from the boundary nodes.

\begin{definition} \label{pertdef} Let $J_D \in \J(X,D)$ and $l \ge 0$.  
A {\em perturbation datum} (resp. perturbation datum of class $C^l$)
adapted to $D$ for a type $\Gamma$ is a smooth (resp. $C^l$) map
$$ J_\Gamma: \ol{\U}_{\Gamma} \to \J_{\tau}(X,\omega) $$
such that 
\begin{enumerate}
\item {\rm (Compatible with the divisor)} $J_\Gamma(z)\in \J_\tau(X,D)$
  and $J_\Gamma(z)|_D \equiv J_D$ for all $z \in \ol{\U}_{\Gamma}$.
\item {\rm (Constant near the nodes and markings)} The restriction of
  $J_\Gamma$ to a neighborhood of any node or boundary marking is
  equal to $J_D$.

\item {\rm (Constant at the boundary)} The restriction of $J_\Gamma$
  to the boundary of the (nodal) strips is equal to $J_D$.
\end{enumerate} 
\end{definition}

The following are three operations on perturbation data:

\begin{definition} 
\begin{enumerate} 
\item {\rm (Cutting edges)} Suppose that $\Gamma$ is a combinatorial
  type and $\Gamma'$ is obtained from $\Gamma$ by cutting edges
  corresponding to nodes connecting strip-like components.  A
  perturbation datum for $\Gamma'$ gives rise to a perturbation datum
  for $\Gamma$ by pushing forward $J_{\Gamma'}$ under the map
  $\ol{\U}_{\Gamma'} \to \ol{\U}_\Gamma$, which is well-defined by the
  (Constant near the nodes and markings) axiom.
\item {\rm (Collapsing edges/making an edge finite or non-zero)}
  Suppose that $\Gamma'$ is obtained from $\Gamma$ by collapsing an
  edge or making an edge finite or non-zero.  Any perturbation datum
  $J_{\Gamma'}$ for $\Gamma'$ induces a datum for $\Gamma$ by pullback
  of $J_{\Gamma'}$ under $\ol{\U}_{\Gamma} \to \ol{\U}_{\Gamma'}$.
\item {\rm (Forgetting tails)} Suppose that $\Gamma'$ is a
  combinatorial type of stable strip is obtained from $\Gamma$ by
  forgetting an marking.  In this case there is a map of universal
  disks $\ol{\U}_{\Gamma} \to \ol{\U}_{\Gamma'}$ given by forgetting
  the marking and stabilizing.  Any perturbation datum $J_{\Gamma'}$
  induces a datum $J_\Gamma$ by pullback of $J_{\Gamma'}$.
\end{enumerate}
\end{definition} 

We are now ready to define coherent collections of perturbation data.
These are data which behave well with each type of operation above.

\begin{definition} \label{coherent} {\rm (Coherent families of perturbation data)} A
  collection of perturbation data $ \ul{J} = (J_\Gamma ) $ is
  {\em coherent} if it is compatible with the Behrend-Manin morphisms
  in the sense that
\begin{enumerate} 
\item {\rm (Cutting edges axiom) } if $\Gamma'$ is obtained from
  $\Gamma$ by cutting an edge corresponding to a strip-like edge, then
  $J_\Gamma$ is the pushforward of $J_{\Gamma'}$;
\item {\rm (Collapsing edges/make an edge finite or non-zero axiom)}
  if $\Gamma'$ is obtained from $\Gamma$ by collapsing an edge, then
  $J_{\Gamma}$ is the pullback of $J_{\Gamma'}$;
\item {\rm (Product axiom)} if $\Gamma$ is the union of types 
  $\Gamma_1,\Gamma_2$ then $J_\Gamma$ is obtained from $J_{\Gamma_1}$
  and $J_{\Gamma_2}$ as follows: Let $\pi_k: \ol{\M}_\Gamma \cong
  \ol{\M}_{\Gamma_1} \times \ol{\M}_{\Gamma_2} \to \ol{\M}_{\Gamma_k}$
  denote the projection on the $k$-factor, so that $\ol{\U}_\Gamma$ is
  the union of $\pi_1^* \ol{\U}_{\Gamma_1}$ and $\pi_2^*
  \ol{\U}_{\Gamma_2}$.  Then we require that $J_\Gamma$ is equal to
  the pullback of $J_{\Gamma_k}$ on $\pi_k^* \ol{\U}_{\Gamma_k}$.
\end{enumerate} 
\end{definition}  

\subsection{Perturbed Floer trajectories} \label{sect_adap}

We now define Floer trajectories satisfying a given perturbation
datum.  If a treed strip is {\em stable} there is a {\em unique}
identification of the curve with a fiber of the universal strip.  More
generally, if $C$ is a possible {\em unstable} treed strip of type
$\Gamma$ with at least one interior marking then one obtains a map
$\pi_C: C \to \ol{\U}_\Gamma$ by identifying the stabilization with a
fiber.  In particular, if $J_\Gamma$ is a perturbation datum then
$\pi_C^* J_\Gamma: C \to \J_\tau(X,\omega)$ is a domain-dependent
almost complex structure taming the symplectic form.

\begin{definition} {\rm (Perturbed Floer trajectories)} A perturbed
  stable Floer trajectory of type $\Gamma$ for a perturbation datum
  $J_\Gamma$ is a treed strip $C$ of type $\Gamma$ together with a map
  $u: C \to X$ which is holomorphic with respect to
  $\pi_C^* J_\Gamma$.
\end{definition} 

The regularity properties of the domain-dependent almost complex
structures are sufficient for the following version of Gromov-Floer
compactness to hold:

\begin{theorem} \label{preconv}  
Suppose that $\Gamma$ is a type of stable treed strip and
$J_{\Gamma,\nu}: \ol{\U}_\Gamma \to \J_{\tau}(X,\omega)$ is a sequence
of domain-dependent almost complex structures of class $C^l, l \ge 2$
converging to a limit $J_\Gamma$ , $C_\nu \subset \ol{\U}_\Gamma$ is a
sequence of stable treed strips of type $\Gamma$, and $u_\nu: C_\nu
\to X$ is a sequence of stable Floer trajectories with bounded energy.
Then after passing to a subsequence, $u_\nu: C_\nu \to X$ converges to
a limiting stable Floer trajectory $u: \hat{C} \to X$.
\end{theorem} 

\begin{proof}[Sketch of proof]
Since $\ol{\M}_\Gamma$ is compact, after passing to a subsequence we
may assume that $[C_\nu]$ converges to a limit $[C] \in
\ol{\M}_\Gamma$.  We may view the development of nodes as stretching
of the neck.  On each compact subset of $C$ disjoint from the nodes,
the almost complex structure $J_\nu |_{C_{\Gamma,\nu}}$ converges to
$J_\Gamma |_{C}$ uniformly in all derivatives.  On the other hand, on
the neck regions $J_\nu$ is equal to $J_D$, by the (Constant near the
nodes and markings) axiom.  Hence the standard compactness arguments
(finitely many bubbles, soft rescaling, bubbles connect) in
e.g. \cite{zilt:dipl} allow the construction of a limiting map
$u: \hat{C} \to X$, where $\hat{C}$ is obtained from $C$ by adding a
finite collection of bubble trees of sphere and disk components.
\end{proof} 

It follows from the Theorem that the spaces of perturbed stable Floer
trajectories in general admits a compactification involving domains
that may be unstable.  Since our perturbation maps have constant
values over the unstable components, we are not in a position to
achieve transversality over the compactification unless we can avoid
such instability of the limit domains. The choice of (either strictly
or weakly) stabilizing divisors containing no non-constant holomorphic
spheres and the restriction to ``adapted'' nodal Floer trajectories on
which the interior markings keep track of the intersections with the
stabilizing divisors will ensure that domain stability is preserved
under taking limits. Their precise definition is the following:

\begin{definition} {\rm (Adapted Floer trajectories)}   
Let $D$ be a divisor and $J_\Gamma$ a perturbation datum adapted to
$D$.  Let $u: C \to X$ be a nodal Floer trajectory which is $\pi_C^*
J_\Gamma$-holomorphic.  The trajectory $u: C \to X$ is {\em adapted}
to $D$, or $D$-adapted, if
\begin{enumerate} 
\item[] {\rm (Stable domain property)} $C$ is a stable marked disk; and
%
%
\vskip .1in
\item[] {\rm (Marking property)} Each interior marking lies in
  $u^{-1}(D)$ and each component of $u^{-1}(D)$ contains an interior
  marking.
\end{enumerate} 
\end{definition} 

We introduce the following moduli spaces.  Let $\ol{\M}_{n}(\ul{L},D)$
be the set of isomorphism classes of stable $D$-adapted Floer
trajectories to $X$ with $n$ interior markings.  The space
$\ol{\M}_{n}(\ul{L},D)$ admits a topology defined by Gromov-Floer
convergence, whose definition is similar to that for pseudoholomorphic
curves in McDuff-Salamon \cite[Section 5.6]{ms:jh}.  Despite the
notation, it is compact only for generic choices of domain-dependent
almost complex structure, see Theorem \ref{compthm} below.  Because
the constraint that the marking map to a divisor is codimension two,
the formal dimensions of the moduli spaces are independent of the
number of markings.  We define a stratification of the moduli space of
adapted trajectories as follows.  The vertices of the graph are
labelled as follows.  Denote by
\begin{enumerate} 
\item $\Pi(X)$ the space of homotopy classes of maps from the
  two-sphere $S^2$ to $X$;
\item $\Pi(X,L_b)$ the space of homotopy classes
  of maps from the disk $(B^2,S^1)$ to $(X,L_b)$, $b \in \{ 0,1 \}$;
\item $\Pi(X,L_0,L_1)$ the space of homotopy classes of maps from the
  square $[-1,1] \times [0,1]$ mapping $[-1,1] \times \{ b \} $ to
  $L_b$ for $b \in \{ \ 0, 1 \}$ and $\{ \pm 1 \} \times [0,1]$ to
  $x_\pm \in \cI(L_0,L_1)$.
\end{enumerate} 
Each stable trajectory $u: C \to X$ gives rise to a labelling
$$\Ver(\Gamma) \to \Pi(X,L_0,L_1) \sqcup \Pi(X,L_0) \sqcup \Pi(X,L_1)
\sqcup \Pi(X) $$ 
giving the homotopy class of $u$ restricted to the corresponding
component of $C$.  Let $\Edge_c(\Gamma) \subset \Edge(\Gamma)$ denote
the set of {\em contact edges} corresponding to markings or nodes that
map to the divisor $D$.  Given an edge $e \in \Edge_c(\Gamma)$
corresponding to a marking $w(e)$ let $\mu(e)$ denote the intersection
multiplicity of $u$ at $w(e)$, that is, the order of tangency.  For an
edge $e \in \Edge_c(\Gamma)$ corresponding to a node $w(e)$, let
$\mu(e) = (\mu_+(e), \mu_-(e))$ denote the intersection degrees on
either side of the node.  The {\em combinatorial type} of a stable
adapted Floer trajectory $u: C \to X$ is the combinatorial type
$\Gamma$ of the domain curve $C$ equipped with a labelling of
$\Ver(\Gamma)$ by homotopy classes
$\Pi(X,L_0,L_1) \sqcup \Pi(X,L_0) \sqcup \Pi(X,L_1) \sqcup \Pi(X) $,
and a labelling of contact edges $\Edge_c(\Gamma)$ by the intersection
degrees $\mu(e)$.

\begin{definition}  
For any connected type $\Gamma$ we
denote by $\M_\Gamma(\ul{L},D)$ the space of adapted stable Floer
trajectories of type $\Gamma$.  If $\Gamma$ is disconnected then
${\M}_\Gamma(\ul{L},D)$ is the product of the moduli spaces for the
connected components of $\Gamma$.
\end{definition}

The Behrend-Manin morphisms in Definition \ref{bmgraphs} naturally
extend to graphs labelled by homotopy classes.  For example, to say
that $\Gamma$ is obtained from $\Gamma'$ by (Collapsing an edge) means
that if $v',v''$ are identified then the homotopy classes of the
components satisfy $d(v) = d(v')+d(v'')$ using the natural additive
structure on homotopy classes of maps, and the labelling of edges of
$\Gamma'$ by intersection multiplicities, intersection degree and
tangencies is induced from that of $\Gamma$.  There is a partial order
on combinatorial types of stable trajectories generated by relations
$\Gamma < \Gamma'$ if $\Gamma$ is obtained from $\Gamma'$ by
collapsing edges or making edges finite or non-zero.  Whenever
$\Gamma < \Gamma'$ we have an inclusion of strata
$\M_\Gamma(\ul{L},D) \subset \ol{\M}_{\Gamma'}(\ul{L},D)$.

The following is an immediate consequence of the definition of
coherent family of domain-dependent almost complex structures in
Definition \ref{coherent}:

\begin{proposition} \label{bmmaps} {\rm (Behrend-Manin maps for Floer trajectories)} 
Suppose that $\ul{J} = (J_\Gamma)$ is a coherent family of
perturbation data.  Then
\begin{enumerate} 
\item {\rm (Cutting edges)} If $\Gamma'$ is obtained from $\Gamma$ by
  cutting a edge connecting strip components, then there is an embedding
  ${\M}_{\Gamma}(\ul{L},D) \to {\M}_{\Gamma'}(\ul{L},D)$
  whose image is the space of stable marked trajectories of type
  $\Gamma'$ whose values at the markings corresponding to the cut
  edges agree.
\item {\rm (Collapsing edges/making edges finite or non-zero)} If
  $\Gamma'$ is obtained from $\Gamma$ by collapsing an edge or making
  an edge finite or non-zero, then there is an embedding of moduli
  spaces ${\M}_{\Gamma}(\ul{L},D) \to \ol{\M}_{\Gamma'}(\ul{L},D)$.
\item[(d)] {\rm (Products)} if $\Gamma$ is the disjoint union of
  $\Gamma_1$ and $\Gamma_2$ then $\M_\Gamma(\ul{L},D)$ is the
  product of $\M_{\Gamma_1}(\ul{L},D)$ and
  $\M_{\Gamma_2}(\ul{L},D)$.
\end{enumerate} 
\end{proposition} 

For the (Cutting edges) morphism, if $\Gamma'$ is obtained by cutting
a strip-connecting edge, then
$\ol{\M}_{\Gamma'}(\ul{L},D) \cong \ol{\M}_{\Gamma_1}(\ul{L},D) \times
\ol{\M}_{\Gamma_2}(\ul{L},D)$.  We denote by
\begin{equation} \label{concat}  \sharp: \ol{\M}_{\Gamma_1}(\ul{L},D) \times_{\cI(L_0,L_1)} \ol{\M}_{\Gamma_2}(\ul{L},D) \to
\ol{\M}_\Gamma(\ul{L},D) \end{equation}
the map obtained by combining this isomorphism with the Behrend-Manin
map and call it {\em concatenation of Floer trajectories}.  

We introduce the following notations for numerical invariants
associated to a combinatorial type.  The {\em index} $i(\Gamma)$ of a
type $\Gamma$ is the Fredholm index of the operators associated with
trajectories of ${\M}_{\Gamma}(\ul{L},D)$ given by linearizing
\eqref{here0}; this is determined by the homotopy classes of the maps
which are part of the definition of the combinatorial type $\Gamma$.
The strata contributing to the definition of the Floer coboundary
operator are those with index one (so that the expected dimension of
the top stratum is zero) while the proof that the coboundary squares
to zero involves those of index two.

We discuss several variations on the definition.  

\begin{remark} {\rm (Boundary divisors)} \label{marginal} In the first
  variation, we use different divisors to stabilize the disk bubbles.
  Recall from \cite[Lemma 8.3]{cm:trans} that for a constant
  $\eps > 0$, two divisors $D,D'$ intersect {\em $\eps$-transversally}
  if at each intersection point $x \in D \cap D'$ their tangent spaces
  $T_x D, T_x D'$ intersect with angle at least $\eps$.
  Cieliebak-Mohnke \cite[Theorem 8.1]{cm:trans} shows that there
  exists an $\epsilon>0$ such that a pair of divisors of sufficiently
  high degree constructed in the last section can be made
  $\epsilon$-transverse. Moreover, for any $\theta>0$,
  $\omega$-compatible almost complex structures $\theta$-close to $J$
  making a pair of $\epsilon$-transverse divisors almost complex exist
  (provided that the degrees are sufficiently large).

  Now suppose that $D_b$ is stabilizing for $L_b$,
  $b\in \{ 0, 1\}$, stabilizing for $L_0\cup L_1$ and assume that
  $D_b$ is $\epsilon$-transverse to $D$, $b\in \{ 0, 1\}$.  Assume
  that there exists a compatible almost-complex structure
  $J \in \J_\tau(X,\omega)$ preserving the tangent spaces to $D$ and
  $D_b$, $b\in \{ 0, 1\}$.  For $b \in \{ 0, 1 \}$ let
$$ C_b = \overline{f^{-1}(b)\backslash (\partial C)_b}
\subset C$$ 
be the union of disk and sphere components of $ f^{-1}(b)$ and for
$t\in [0,\infty]$.  Let $\ell_b^{-1}(t) \subset C_b$ be made of the
disk and sphere components of $f^{-1}(b)$ connected to a strip
component through a sequence of edges of total length equal to
$t$. Let $z_i, i =1,\ldots, n$ be the same markings as before, choose
extra (interior) markings $\ z_i^b, i = 1,\ldots, n_b$ on $C_b$,
$b \in \{ 0, 1 \}$, and set $\ul{n}=(n,n_0,n_1)$.  Choose $(J_\Gamma)$
a coherent family of perturbation data that is, as in the previous
section, depending on the position relative to the $z_i$ markings on
the strata with $C_0 \cup C_1 =\emptyset$.  Suppose that $J_\Gamma$ is
equal to $J$ near the nodes, boundary markings and boundary
components, and suppose that its restriction to $D$ is equal to
$J_D|_D$.  We may assume that $(J_\Gamma)$ depends only on the $z_i^b$
markings over the components in $\ell_b^{-1}(\infty)$,
$b \in \{ 0, 1 \}$, that is, so that $J_\Gamma$ is compatible with the
forgetful morphism on $\ell_b^{-1}(\infty)$ forgetting all but the
markings $z_i^b$.  The coherence condition on $(J_\Gamma)$ implies
that each $J_\Gamma$
is given as a product on the product strata associated with edges of
infinite length, and the latter condition means that the perturbation
on the components of $\ell_b^{-1}(\infty)$ will now be independent of
the $z_i$ markings, depending on the $z_i^b$ markings instead. On
components with $\ell_b^{-1}(]0,\infty[)\neq \emptyset$, the
    perturbation is allowed to depend on both the $z_i$ markings and
    the $z_i^b$ markings.

    Let $u: C \to X$ be a stable $\ul{n}$-marked $J$-holomorphic Floer
    trajectory $u: C \to X$ for $\ul{L} = (L_0,L_1)$ with markings
    $z_i, i =1,\ldots, n , \ z_i^b, i = 1,\ldots, n_b$.  The
    trajectory $u: C \to X$ is adapted to $\ul{D} = (D,D_0,D_1)$ if
 in 
    addition to the conditions above describing the intersections with 
    the divisor $D$, for the divisors $D_b, b \in \{ 0,1 \}$ the
    following holds:
\begin{enumerate} 
%
\item[] {\rm (Marking property)} For each $b \in \{ 0,1 \}$, each
  interior marking $z_i^b$ lies in $u^{-1}(D_b) \cap C_b$ and each
  component of $u^{-1}(D_b) \cap C_b$ contains a marking $z_i^b$.
\end{enumerate}

Note that the intersection loci $u^{-1}(D_b)$ in $C\backslash C_b$ are
not required to contain markings, $b \in \{ 0,1 \}$.  On the disks and
spheres at distance $0$ from a strip component, the location of these
intersections do not influence the perturbed Floer equation, but as
their distance goes to infinity, they become the only intersections
influencing the perturbation on their components. This choice of
perturbation on disk bubbles is very similar to the one used to prove
invariance in \cite{cm:trans}.

The combinatorial type $\Gamma$ now includes for each
$b \in \{ 0, 1 \}$ a subset
$\Edge_{\infty,b}(\Gamma) \subset \Edge_{\infty}(\Gamma)$ of
semi-infinite edges (on $C_b$) corresponding to markings that are
required to map to $D_b$.  If $\Gamma$ has edges of length $0$, take
$\Gamma'$ to be the type obtained by forgetting the edges of
$\Edge_{\infty,b}(\Gamma)$ that are on components at distance $0$ from
a strip component. One gets a morphism of combinatorial types, a map
$\ol{\M}_{\Gamma} \to \ol{\M}_{\Gamma'}$ that lifts to a map of Floer
trajectories
$\ol{\M}_{\Gamma}(\ul{L},\ul{D}) \to
\ol{\M}_{\Gamma'}(\ul{L},\ul{D})$.
Two stable Floer trajectories are {\em equivalent} if they are
isomorphic or they are related by such a forgetful morphism.

We denote by $\ol{\M}(\ul{L},\ul{D})$ the moduli space of {\em
  equivalence} classes of adapted stable Floer trajectories,
\begin{equation} \label{glue}
 \ol{\M}(\ul{L},\ul{D}) = \bigcup_\Gamma \ol{\M}_\Gamma
 (\ul{L},\ul{D})/ \sim 
\end{equation}
where the union is over connected types $\Gamma$ and $\sim$ is the
equivalence relation defined by the above forgetful maps and the
Behrend-Manin maps of Floer trajectories.  Note that, as opposed to
the Behrend-Manin maps of Floer trajectories, the forgetful morphisms
are in general many-to-one.  Indeed, the labels of the forgotten
markings of $\Edge_{\infty,b}(\Gamma)$ can be permuted without
affecting the Floer equation. However, counting $n$-marked Floer
trajectories with a weight $1/n!$, as our perturbation setting
suggests, could ensure that forgetful equivalences still lead to
well-defined rational fundamental classes.
\end{remark} 

\begin{remark}\label{families}  {\rm (Families of divisors)} In this remark let
  $\ul{J}=(J_D,J_{D_0},J_{D_1})$ be a triple of almost complex
  structures where $J_D$ (resp. $J_{D_b}$) preserves $D$ (resp. $D_b$,
  $b\in \{0,1\}$).  Assuming that the divisors $D,D_b$ have the same
  degree and were built from homotopic sections $\tau_{k}|_{L_b}$ and
  $\tau_{k_b}$, $b\in \{0,1\}$, Lemma \ref{auroux} \eqref{uniq} gives
  a family of divisors $D_{t,b}$ in the complement of $L_b$ that are
  $J_{t,b}$-holomorphic for some homotopy of almost-complex structures
  $J_{t,b}$, $t\in [0,1]$, from $J_{0,b} = J_{D_b}$ to
  $J_{1,b} = J_{D}$, $b\in \{0,1\}$.
  Let $(J_\Gamma)$ be a coherent family of perturbation data equal to
  $J_{t,b}$ near the nodes, markings and boundary components over the
  components of $\ell_b^{-1}(-\on{log}(t))$ and such that
  $J_\Gamma|_{D_{t,b}} = J_{t,b}$ over $\ell_b^{-1}(-\on{log}(t))$.

Given such perturbations, adapted trajectories are defined as follows.
Let $u: C \to X$ be a stable $\ul{n}$-marked $(J_\Gamma)$-holomorphic
Floer trajectory $u: C \to X$ for $\ul{L} = (L_0,L_1)$ with markings
$z_i, i =1,\ldots, n , \ z_i^b, i = 1,\ldots, n_b$.  The trajectory
$u: C \to X$ is adapted to $\ul{D} = (D,D_0,D_1)$ if in addition to
the conditions above describing the intersections with the divisor
$D$, for the divisors $D_b, b \in \{ 0,1 \}$ we have
\begin{enumerate} 
%
\item[] {\rm (Marking property)} For each $b \in \{ 0,1 \}$, each
  interior marking $z_i \in \ell_b^{-1}(-\on{log}(t))$ lies in
  $u^{-1}(D_{t,b}) \cap C_b$ and each component of $u^{-1}(D_{t,b})
  \cap C_b$ contains a marking $z_i \in \ell_b^{-1}(-\on{log}(t))$.
\end{enumerate}
\end{remark}

\begin{remark} \label{anti} {\rm (Anti-symplectic involutions)}
  Suppose now that $\iota_b: X \to X$ is an anti-symplectic
  involution, i.e.  $\iota_b^2= \on{Id}$ and
  $\iota_b^*\omega = -\omega$, with fixed locus $L_b$,
  $b \in \{0,1\}$. We may then assume that the anti-symplectic
  involution $\iota_b$ preserves the divisor $D_b$ in the sense that
  $\iota_b(D_b) = b$ for $b \in \{ 0, 1 \}$.  Indeed if $L_b$ is the
  fixed point set of an anti-symplectic involution $\iota_b$ which
  lifts to $\ti{X}$, a result of Gayet \cite[Proposition 7]{gayet:re}
  implies that the divisor $D_b$ may be chosen stable under
  $\iota_b, b \in \{ 0, 1 \}$.  

Perturbations compatible with anti-symplectic involutions may be
chosen as follows.  Let $\Gamma$ be a type corresponding to a (nodal)
disk component in $\ell_b^{-1}(\infty) \subset C_b$.  Reversing the
complex structure on the disk at infinity defines an involution $*:
\ol{\M}_\Gamma \to \ol{\M}_\Gamma$, which lifts to an involution on
the universal moduli space $*: \ol{\U}_\Gamma \to \ol{\U}_\Gamma$.
Identifying the disk components with the complex unit disk with the
boundary marking identified to $-1$, the fixed points of $*$ on
$\M_\Gamma$ are the configurations where the special points on the
disk component all lie on the real locus of the disk, while the fixed
points of $*$ on $\U_\Gamma$ consist of the real locus over any fixed
disk.  In addition to choosing perturbations as in (Families of
divisors) (or (Boundary divisors)), one may then choose the
perturbation datum so that the involution induces an involution on the
moduli space of trajectories with disk bubbles at infinity.  Given a
stratum $\M_\Gamma(\ul{L},\ul{D})$ corresponding to a single
disk component in $\ell_b^{-1}(\infty)$, we require that $J_\Gamma$ is
anti-invariant under the symplectic involution in the sense that
\begin{enumerate}
\item $J_\Gamma \circ * = - \iota_b^* J_\Gamma$ on the 
disk 
component
of $\ell_b^{-1}(\infty)$, and 
\item $J_\Gamma \circ * = J_\Gamma$ over the other components.
\end{enumerate}
Note that if we choose perturbations $J_\Gamma$ in the set
$\J_{\iota_b}(X,\omega)$ of $\iota_b$-anti-invariant almost-complex
structures, that is, almost-complex structures
$J \in \J_\tau(X,\omega)$ such that $\iota_b^*J = -J$, one obtains
that $ J_\Gamma \circ * = -\iota_b^* J_\Gamma = J_\Gamma$.
\end{remark} 

\begin{remark} \label{spheres} {\rm (Perturbations for the diagonal)}  
In the case that one of the Lagrangians is a diagonal an alternative
regularization scheme uses the identification of disks with Lagrangian
boundary with spheres.  That is, suppose that
$$m \in \{ 0, 1 \}, \quad  L_m = \Delta \subset X = Y^- \times Y .$$
Choose a Donaldson hypersurface $D_Y \subset Y$ of sufficiently large
degree.  Given a treed strip $C$, a {\em sphere lifting with markings}
consists of, for each disk component $C_i \subset C$ in the strip, a
holomorphic sphere $\ti{C}_i$ equipped with an anti-holomorphic
involution $\iota_i$ so that $\ti{C}_i/\iota_i = C_i$ as surfaces with
boundary and a collection of markings on $\ti{C}_i$.  Let
$\ti{\U}_\Gamma \to \M_\Gamma$ be the universal bundle whose fiber
over $C$ is the curve obtained by replacing each disk at infinity
distance from the strip with the corresponding sphere.  The map
$\ti{\U}_\Gamma \to \U_\Gamma$ taking the quotient of these spheres by
the involution is a covering on each stratum.  A perturbation datum is
then a domain-dependent almost complex structure
$\ti{J}_\Gamma: \ti{\U}_\Gamma \to \J_\tau(X,\omega)$.  Below we will
assume that $(J_\Gamma)$ is a collection of perturbations satisfying
the conditions in Definition \ref{coherent}, and if in addition for
any component at infinite distance from the strip components, the
perturbations depend only on the spherical markings.

For later use we note that there is also a map of moduli spaces
obtained by replacing each disk at infinity with the sphere obtained
by gluing together two copies of the disk.  Let $\ti{\Gamma}$ denote
the marked sphere obtained in this way and consider corresponding map
of moduli spaces $\M_{\Gamma} \to \M_{\ti{\Gamma}}$.  This map has
positive-dimensional fiber whenever a disk at infinity has at most two
boundary markings, since the difference between dimensions of
automorphisms of the sphere and disk is three.
\end{remark}

\section{Transversality and compactness} 

\subsection{Transversality} 

In this section we achieve transversality of strata of index one or
two by a choice of domain-dependent almost complex structure.  Recall
that a {\em comeager} subset of a topological space is a countable
intersection of open dense sets \cite{royden}, while a {\em Baire
  space} is a space with the property that any comeager subset is
dense.  Any complete metric space is a Baire space.  In particular,
the space of almost complex structures of any class $C^l$ or
$C^\infty$ is a Baire space, since each admits a (non-linear) complete
metric.  We construct comeager sets of perturbation data by induction
on the type of stable trajectory making the strata of index one or two
smooth of expected dimension.  We assume for simplicity of notation
that the Hamiltonian perturbation vanishes; the same discussion goes
through with Hamiltonian perturbations that are supported away from
the divisor.

In preparation for the transversality argument we define a space of
perturbations as follows.  We fix an open neighborhood of the nodes
and markings $\ol{\U}_{\Gamma}^{\thin} \subset \ol{\U}_{\Gamma} $, on
which the perturbation is assumed to vanish.  Suppose that
perturbation data $J_{\Gamma'}$ for all boundary types
$\U_{\Gamma'} \subset \ol{\U}_\Gamma$ have been chosen.  Let
$\cJ_\Gamma^l(X,D)$ denote the space of domain-dependent almost
complex structures on $X$ of class $C^l$ with $l \ge 2$ for a type
$\Gamma$ such that the following conditions hold:
\begin{enumerate} 
\item The restriction of $J_{\Gamma}$ to $\ol{\U}_{\Gamma'}$ is equal
  to $J_{\Gamma'}$, for each boundary type $\Gamma$', that is, type of
  lower-dimensional stratum $\M_{\Gamma'} \subset \ol{\M}_\Gamma$.
  This condition will guarantee that the resulting collection
  satisfies the (Collapsing edges/Making edges finite or non-zero)
  axiom of the coherence condition Definition \ref{coherent}.
\item The restriction of $J_\Gamma$ to $\ol{\U}_{\Gamma}^{\thin}$ is
  equal to $J_D$.
\end{enumerate} 
For any fiber $C \subset \ol{\U}_\Gamma$ we denote by $C^{\thin}$ the
neighborhood of the nodes and markings of $C$ given by the
intersection of $C$ with $\ol{\U}^{\thin}_\Gamma$, and $C^{\thick}$
the complement of $C^{\thin}$.  Let $\cJ_\Gamma(X,D)$ denote the
intersection of the spaces $\cJ_\Gamma^l(X,D)$ for $l \ge 0$.

We wish to achieve transversality by choosing generic perturbations of
the given almost complex structure.  However, it is not possible to
obtain transversality with expected dimension for {\em all} strata
using domain-dependent almost complex structures.  To explain the
problem, we introduce the following terminology.  A {\em maximal ghost
  component} of a stable marked strip is a union of ghost components
that is connected and maximal among such unions.  Consider the case
that a {\em maximal ghost component contains several markings}.  From
any such stable trajectory we can obtain a stable trajectory with a
single marking on each maximal ghost component, by forgetting all but
one such marking.  The forgetful map has positive dimensional fibers
but the two strata have the same expected dimension.  It follows that
the strata with multiple such markings cannot be made regular by this
method. A type $\Gamma$ will be called {\em uncrowded} if each maximal
ghost component contains at most one marking.  Given any crowded type
an uncrowded type may be obtained by forgetting all but one marking on
each maximal ghost component.

\begin{theorem} \label{main} {\rm (Transversality)} Suppose that
  $\Gamma$ is an uncrowded type of stable trajectory of expected
  dimension at most one.  Suppose that regular coherent perturbation
  data for types of stable trajectories $\Gamma'$ with
  $\Gamma' > \Gamma$ are given.  Then there exists a comeager subset
$ {\cJ}^{\reg}_\Gamma(X,D) \subset {\cJ}_\Gamma(X,D) $
of {\em regular perturbation data} for type $\Gamma$ compatible 
with the previously chosen
perturbation data such that if $J_\Gamma \in \cJ_\Gamma^{\reg}(X,D)$
then
\begin{enumerate} 
\item {\rm (Smoothness of each stratum)} the stratum
  ${\M}_{\Gamma}(\ul{L},D)$ is a smooth manifold of expected
  dimension;
\item {\rm (Tubular neighborhoods)} if $\Gamma$ is obtained from
  $\Gamma'$ by collapsing an edge or making an edge finite or non-zero
  then the stratum ${\M}_{\Gamma'}(\ul{L},D)$ has a tubular
  neighborhood in $\ol{\M}_{\Gamma}(\ul{L},D)$; and 
\item {\rm (Orientations)} there exist orientations on
  $\M_\Gamma(\ul{L},D)$ compatible with the Behrend-Manin morphisms
  {\rm (Cutting an edge)} and {\rm (Collapsing an edge/Making an edge finite or non-zero)} in the
  following sense:
\begin{enumerate} 
\item If $\Gamma$ is obtained from $\Gamma'$ by {\rm (Cutting an edge)}
  then the isomorphism $\M_{\Gamma'}(\ul{L},D) \to
  \M_{\Gamma}(\ul{L},D)$ is orientation preserving.
\item If $\Gamma$ is obtained from $\Gamma'$ by {\rm (Collapsing an
  edge)} or {\rm (Making an edge finite or non-zero)} then the
  inclusion $\M_{\Gamma'}(\ul{L},D) \to \M_{\Gamma}(\ul{L},D)$ has
  orientation (using the decomposition
$$ T\M_{\Gamma}(\ul{L},D) |_{\M_{\Gamma'}(\ul{L},D)} \cong \R \oplus  T\M_{\Gamma'}(\ul{L},D)$$
and the outward normal orientation on the first factor) given by a
universal sign depending only on $\Gamma,\Gamma'$.  (In particular,
this condition implies that contributions from opposite boundary
points of the one-dimensional connected components cancel.)
\end{enumerate} 
In particular, any isomorphism class $[u]$ in any zero-dimensional
stratum 
$\ol{\M}_\Gamma(\ul{L},D)_0$ of 
an index two moduli
$\ol{\M}_\Gamma(\ul{L},D)$
is associated a sign $\eps([u]) = \pm 1$ by comparing its 
orientation with the canonical orientation of a point.  

In the case of {\rm (Collapsing an edge)} the inclusion
$\M_{\Gamma'}(\ul{L},D) \to \ol{\M}_{\Gamma}(\ul{L},D)$ is orientation
preserving resp. reversing if $\Gamma$ corresponds to splitting of
Floer trajectories or breaking off a disk bubble in $L_0$
resp. breaking off a disk bubble in $L_1$.
\end{enumerate} 
\end{theorem} 

\begin{proof} The proof of the first part of the theorem is an
  argument using the Sard-Smale theorem on universal moduli spaces.
  The universal moduli spaces are constructed using the implicit
  function theorem for Banach manifolds.  For $p \ge 2, k \gg 2/p$ as
  in \eqref{mapcharts} let $\Map^{k,p}_\Gamma(C^\times,X,\ul{L})$ be
  the space of maps $u$ from $C^\times$ to $X$ of Sobolev class
  $W^{k,p}$ with type specified by the labellings of the vertices and
  edges of $\Gamma$ by homotopy classes and tangencies, mapping
  $(\partial C)_b$ to $L_b$ for $b \in \{ 0,1 \}$.  The homotopy class
  of the component $C^\times_v$ corresponding to a vertex
  $v \in \Ver(\Gamma)$ is $d(v)$; for each edge
  $e \in \Edge_{\infty}(\Gamma) \backslash \bigcup_{i\in \{0,1\}}
  \Edge_{\infty,b}(\Gamma)$
  (resp. $\Edge_{\infty,b}(\Gamma)$) the intersection degree $u$ with
  $D$ is $\mu(e)$. In order that these derivatives be well-defined we
  require $k \ge \mu(e) + 2/p + 1$ for each edge
  $e \in \Edge_\infty(\Gamma)$, where $\mu(e)$ is the intersection
  multiplicity of the map with the divisor at the corresponding
  marking.  The conditions on the homology classes $d(v)$ are
  topological, that is, locally constant among maps with a fixed
  domain.  Each of the other conditions defining
  $\Map^{k,p}_\Gamma(C^\times,X,\ul{L})$ corresponds to a $C^q$
  differentiable map from $\Map^{k,p}(C^\times,X,\ul{L})$ for
  $q < k - 2/p - \min_e \mu(e)$ with surjective linearization for
  $q \ge 1$, which we assume.  It follows from the implicit function
  theorem for Banach manifolds that
  $\Map^{k,p}_\Gamma(C^\times,X,\ul{L})$ is a Banach submanifold of
  the space $\Map^{k,p}(C^\times,X,\ul{L})$.

The universal space will incorporate perturbation data.  The space
$\cJ_\Gamma^l(X,D)$ of domain-dependent perturbations of class $C^l$
is a Banach manifold, and similarly for the space $\cJ_\Gamma^l(X,D)$
of domain-dependent perturbations whose value is fixed on the boundary
of $\ol{\U}_\Gamma$ corresponding to lower-dimensional strata.
Indeed, the space of almost complex structures $J_\Gamma$ equal to
$J_D$ near the nodes is a smooth Banach manifold as in
McDuff-Salamon \cite[Proposition 3.2.1]{ms:jh}.  Each of the
conditions fixes $J_\Gamma$ on a subset.  Hence $\cJ_\Gamma^l(X,D)$ is
also a smooth Banach manifold.  We leave it to the reader to show that
$\cJ_\Gamma^l(X,D)$ is non-empty, using the fact that $C^l$ functions
on the boundary of smooth manifolds with corners have extensions over
the interior.

The universal moduli spaces are cut out locally by Fredholm sections
of a Banach vector bundle.  Recall from \eqref{coll} the local
trivializations $\U^i_\Gamma \to \M^i_\Gamma \times C$.  Our local
universal moduli spaces are cut out from the spaces
$$ \B^i_{k,p,l,\Gamma} := {\M}^i_{\Gamma} \times
\Map^{k,p}_{\Gamma}(C^\times,X,\ul{L},D) \times \cJ_\Gamma^l(X,D)
.$$
Consider the map given by the local trivializations \eqref{localtriv}.
Let $C$ be a nodal disk and $u \in \Map^{k,p}(C,X)$.  Let
$\Omega^0(C^\times,u^*(TX,T\ul{L},TD))_{k,p}$ denote the space of
sections of class $W^{k,p}$ on each component that match at the nodes,
mapping the boundary $(\partial C)_b$ to $TL_b$ for $b \in \{ 0,1 \}$
and the markings $z_i$ to $TD$, $1\leq i \leq n$.  Let
$\Omega^{0,1}(C^\times,u^*TX)_{k-1,p}$ denote the space of $0,1$-forms
of class $W^{k-1,p}$ with no matching condition.  In the case of an
intersection with non-vanishing first derivative, define a fiber
bundle $\E^i_{k,p,l,\Gamma}$ over $\B^i_{k,p,l,\Gamma}$ whose fiber at
$m,u,J$ is
\begin{equation}  (\E^i_{k,p,l,\Gamma})_{m, u, J } :=
 \Omega^{0,1}_{j(m),J}(C^\times, u^*TX)_{k-1,p}
 .\end{equation}
Local trivializations of $\E^i_{k,p,l,\Gamma}$ of class $C^q$ are
provided by geodesic exponentiation from $u$ and parallel transport
using the Hermitian connection defined by the almost complex
structure, see for example \cite[p. 48]{ms:jh}.  We may suppose that
the metric on $X$ is chosen so that $L_b$, $b\in \{0,1\}$, and $D$ are
totally geodesic.  The derivatives of the transition maps involve
derivatives of these parallel transports and hence the derivatives of
the almost complex structure.  In order for the $q$-th derivative of
the transition map to preserve $W^{k-1,p}$ one needs $J_\Gamma$ to be
of class $C^l$ for $q < l - k$, and so the transition maps are of
class $C^q$ only for $q$ also satisfying $ q < l - k $, which we now
assume.  The Cauchy-Riemann operator defines a $C^q$ section
\begin{equation} \label{here} 
\olp : \ \B^i_{k,p,l,\Gamma} \to \E^i_{k,p,l,\Gamma}, \quad (m,u, J)
\mapsto \olp_{j(m),J} u \end{equation}
where 
\begin{equation} 
 \label{olp} \olp_{j(m), J} u = \d u^{0,1}  =  \frac{1}{2} ( \d u  + J \circ \d u \circ j(m) ) ,\end{equation} 
and the almost complex structure $J = J_{\Gamma,m,z,u(z)}$ depends on
$(m,z) \in \M^i_\Gamma \times C^\times$.  The section $\olp$ has
Fredholm linearization on Sobolev class $W^{k,p}$ sections by results
of Lockhart-McOwen \cite{loc:ell} on elliptic operators on strip-like
end manifolds which are non-degenerate at infinity, as for
\eqref{here0}.  For $p > 2$ these results depend on work of
Maz'ja-Plamenevskii \cite{mazja:est}, see also the treatments in
Schwarz \cite{sch:coh} and Donaldson \cite[Section 3.4]{don:floer}.
The {\em local universal moduli space} is
$$\M^{\univ,i}_{k,p,l,\Gamma}(\ul{L},D) = \olp^{-1}
(\B^i_{k,p,l,\Gamma}) $$
where $\B^i_{k,p,l,\Gamma}$ is considered as the zero section in
$\E^i_{k,p,l,\Gamma}$.  We will later show that the linearization of
\eqref{olp} is surjective, under some assumptions.  Assuming this, it
follows from the implicit function for Banach manifolds that each
local universal moduli space ${\M}^{\univ,i}_{k,p,l,\Gamma}(\ul{L},D)$
is a Banach manifold of class $C^q$. The forgetful morphism
$$ \varphi_i: \M^{\univ,i}_{k,p,l,\Gamma}(\ul{L},D) \to \cJ_\Gamma^l(X,D)$$
is the restriction of a $C^q$ Fredholm map (the projection times the
Cauchy-Riemann operator) and so also $C^q$ Fredholm.

The Sard-Smale theorem may be applied in each local trivialization to
guarantee the existence of a comeager set of perturbations for which
the moduli space is cut out transversally.  Let
$$\M^{\univ,i}_{k,p,l,\Gamma}(\ul{L},D)_d \subset
\M^{\univ,i}_{k,p,l,\Gamma}(\ul{L},D) $$
denote the subset on which $\varphi_i$ has Fredholm index $d \ge 0$,
and is therefore submersive since the linearization of \eqref{olp} is
assumed to be surjective.  By the Sard-Smale theorem, for $q$ greater
than $d$ the set of regular values $\cJ^{i,l,\reg}_{\Gamma}(X,D)$ in
the image of $\ol{\M}^{\univ,i}_{k,p,l,\Gamma}(\ul{L},D)_d$ in
$\cJ^l_\Gamma(X,D)$ is comeager.  The set $\cJ^{i,\reg}_\Gamma(X,D)$ of
smooth domain-dependent structures is also comeager.  Indeed, for any
fixed bound on the first derivative, the space of $\cJ^{i,\reg, l, <
  B}_\Gamma(X,D)$ that are regular for every stable Floer trajectory
$u: C \to X$ with $\sup | \d u | < B$ is open and dense, as in
Floer-Hofer-Salamon \cite[Proof of Theorem 5.1]{fhs:tr}.  Taking the
intersection over the first derivative bounds $B$ implies that
$\cJ^{i,\reg}_\Gamma(X,D)$ is comeager in $\cJ_\Gamma(X,D)$.  Let
$$ \cJ^{\reg}_\Gamma(X,D) = \cap_{i} \cJ^{i,\reg}_{\Gamma}(X,D) .$$
Fix $J_\Gamma \in \cJ^{\reg}_\Gamma(X,D)$.  The moduli space
$\M^i_\Gamma(\ul{L},D) = \varphi_i^{-1}(J_\Gamma)$ is a
finite-dimensional manifold of class $C^q$.  By elliptic regularity,
every element of ${\M}^i_{\Gamma}(\ul{L},D) $ is smooth and so this
definition agrees with the previous definition and is independent of
$q$.

Finally we patch together the moduli spaces defined using the local
trivializations of the universal disk.  The transition maps for the
local trivializations define smooth maps
$$ {\M}^i_{\Gamma}(\ul{L},D) |_{ {\M}^i_\Gamma \cap {\M}^{i'}_{\Gamma}} \to
{\M}^{i'}_{\Gamma}(\ul{L},D) _{{\M}^i_{\Gamma} \cap {\M}^{i'}_{\Gamma}} .$$
Therefore the space
${\M}_{\Gamma}(\ul{L},D) = \cup_i {\M}^i_{\Gamma}(\ul{L},D) $
has a smooth atlas.  Since the moduli space of stable strips $\M_\Gamma
= \cup_i \M_\Gamma^i$ of type $\Gamma$ is Hausdorff and second
countable and each piece $\M^i_\Gamma(\ul{L},D)$ is Hausdorff and
second countable, the union ${\M}_{\Gamma}(\ul{L},D)$ is
Hausdorff and second countable.  So $\M_\Gamma(\ul{L},D)$ has the
structure of a smooth manifold.

We now prove that the linearization of \eqref{olp} is surjective
provided that there is at most one marking on each maximal ghost
component. Let $\eta$ be a distributional $0,1$-form representing an
element in the cokernel of the linearization of \eqref{here}.  On each
component of $C$, we obtain an element in the kernel of the adjoint on
the complement of the interior markings.  By elliptic regularity
$\eta$ is class $W^{k,p}$ on this complement.  We show $\eta$ vanishes
on each component.

On components without tangencies, the element $\eta$
vanishes near any point $z \in C$ at which the derivative $\d u(z)$ of
the map is non-zero, by an argument similar to that in McDuff-Salamon
\cite[Proposition 3.2.1]{ms:jh}: The linearization of \eqref{olp} with
respect to the almost complex structure on the sphere components is
\begin{equation} \label{iota}
 T_{J_\Gamma} \cJ_\Gamma(X,D) \to \Omega^{0,1}(C^\times, u^*
 TX)_{k-1,p} , \quad \tau \mapsto 
 ( ( \tau | C^\times) \circ \d u
 \circ j)/2 .\end{equation}
At any point $z \in C^{\thick}$ where $\d u (z) \neq 0$ we may find an
infinitesimal almost complex structure $\tau$ such that the right hand
side of \eqref{iota} pairs non-trivially with $\eta(z)$.  It follows
that $\eta$ must vanish in a neighborhood of $z$.  Unique continuation
for solutions to $D_u^* \eta = 0$ implies that $\eta$ vanishes
identically. 

Next we consider components of the two-dimensional part on which the
map is constant.  If $u: C \to X$ is a map that is constant then the
linearized operator is constant on each disk component of $C$ and
surjective by a doubling trick.  However, we also must check that the
matching conditions at the nodes are cut out transversally.  Consider
a ``ghost component'' $C' \subseteq C$ consisting of a union of disks
on which $u$ is constant, attached by nodes, say with boundary in
$L_b$ for $b \in \{ 0, 1 \}$.  Let $C''$ denote the normalization of
$C'$, obtained by replacing each nodal point $w_i$ in $C'$ with a pair
of points $w_i^\pm$ in $C''$.  Since the combinatorial type of the
component is a subgraph of a tree, the combinatorial type must itself
be a tree.  We denote by $T_u L_b$ the tangent space at the constant
value of $u$ on $C'$.  We suppose that there are $m$ ghost components.
Given a choice of a distinguished ``root'' component, each non-root
component has a unique ``outgoing node'' pointing towards the root
component.  Taking the differences of the maps at the nodes defines a
map
\begin{equation} \label{diffs} \delta: \ker(D_u | C'') \cong T_u L_b^m
  \to T_u L_b^m, \quad \xi \mapsto ( \xi(w_i^+) - \xi(w_i^-) )_{i=1}^m
  .\end{equation}
An explicit inverse to $\delta$ is given by defining recursively as
follows.  Consider the orientation on the combinatorial type $\Gamma''
\subset \Gamma$ induced by the choice of outgoing semi-infinite edge
of $\Gamma$.  For $\eta \in T_u^m$ define an element $\xi \in T_u L^k$
by 
$$\xi(h(e)) - \xi(t(e)) = \eta(e) $$
whenever $t(e),h(e)$ are the head and tail of an edge $e$
corresponding to a node.  The element $\xi$ may be defined recursively
starting from any choice of component in the maximal ghost components.
The matching conditions at the nodes connecting $C'$ with the
complement $C - C'$ are also cut out transversally.  Indeed on the
adjacent components the linearized operator restricted to sections
$\xi$ vanishing at each node $w$ connecting $C'$ with $C- C'$ is
already surjective by \cite[Lemma 6.5]{cm:trans}.  This implies that
evaluation map on the universal moduli space
$\xi \mapsto \xi(w_-) - \xi(w_+)$ is transverse.  A similar discussion
for collections of sphere components on which the map is constant
implies that the matching conditions at the spherical nodes are also
cut out transversally, that is, the evaluation map
$\xi \mapsto (\xi(w_-) , \xi(w_+))$ is transverse to the diagonal for
each spherical node.  In the case that the maximal ghost component
$C'$ has a marking $z_k$, the inductive argument starts with a root
component $C_i \subset C'$ containing the marking $z_k$ as the first
step at the first step in the recursion.  Combining these arguments
completes the proof of surjectivity of the linearized operator except
in the case of a constant strip with values in $L_0 \cap L_1$.  Since
the intersection is assumed transversal, the linearized operator on
such a component is also transverse, and a similar argument shows
transversality at nodes connecting to other ghost components.

 To show transversality in the presence of tangencies with divisors we
 require an auxiliary result of Cieliebak-Mohnke \cite[Lemma
   6.5]{cm:trans} that if the Hamiltonian perturbation vanishes, then
 on the universal moduli space, the map taking the jets up to order
 $k- 2/p$ at any marking is surjective onto the space of jets of
 $0,1$-forms of holomorphic functions.  A similar result holds for
 Floer trajectories as well as pseudoholomorphic maps.  In the case of
 a single tangency at $w(e)$, set
$$\E_{k,p,l,\Gamma} = \{ \eta \in \Omega^{0,1}_{j(m),J}(C^\times, u^*
TX)_{k-1,p} \ | \ \eta(w(e)) = 0 \} $$ 
as in \cite[Lemma 6.6]{cm:trans}.  Then the map $\olp$ of \eqref{here}
has surjective linearization and has Fredholm index two less than the
corresponding map for transverse intersection.  This shows that these
strata are of codimension at least two, and so empty.

Next we prove part (b) of Theorem \ref{main}. That is, we show that
each stratum of index one or two corresponding to a broken trajectory
or disk bubble at distance $0$ or $\infty$ has a tubular neighborhood
in any larger stratum whose closure contains it.  First consider the
case of a broken Floer trajectory.  The existence of a tubular
neighborhood is a consequence of a parametrized version of the
standard gluing theorem for Floer trajectories, keeping in mind that
the gluing parameter is not (at least obviously) a local coordinate on
the moduli space.  Let $u = (u_1,u_2)$ be a broken trajectory with
domain $C = (C_1,C_2)$ of type $\Gamma = \Gamma_1 \# \Gamma_2$ with
$n = n_1 + n_2$ markings.  One constructs for any sufficiently small
gluing parameter $\delta$ and {\em approximate} Floer trajectory
$u_\delta : C_\delta \to X $ with domain the stable strip $C_\delta$
given by gluing $C_1,C_2$ using a neck of length $1/\delta$.  Choose a
local trivialization of the universal bundle as in \eqref{localtriv}.
We identify $\M_\Gamma^i$ locally with $T_{[C]} \M_\Gamma^i$ by
choosing coordinates.  For any $m \in \M_\Gamma^i$ let $j(m)_\delta$
denote the complex structure on $C_\delta$ obtained by gluing together
the given complex structures on the components of $C$ given by $j(m)$.
Consider the map
\begin{align} 
 \cF_u^\delta : \ \M_\Gamma \times \Omega^0(C_\delta^\times,
 u_\delta^* (TX, T \ul{L},TD))_{1,p} \to &\Omega^{0,1}(C_\delta^\times,
 u_\delta^* TX)_{0,p}
 \\ (m,\xi)
 \mapsto &\Phi_\xi^{-1} \olp_{j(m)_\delta, J_\Gamma }
 \exp_{u_\delta}(\xi)
 .\end{align}
We show that $(\cF_u^{\delta})^{-1}(0)$ is non-empty and cut out
transversally.  To show that the linearized operator is surjective,
one constructs a right inverse from the right inverses associated with
$u_a, a = 1,2$. Namely consider the maps
\begin{align}  
\cF_{u_a}: \M_{\Gamma_a} \times \Omega^0(C_a^\times, u_a^* (TX, T
\ul{L},TD))_{1,p} \to &\Omega^{0,1}(C_a^\times, u_a^* TX)_{0,p}
\\ (m_a,\xi) \mapsto &\Phi_\xi^{-1} \olp_{j(m_a), J_{\Gamma_a}}
\exp_{u_a}(\xi) .\end{align} 
Since the moduli spaces of index one are regular, one has right
inverses $Q_a$ for the linearized operators $D\cF_{u_a}$ for $a =1,2$.
Using these, one constructs a right inverse $Q_\delta$ for the
linearization $D_0 \cF_u^\delta$.  Then one checks that the following
zeroth-order, first-order, and quadratic estimates hold: For some
constant $\rho > 0$ and $m,\xi$ sufficiently close to $0$, there
exists a monotonically decreasing function $\eps(\delta) \to 0$ such
that for all $\delta \in (0,\delta_0]$
$$ \Vert \cF_u^\delta(0) \Vert_{0,p,\delta} < \eps(\delta), \quad
  \Vert Q^\delta \Vert < \rho, \quad \Vert D_0 \cF_u^\delta -
  D_{(m,\xi)} \cF_u^\delta \Vert < \rho ( \Vert m \Vert + \Vert \xi
  \Vert_{1,p}) $$
where the second norm is the operator norm from $W^{0,p}$ to $W^{1,p}$
on $C_\delta$ and the third norm is the operator norm from $W^{1,p}$
to $W^{0,p}$ on $C_\delta$ as in Ma'u \cite[5.2.1]{mau:gluing},
\cite[5.4]{mau:gluing}, \cite[5.5]{mau:gluing}.  (It seems that
versions of these estimates for $W^{k,p}, k > 1$ are missing from the
literature and so these gluing arguments do not apply to the case that
$\Gamma$ includes non-trivial tangency conditions.) The quantitative
version of the implicit function theorem as in \cite[Appendix
  A.3]{ms:jh} shows that there exists a unique solution
$(m(\delta),\xi(\delta))$ to
$\cF_u^\delta(m(\delta),\xi(\delta)) = 0$ with
$(m(\delta),\xi(\delta)) \in \on{Im}(Q^\delta) $.
Thus the solution space is non-empty. By the implicit function theorem
in its standard form, $(\cF_u^{\delta})^{-1}(0)$ is locally a smooth
manifold modelled on $\ker D_0 \cF_u^{\delta}$.  

Each broken trajectory is the limit of a {\em unique} end of the
one-dimensional component of the moduli space.  Indeed, if $v$ is an
adapted Floer trajectory sufficiently close to a broken trajectory $u$
then $v$ corresponds to an element of $(\cF_u^\delta)^{-1}(0)$ near
$(m(\delta),\xi(\delta))$ for {\em some} $\delta \in (0,\delta_0)$,
see for example \cite[Section 5.7]{mau:gluing}.  Therefore $v$ can be
connected by a small path of Floer trajectories to the trajectory
corresponding to $(m(\delta), \xi(\delta))$.  Thus the gluing map is
locally surjective.  Since $[u]$ is the limit of a unique end, there
exists a neighborhood of $[u]$ in $\ol{\M}(\ul{L},D)$ homeomorphic to
$[0,1)$, with $[u]$ mapping to $0$ under the homeomorphism.  This
argument shows that $[u]$ has a tubular neighborhood in
$\ol{\M}(\ul{L},D)$.  Note that this argument does not show that the
parameter $\delta$ is a coordinate, that is, that the gluing map is
injective.

The case of gluing at a boundary node is treated in Abouzaid
\cite[5.50]{ab:ex} and Biran-Cornea \cite[Section 4]{bc:ql}, for
example, in the domain-independent case.  Let $u: C \to X$ be a map
from a nodal disk to $X$, with value $x \in X$ at a node $w \in C$.
For simplicity we assume that $w$ separates $C$ into components $C_a,
a = 1,2$.  Associated to the punctured surface $C^\times - \{ w \}$
there is a surjective linearized operator defined as follows.  Choose
strip-like coordinates $s_a \in (0,\infty)$ and $t_a \in [0,1]$ near
$z$ in $C_a$.  Consider the Banach manifold $\Map_{1,p,\eps}(C,X)$ of
continuous maps locally of class $W^{1,p}$, with finite weighted norm
$$ \Vert u \Vert_{1,p,\eps} := \left( \int ( |\d u(z) |^p + \dist(
x,u(z))^p )) e^{ \eps p | s|} \d z \right)^{1/p} $$ 
where the integral is defined using a measure $\d z$ constructed using
cylindrical coordinates near the node, that is, $\d z = \d s \d t$
near the node, and $\epsilon >0$.  Let
$$ \B_{1,p,\eps} = \M_\Gamma^i \times \Map_{1,p,\eps}(C,X) .$$
For such $u$ and one-forms $\eta \in \Omega^{0,1}(C, u^* TX)$ there is
a similar norm
$$ \Vert \eta \Vert_{1,p,\eps} := \left( \int_C | \eta(z) |^p e^{\epsilon p | s|}
\d z \right)^{1/p} .$$
Let $(\cE_{1,p,\eps})_u$ denote the space of one-forms with finite
$1,p,\eps$-norm and $\cE_{1,p,\eps}= \cup_{(m,u) \in \B_{1,p,\eps}}
(\cE_{1,p,\eps})_u$.  Then $\cE_{1,p,\eps}$ is a smooth Banach vector
bundle over $\B_{1,p,\eps}$, with a smooth section $ \B_{1,p,\eps} \to
\cE_{1,p,\eps}$ given by the Cauchy-Riemann operator
$\olp_{j(m)_\delta, J_\Gamma }$.  For $\delta \in (0,1)$ the
linearization $D_u$ of $\olp_{j(m)_\delta, J_\Gamma}$ may be
identified with the linearized Cauchy-Riemann operator on the
un-punctured curve $C$ and so is surjective by assumption.

Given a gluing parameter $\delta$ as in \eqref{gluing}, form a glued
curve $C(\delta)$ by removing small balls around the nodes and gluing
together using a map $z \mapsto \delta/z$; the image of $ |z| \in
(\delta^{2/3}, \delta^{1/3})$ is called the {\em neck region}.  The
local coordinates on the small balls induce cylindrical coordinates on
the neck region.  Associated to $u$ is the preglued map $u(\delta)$
defined as follows. Fix a cutoff function $\chi(s)$ such that $\chi(s)
= 1$ for $s \leq -1$ and $\chi(s) = 0$ for $s \ge 1$.  Given $\lambda
> 0$ choose $\delta $ sufficiently small so that $\lambda \ll -
\on{log}(\delta^{1/6})$.  The {\em preglued map} $u(\delta)$ is given
by $u$ away from the neck region, and given by
$$ u(\delta) = \exp_x( \chi_{|2\lambda-1|} \exp_x^{-1} u) $$
on the neck region.  Thus $u(\delta)$ is constant and equal to $x$ on
the part of the neck with coordinates $[-2\lambda,2\lambda] \times
[0,1]$.  Consider a weighted Sobolev space $\Omega^{0,1}(C(\delta),
u(\delta)^* TX)_{0,p,\eps}$ of sections of $ u(\delta)^* TX \otimes
\Lambda^{0,1} T^\dual C(\delta)$ by
$$ \Vert \eta \Vert_{p,\delta,\eps} = \left( \int_{C(\delta)} |\eta|^p e^{
  (2\lambda - |s|) p \eps}  \d z \right)^{1/p}. $$
Similarly define a weighted Sobolev space $\Omega^0(C(\delta),
u(\delta)^* TX)_{1,p,\eps}$ of sections of $u(\delta)^* TX$ by 
$$ \Vert \xi \Vert_{1,p,\eps} = \left( | \xi(0,0) |^p +
\int_{C(\delta)} | \xi - \xi(0,0) |^p e^{ (2\lambda - |s|) p \eps} |^p
\right)^{1/p} \d z. $$
Let $\Phi_\xi$ denote parallel transport along
$\exp_{u(\delta)}(\xi)$.  Then the non-linear map
$$ \cF_\Gamma^i: \ \M_\Gamma^i \times \Omega^0(C(\delta), u(\delta)^*
TX) \to \Omega^{0,1}(C(\delta), u(\delta)^* TX), \quad \xi \mapsto
\Phi_\xi^{-1} \olp_{J,j(m)} $$
cuts out the moduli space locally.  The gluing estimates in
e.g. \cite[Lemma 5.2]{ab:ex} show that for some constant $C$
independent of the gluing parameter $\delta$ and neck length $\lambda$,
$$ \Vert \cF_\Gamma^i(0) \Vert_{p,\eps} \leq C e^{-2(1 - \eps)\lambda} $$
$$ \Vert \cF_\Gamma^i(\xi_1) - \cF_\Gamma^i(\xi_2) - D_0 \cF^i_\Gamma
( \xi_1 - \xi_2) \Vert_{p,\eps} \leq C \Vert \xi_1 + \xi_2
\Vert_{1,p,\eps} \Vert \xi_1 - \xi_2 \Vert_{1,p,\eps} $$
and the linearized operator $D_0 \cF^i_\Gamma$ has a uniformly bounded
right inverse $Q_{u(\delta)}$.  By the quantitative version of the
implicit function theorem, there exists a unique solution to
$\cF_\Gamma^i(\xi) = 0$ with $\xi$ in the image of the right inverse
$Q_{u(\delta)}$.  Thus any trajectory with a boundary node of index
one is a boundary point of a component of the top-dimensional stratum
of trajectories without nodes.  A similar argument to the case of
gluing along strip-like ends shows that each nodal trajectory is a
limit of a {\em unique} end of the one-dimensional stratum.

The existence of systems of orientations in part (c) of Theorem
\ref{main} compatible with the Behrend-Manin morphisms is a special
case of the construction of orientations for Lagrangian Floer theory
in Fukaya-Oh-Ohta-Ono \cite{fooo}, see also \cite{orient} or
\cite{charest:clust}.  The tangent space to $\ol{\M}_n(\ul{L},D)$ at
any element $[u:C \to X]$ is the product of the tangent space $T_{[C]}
\ol{\M}_n$ to $\ol{\M}_n$ with the kernel $\ker(D_u)$ of the
linearized Cauchy-Riemann operator $D_u: \Omega^0(C,u^* (TX,TL)) \to
\Omega^{0,1}(C, u^* TX)$.  The former was oriented in the discussion
following \eqref{orientmod}.  To orient the latter, for each element
$x \in \cI(L_0,L_1)$ one chooses an {\em end datum} consisting of a
Cauchy-Riemann operator $D_x$ on a map from a disk with one removed
marking on the boundary (considered as a surface with a strip-like
end) asymptotic to $x$ at the end with Lagrangian boundary condition
$F_{t,x} \in \Lag( T_x X)$ interpolating between $T_x L_0$ and $T_x
L_1$.  Given a trajectory $u$, a degeneration argument gives an
isomorphism of determinant lines $\det(D_u) \to \det(D_{x_-})^{-1}
\otimes \Lambda^{\on{top}}(T_x L_0) \otimes \det(D_{x_+})$ for the
ends $x_\pm$ of $u$, up to the determinant line of a Cauchy-Riemann
operator on a sphere which is canonically oriented by the almost
complex structure.  A choice of orientations for $\det(D_x), x \in
\cI(L_0,L_1)$ induces orientations on the moduli spaces of Floer
trajectories.  The gluing sign for the kernel of $D_u$ is positive
while the sign for the inclusion $\M_\Gamma \to \ol{\M}_{\Gamma'}$ is
computed in the discussion after \eqref{orientmod}.  This ends the
proof of Theorem \ref{main}.
\end{proof} 

\begin{remark} \label{loss} {\rm (Loss of derivatives issues)} The
  reader is warned that the above constructions do not give a
  universal moduli space over all of $\M_\Gamma$. The reason is that
  the transition maps for the universal disk do not induce
  differentiable maps of Sobolev spaces because of the loss of
  derivatives: the derivative of the transition map from $\U_\Gamma^i$
  to $\U_\Gamma^{i'}$ does not map
  ${\M}^{\univ,i}_{k,p,l,\Gamma}(\ul{L},D)$ to
  ${\M}^{\univ,i'}_{k,p,l,\Gamma}(\ul{L},D)$ but rather to
  ${\M}^{\univ,i'}_{k-1,p,l,\Gamma}(\ul{L},D)$.  It seems likely that
  using elliptic regularity as in Dragnev \cite{dragnev:fred} one can
  show that $\bigcup_i {\M}^{\univ,i}_{k,p,l,\Gamma}(\ul{L},D)$ is a
  $C^q$ Banach submanifold by showing that the transition maps are
  differentiable after restricting to holomorphic maps, by an
  inductive argument showing elliptic regularity for each derivative.
  The more straightforward approach taken here is to apply Sard-Smale
  in each local trivialization, and then show that the moduli space
  with fixed almost complex structure is $C^q$.  For similar reasons,
  there are issues representing the deformations of complex structure
  on the domains as variations of the nodal points on curves with
  fixed complex structures: if one does so, then the requirement that
  the maps on either side of the node agree is not differentiable and
  so the resulting space is not a differentiable Banach manifold.  In
  this respect, \cite[Proposition 5.7]{cm:trans} is incorrect, but the
  treatments in \cite{gersten:trans} and \cite[Section 6.2]{ms:jh}
  avoid this problem, by different means; essentially McDuff-Salamon
  \cite{ms:jh} solve this problem by breaking the construction down
  into stages.
\end{remark} 

\begin{remark} {\rm (Crowded types)} \label{uncrowd} 
Given a crowded type $\Gamma$, let $\Gamma'$ denote a type 
obtained by forgetting all but one marking on each maximal ghost
component.  Because of the (Product axiom), that forgetful map lifts to
a map $\M_{\Gamma}(\ul{L},D) \to \M_{\Gamma'}(\ul{L},D)$.  Thus
$\M_\Gamma(\ul{L},D)$ is also smooth, although not of expected
dimension.  In particular this shows that if $\Gamma$ is a type with a
non-transverse intersection, then $\M_{\Gamma}(\ul{L},D)$ is empty
since $\M_{\Gamma'}(\ul{L},D)$ is.
\end{remark} 

The theorem above implies the existence of regular collections of
perturbation data.  We say that a coherent collection
$\ul{J} = (J_\Gamma)$ is {\em regular} if each
$J_\Gamma \in {\cJ}^{\reg}_\Gamma(X,D)$ for each uncrowded type
$\Gamma$ of stable trajectory of index one or two.  Since the set of
types of stable trajectories with a given underlying stable strip is
countable, the space of regular coherent collections is non-empty.

\begin{remark}  \label{antiinv}
\begin{enumerate} 
\item {\rm (Extension to boundary divisors)} The results of Theorem
  \ref{main} hold for the moduli spaces $\M_\Gamma(\ul{L},\ul{D})$
  when we add a pair of boundary divisors $(D_0,D_1)$ by an adaptation
  similar to that of \cite[Theorem 9.8]{cm:trans}.

\item 
{\rm (Extension to families of divisors)} The results of Theorem
  \ref{main} hold for the moduli spaces $\M_\Gamma(\ul{L},\ul{D})$
  when we add families of divisors by arguments similar to that of the
  proof of invariance of \cite{cm:trans}.

\item {\rm (Anti-symplectic involutions)} In the case that $L_b$ is
  the fixed point set of a anti-symplectic involution $\iota_b$, the
  same result (regularity for a comeager set of perturbations) holds
  for perturbation data $J_\Gamma$ that are $\iota_b$-anti-invariant
  on $\ell_b^{-1}(\infty)$, $b \in \{ 0 , 1 \}$.  Adapted stable maps
  $u: C \to X$ on involution-fixed domains $C$ automatically satisfy
$u(z) \neq \iota_b(u(*z)) = \iota_b(u(\ol{z}))$
for some $z \in C$, where $*$ acts as complex conjugation on the disks
at infinity, the latter being identified to complex unit disks.
Indeed, an involution-fixed domain $C$ automatically has markings
$z_i$ contained in the fixed-point set of the involution $C_\R$.  If
$u(z_i) = \iota_b(u(\ol{z_i})) = \iota_b(u(z_i))$ then $u(z_i)$ is
$\iota_b$-fixed.  So $u(z_i) \in L_b$ which is impossible since $L_b
\cap D = \emptyset$.  Hence $u(z_i) \neq \iota_b(u(\ol{z_i})) $ and
the same is true for $z \in C$ near $z_i$ as well.  In this case an
anti-invariant perturbation which is non-vanishing near $u(z)$ but
vanishing near $\iota_b u(\ol{z})$ which makes the universal moduli
space transverse.  The map $u \mapsto \iota_b \circ u \circ * $
defines an involution of the moduli space $\M_\Gamma(\ul{L},\ul{D})$
also denoted $\iota_b$.

\end{enumerate} 
\end{remark} 

\begin{remark} \label{true} The strata $\M_\Gamma(\ul{L},\ul{D})$ of
  index two and expected dimension zero are of three possible types.
\begin{enumerate} 
\item The first possibility is that $\Gamma$ is a tree with a single
  edge corresponding to a node connecting two components connected by
  a segment of infinite length.  The normal bundle to the stratum has
  fiber canonically isomorphic to $\R_{\ge 0} ,$ corresponding to
  deformations that make the length finite.  We call such a
  $\M_\Gamma(\ul{L},\ul{D})$ a {\em true boundary stratum}.  We denote
  by $\cT$ the set of combinatorial types of true boundary strata.
  Each is represented by a tree with two vertices, one finite edge
  endowed with an infinite metric and some number of semi-infinite
  edges.  There are two subcases, depending on whether the incoming
  and outgoing markings lie on the same disk component, or different
  disk components; the latter subcase corresponds to concatenation of
  Floer trajectories.
\item The second possibility is that $\Gamma$ corresponds to a stratum
  with a boundary node of length zero.  We call such a
  $\M_\Gamma(\ul{L},\ul{D})$ a {\em fake boundary stratum.}  The
  stratum is not a boundary in the sense that the space is not locally
  homeomorphic to a manifold with boundary near such a stratum since
  any such trajectory may be deformed either by deforming the disk
  node, or deforming the length of the node to a positive real number.
\end{enumerate} 
\end{remark} 

\begin{remark} \label{spheres2} {\rm (Diagonal boundary conditions)}
  We continue Remark \ref{spheres}.  In the case of diagonal boundary
  conditions, for the disk components at infinite length we assume
  that the domain-dependent almost complex structure $J_\Gamma$ is
  pulled back from the universal moduli space of {\em marked spheres}
  on each such component.  Let $\Gamma$ denote the type of stable
  marked strip with at at least one component at infinity.  Denote the
  $\ti{\Gamma}$ the type of stable curve where each disk has been
  replaced by a sphere.  Since there are no boundary markings on each
  component at infinity, there must be at least one component at
  infinity with a single node, connecting that component to the rest
  of the configuration.  Such a component cannot be a sphere since
  configurations containing these are codimension two.  In the case of
  a disk, by assumption the perturbation system on the disk
  (considered as a sphere with anti-holomorphic involution) is
  independent of the choice of involution.  By rotating the sphere in
  a way fixing the node but not fixing any other marking on the
  boundary, we obtain a family of configurations which, as disks with
  markings, are all distinct.  Hence again such a configuration is not
  isolated.
\end{remark}

\subsection{Compactness}  

Next we show that the subset of the moduli space satisfying an energy
bound is compact for suitable perturbation data and stabilizing
divisors of sufficiently large degree.  In general, compactness of the
spaces of adapted Floer trajectories (so that they have stable
domains) can fail since unstable components can develop.  For
simplicity, we restrict to the case of a single stabilizing divisor;
the other cases are similar.

\begin{definition} \label{stabdef} 
For $E > 0$, an almost complex structure $J_D \in \J(X,D)$ is {\em
  $E$-stabilized} by a divisor $D$ if and only if the following holds:
\begin{enumerate}
%
\item[] {\rm (Sufficient intersection condition)} Each non-constant
  $J_D$-holomorphic sphere $u: S^2 \to X$ with energy less than $E$
  has at least three intersection points with the divisor $D$, that
  is, $u^{-1}(D)$ has order at least three.
\end{enumerate} 
\end{definition} 

\begin{definition} \label{suffic} A divisor $D$ with Poincar\'e dual 
$[D]^\dual = k m_0[\omega]$ for some $k \in \N$ has {\em sufficiently
    large degree} for an almost complex structure $J_D$ if and only if
\begin{itemize} \label{large} 
\item[] {\rm (Sphere condition)}  $([D]^\dual,\alpha) \ge 2(c_1(X),\alpha) + \dim(X) + 1$ 
for all $\alpha \in H_2(X,\Z)$ representing
 non-constant 
$J_D$-holomorphic spheres, and

\vskip .1in

\item[] {\rm (Disk condition)} $([D]^\dual,\beta) \geq 1$ for all
  $\beta \in H_2(X,L,\Z)$ representing non-constant $J_D$-holomorphic
  disks.
\end{itemize}
\end{definition} 

Sufficiently large divisors always exist by an argument of
Cieliebak-Mohnke \cite{cm:trans}: Let $J \in \J(X,\omega)$ be a
compatible almost complex structure.  By \cite[Lemma 8.11]{cm:trans}
and Lemma \ref{rati_divi2}, for any $\theta > 0$, there exists
$d_0(\theta)$ such that if $k m_0 \ge d_0(\theta)$ then $D$ is
sufficiently large for any almost complex structure $J_D$ that is
$\theta$-close to $J$.

We introduce the following notation for almost complex structures
close to the given one.  Given $J \in \J(X,\omega)$ (resp.
$J_t \in \J(X,\omega)$ depending smoothly on $t\in [0,1]$), denote by
$\J_\tau(X,D,J,\theta)$ (resp. $\J_{\tau,t}(X,D,J_t,\theta)$) the
space of tamed almost complex structures $J_D \in \J_\tau(X,\omega)$
such that $\Vert J_D - J \Vert < \theta$ (resp. of families
$J_{D,t}\in \J_\tau(X,\omega)$, $t\in [0,1]$, such that
$\Vert J_{D,t} - J_t \Vert < \theta$) in the sense of
\cite[p. 335]{cm:trans}.  The following lemma on existence of
stabilizing almost complex structures is a special case of
Cieliebak-Mohnke \cite[Proposition 8.14, Corollary 8.20]{cm:trans}.

\begin{lemma} \label{largelem} Suppose that $D$ has sufficiently large 
degree for an almost complex structure $\theta$-close to $J$.  For
each energy $E > 0$, there exists an open and dense subset
$\J^*(X,D,J,\theta,E)$ in $\J_\tau(X,D,J,\theta)$ such that if $J_D
\in \J^*(X,D,J,\theta,E)$, then $J_D$ is $E$-stabilized by $D$.
Similarly, if $D = (D_t)$ is a family of divisors, then for each
energy $E > 0$, there exists a dense and open subset
$\J^*_t(X,D,E,J_t,\theta)$ in $\J_t(X,D,J_t,\theta)$ such that if
$J_{D,t} \in \J^*_{t}(X,D,E,J_t,\theta)$, then $J_{D,t}$ is
$E$-stabilized for all $t$.
\end{lemma}

\begin{proof} 
  For the sake of completeness, we recall the proof in the
  time-independent case.  An application of Sard-Smale shows that the
  set $\J^*(X,D,J,\theta,E)$ of almost complex structures such that
  all simple holomorphic spheres of energy at most $E$ are regular is
  comeager in $\J(X,D,J,\theta)$. An argument using Gromov compactness
  (see \cite{cm:trans}) shows that $\J^*(X,D,J,\theta,E)$ is open.  To
  complete the proof it remains to show that any
  $J_D \in \J^*(X,D,J,\theta,E)$ is $E$-stabilized.

  We compute the dimension of the moduli space of holomorphic spheres
  in the divisor as follows.  If $i: D \to X$ is the inclusion then
  $ i^*TX $ is the sum of $TD$ and the normal bundle $N$ to $D$.
  Hence $\Lambda^{\top} i^*TX \cong \Lambda^{\top} TD \otimes N $ and
$c_1(D) = c_1(X) |_{D} - [D]|_{D} .$
For $[D]$ sufficiently large the expected dimension of the
parametrized moduli space of simple holomorphic spheres in $D$ of
class $d \in H_2(D)$ is
\begin{equation} \label{expdim} 
\dim(X) + 2 (c_1(D),d) - 5 = \dim(X) + 2 (c_1(X), i_* d) - 2([D],i_*
d)-5 < 0. \end{equation}
So the moduli space of such spheres is empty for generic almost
complex structures.  Since any $J_D$-holomorphic sphere covers a simple
holomorphic sphere, there are no multiply covered 
$J_D$-holomorphic spheres of energy less than $E$ in $D$ either.  

The lower bound on intersection points follows from an upper bound on
intersection multiplicity.  The dimension of the moduli space of
sphere components with intersection multiplicity $\mu$ at one point in
the homology class $d \in H_2(X)$ is
$$ \dim(X) - 4 + 2 (c_1(X),d) - 2 \mu \ge 0 .$$
Hence
$ \mu \leq \dim(X)/2 - 2 + (c_1(X),d) .$
On the other hand, since the divisor $D$ is sufficiently large the
total intersection number with $D$ is
$ ( [D], d) \ge 2(c_1(X),d) + \dim(X) > 2 \mu .$
If there were two or fewer intersection points with $D$ each with
multiplicity at most $\mu$, then since $([D],d)$ is the sum of
intersection multiplicities we would have $([D],d) \leq 2 \mu$, a
contradiction.  Hence each $J_D$-holomorphic sphere must have at least
three intersection points with the divisor $D$.
\end{proof}

A version of Gromov compactness holds for moduli spaces of
trajectories defined using stabilizing almost complex structures.  We
restrict to perturbation data taking values in $\J^*(X,D,J,\theta,E)$
for a (weakly or strictly) stabilizing divisor $D$ having sufficiently
large degree for an almost-complex structure $\theta$-close to $J$.
Let $J_D \in \J(X,D,J,\theta)$ be an almost complex structure that is
stabilized for all energies, for example, in the intersection of
$\J^*(X,D,J,\theta,E)$ for all $E$.  For each energy $E$, there is a
contractible open neighborhood of $J_D$ in $\J^*(X,D,J,\theta,E)$ that
is $E$-stabilized.  Let $\Gamma$ be a type of stable trajectory.
Disconnecting the components that are connected by boundary nodes with
positive length one obtains types $\Gamma_1,\ldots, \Gamma_l$, and a
decomposition of the universal curve $\U_\Gamma$ into components
$\ol{\U}_{\Gamma_1},\ldots, \ol{\U}_{\Gamma_l}$.  Since
$PD[D] = k m_0[\omega]$, any stable trajectory with domain of type
$\Gamma$ and only transverse intersections with the divisor has energy
at most
$$   n(\Gamma_i,k) = \frac{n(\Gamma_i)}{C(k)}$$ 
on the component in $\ol{\U}_{\Gamma_i}$, where $n(\Gamma_i)$ is the
number of markings on $\ol{\U}_{\Gamma_i}$ and $C(k)$ is the
increasing linear function of $k$ arising in the construction of $D$
in Section \ref{stab}.  A perturbation datum $J_\Gamma$ for a type of
stable strip $\Gamma$ is {\em stabilized} by $D$ if $J_\Gamma$ takes
values in $\J^*(X,D,J,\theta, n(\Gamma_i,k))$ on $\ol{\U}_{\Gamma_i}$.
For example, in Figure \ref{onetwo} assuming $C(k) = \frac{k}{K}$ for
a certain $K>0$ and $k = 1$, this means that the perturbation data
should be chosen $K$-stabilized on the boundary of the square, while
on the smaller middle square the perturbation data should be chosen in
the smaller set of $2K$-stabilized perturbation data.

In order to rule out configurations involving markings on ghost
bubbles appearing in the compactification we introduce a second
further assumption on the perburbations.  A perturbation data
$J_\Gamma$ for a type of stable strip $\Gamma$ is {\em
  ghost-marking-independent} if the following holds: Suppose that
$\Gamma$ has a crowded ghost component and let $\Gamma'$ denote the
combinatorial type obtained by forgetting the first marking on that
component.  Then $J_\Gamma$ is obtained from $J_{\Gamma'}$ by pullback
under the forgetful map.  \footnote{We thank G. Xu for pointing out
  the omission of the second item which was present in a different
  form in Cieliebak-Mohnk \cite{cm:trans} but left out in the
  published version of the paper. }

\begin{theorem} \label{compthm} {\rm (Compactness for fixed type)} For
  any collection $(J_\Gamma)$ of coherent, regular, stabilized,
  ghost-marking-independent perturbation data and any uncrowded type
  $\Gamma$ of expected dimension at most one, the moduli space $
  \ol{\M}_\Gamma(\ul{L},D)$ of adapted stable trajectories of type
  $\Gamma$ is compact and the closure of $\M_\Gamma(\ul{L},D)$
  contains only configurations with disk bubbling.
\end{theorem} 

\begin{proof}  
Because of the existence of local distance functions, similar to
\cite[Section 5.6]{ms:jh}, it suffices to check sequential
compactness.  Let $u_\nu: C_\nu \to X$ be a sequence of stable
trajectories of type $\Gamma$, necessarily of fixed energy
$E(\Gamma)$.  The sequence of stable strips $[C_\nu]$ converges to a
limiting stable strip $[C]$ in $\ol{\M}_\Gamma$.  Then $u_\nu: C_\nu
\to X$ has a stable Gromov-Floer limit $u: \hat{C} \to X$, where $\hat{C}$ is
a possibly unstable strip with stabilization $C$, see Theorem
\ref{preconv}.  We show that $u$ is adapted.

From (Compatible with the divisor), we have
$J_\Gamma = J_D \in \J^*(X,D,J,\theta,n(\Gamma_i,k))$ over $D$.  Let
$C_i$ be a connected component of $u^{-1}(D)$.  Either $C_i$ is a
single point, or a union of sphere components on which $u$ is
constant.  In the first case, $u$ has positive intersection
multiplicity with $D$ at $C_i$.  It follows from conservation of local
intersection multiplicity that $C_i$ is the limit of components of
$u^{-1}_\nu(D)$, which must contain markings by the (Marking Property)
for $u_\nu$.  Similarly, if $C_i$ is a union of sphere components,
then the intersection multiplicities at the nodes joining $C_i$ with
$C \backslash C_i$ are positive.  Let $C_{i,\nu}$ be a sequence of
subsets of $C_\nu$ converging to a small neighborhood $C_i'$ of $C_i$
in $C$.  Once again, the intersection multiplicity of $u | C_{i,\nu}$
with $D$ must be positive, hence $u_\nu^{-1}(D) \cap C_{i,\nu}$ is
non-empty for each $\nu$.  It follows that $C_i'$ also contains a
marking; since this holds for any neighborhood $C_i'$ of $C_i$, a
marking must be contained in $C_i$. Note that if
$u_\nu(z_{i,\nu}) \in D$ then $u(z_i) \in D$, by convergence on
compact subsets of complements of the nodes.  This shows the (Marking
property).

To see (Stable domain) property, consider possibly unstable sphere
components.  Since $J_\Gamma$ is regular, the trajectories $u_\nu$
have only transverse intersections with $D$ on the strip components.
Any unstable spherical component $\hat{C}_i$ of $\hat{C}$ attached to
a component of $C$ in $\ol{\U}_{\Gamma_i}$ has energy at most
$n(\Gamma_i,k)$.  Suppose that $u$ is non-constant on
$\hat{C}_i$. Then since $J_\Gamma$ is constant with value an element
of $\J^*(X,D,J,\theta, n(\Gamma_i,k))$ on $\hat{C}_i$, the restriction
of $u$ to $\hat{C}_i$ has at least three intersection points with $D$.
Since $D$ contains no non-constant holomorphic spheres, these
intersection points must be isolated and so markings, which
contradicts the instability of $\hat{C}_i$.  Hence the stable map $u$
must be constant on $\hat{C}_i$, and thus $\hat{C}_i$ must be stable.

Similarly any strip or disk component without interior markings occurs
via bubbling at a bubbling sequence approaching the boundary.  Since
the almost complex structure $J_\Gamma$ is equal to $J_D$ at the
boundary, the disk or strip is $J_D$-holomorphic.  Since $J_D$ is
stabilizing for $D,L_0\cup L_1$, any disk or strip component must have
at least one interior intersection point $z \in u^{-1}(D)$ with $D$.
The corresponding component of $u^{-1}(D)$ must contains a marking, so
either there is another component of the domain attached at $z$, or
$z$ is itself a marking; either way, this disk or strip component is
stable.  This shows that $\hat{C}$ is equal to $C$ and shows the
(Stable domain property).

It remains to check that the limiting configuration is uncrowded.
Suppose $C$ has a spherical component.  After forgetting all but one
marking on maximal ghost components (see Remark \ref{uncrowd}) we
obtain by pull-back under the forgetful map and the
(Ghost-marking-independent) axiom a configuration in an uncrowded
stratum $\M_{\Gamma'}(\ul{L},D)$ of negative expected dimension.  This
contradicts the transversality assumption for uncrowded combinatorial
types.  Hence all components of $C$ are disks.  But there are no disk
ghost components, since the divisor is disjoint from the Lagrangian.
\end{proof}

\begin{proposition}  There exist coherent collections of 
regular, stabilizing, ghost-marking-independent perturbation data. 
\end{proposition} 

\begin{proof} Since the set of stabilizing perturbations contains an
  open neighborhood of the fixed base almost complex structure $J_D$
  stabilizing for all energies, the intersection with the set of
  regular perturbations is non-empty.  In the proof of Theorem
  \ref{main}, perturbations on the ghost components were not necessary
  since constant maps are regular for $J_D$, so in fact that almost
  complex structure may be taken to equal $J_D$ on the ghost
  components.
\end{proof}

\begin{remark}\label{anti2} {\rm (Involutions)}   Continuing Remark \ref{anti},
  \ref{antiinv} suppose that $\iota_b: X \to X$ are anti-symplectic
  involutions, i.e.  $\iota_b^2= \on{Id}$ and
  $\iota_b^*\omega = -\omega$, with fixed locus $L_b$, $b \in \{0,1\}$
  and preserving $D_b$.  We show that perturbations exist satisfying
  good transversality and compactness properties, and so that the
  moduli spaces inherit the involution.

  First note that generic {\em anti-invariant} almost complex
  structures are stabilizing for all energies: Let
  $\J_\tau(D_b)^{\iota_b}$ be the space of tamed almost complex
  structures that are anti-invariant under the symplectic involution,
  that is, elements $J\in \J_\tau(D_b)$ such that $\iota_b^* J = - J$.
  Let $\M(D_b)$ be the moduli space of simple holomorphic spheres in
  $D_b$, and $\M(D_b)^{\iota_b}$ the moduli space of simple real
  holomorphic spheres with respect to the involution, that is, spheres
  $C$ equipped with an anti-holomorphic involution $\iota_C$ and a
  pseudoholomorphic map $u: C \to D_b$ such that
  $u \circ \iota_C = \iota_b|_{D_b}\circ u $.  For a comeager subset
  of $\J_\tau(D_b)^{\iota_b}$, $\M(D_b) \backslash \M(D_b)^{\iota_b}$
  is a smooth manifold of expected dimension and $\M(D_b)^{\iota_b}$
  is a smooth manifold of dimension
  $ \dim( \M(D_b)^{\iota_b}) = \dim \M(D_b)/2$, see \cite[Theorem
  1.11]{wel:inv}.  It follows from the argument of Lemma
  \ref{largelem} (but replacing the dimension on the left hand side of
  \eqref{expdim} by its half) that generic elements of
  $\J_\tau(D_b)^{\iota_b}$ are stabilizing for all energies.

  We wish to choose generic domain-dependent almost complex structures
  for which the anti-symplectic involutions induce an involution on
  the moduli space of Floer trajectories with disk bubbles. Let
  $\J_\tau(X,\ul{D})^{\iota_0,\iota_1}$ be the space of time-dependent
  almost complex structures
  $J_t \in \J_\tau(X), t \in [-\infty,\infty]$ such that $J_0 = J_D$
  and $J_{-\infty}= J_{D_0}$ resp. $J_\infty=J_{D_1}$ is
  anti-invariant under $\iota_b$ and preserves $D_b$ for $b = 0 $
  resp. $b = 1$.  Generic elements of
  $\J_\tau(X,\ul{D})^{\iota_0,\iota_1}$ are stabilized for all
  energies, by a time-dependent version of the argument from the
  previous paragraph.  Each element $J_t$ of
  $\J_\tau(X,\ul{D})^{\iota_0,\iota_1}$ induces a domain-dependent
  almost complex structure $J_\Gamma = J_t$ except that the domain is
  the {\em normalization} of any fiber of $\ol{\U}_\Gamma$, obtained
  by taking the disjoint union of the components and defining
  $J_\Gamma$ to equal $J_t$ on $\ell_0(-t), t \in [-\infty,0]$ resp.
  $\ell_1^{-1}(t), t \in [0,\infty]$.  For $J_\Gamma$ a
  domain-dependent perturbation of $J_{t}$ sufficiently close to $J_t$
  as in Remark \ref{anti}, all moduli spaces of adapted stable strips
  of expected dimension at most one are regular {\em and} compact.
  Furthermore, the involution in Remark \ref{antiinv} is well-defined.
\end{remark} 

\section{Floer cohomology}

In this section we construct Floer cohomology for admissible
Lagrangians, using the regularity of the perturbed moduli spaces in
the previous section.

\subsection{Fundamental classes} 

Although the moduli of index two Floer trajectories is not a manifold
(it is a cell complex), it has a natural rational fundamental class.
This class is a homology class of top dimension generating the local
homology groups at any point in the interior of a one-cell.  In this
section we use these classes to construct the Floer operator. Fix a
family of domain-dependent almost complex structures
$\ul{J} = (J_\Gamma)$ that is regular and stabilized by $D$, $D_0$ and
$D_1$ where $\ul{D} = (D,D_0,D_1)$ is a collection of stabilizing
divisors as in Remark \ref{marginal}, \ref{families} or \ref{anti2}.
Then transversality holds and the moduli spaces of expected dimension
at most one have the expected boundary.

\begin{proposition} \label{zeroone} {\rm (Zero and one-dimensional moduli spaces)} 
\begin{enumerate} 
\item The subset ${\M}_0(\ul{L},\ul{D})$ of expected dimension zero of
  $\ol{\M}^{}(\ul{L},\ul{D})$ is discrete; and
\item the components of the expected dimension one subset
  $\ol{\M}_1(\ul{L},\ul{D})$ of $\ol{\M}^{}(\ul{L},\ul{D})$ have
  one-dimensional cell complex structures.
\item The cell structures may be chosen so that the $0$-skeleton
  $\M_1(\ul{L},\ul{D})_0$ of the non-circle components of
  $\ol{\M}_{1}(\ul{L},\ul{D})$ is the (disjoint) union of
  zero-dimensional strata ${\M}_{\Gamma}(\ul{L},\ul{D})$, where
  $\Gamma$ ranges over true and fake boundary types. 
\end{enumerate} 
\end{proposition} 

\begin{proof}  
  By Theorems \ref{compthm}, \ref{main} each stratum of
  $\ol{\M}^{< E}(\ul{L},\ul{D})$ of expected dimension at most one has
  expected dimension and has compact closure and only finitely many
  combinatorial types occur.  Part (a) 
  follows from the fact that any compact $0$-manifold is a finite set
  of points. 

  For parts (b) and (c) note that by Theorem \ref{main}, for every
  type $\Gamma$ of index two with only one vertex, each connected
  component of $\ol{\M}_\Gamma(\ul{L},\ul{D})$ is a compact connected
  one-manifold with (possibly empty) boundary corresponding to the
  boundary types in Remark \ref{true} and so either homeomorphic to a
  closed interval or a circle.  The space $\ol{\M}_{1}(\ul{L},\ul{D})$
  is obtained from their union by gluing together the closed intervals
  along the boundary points, and so has the structure of a
  one-dimensional cell-complex.
\end{proof} 

Introduce the following notation for moduli spaces.  Denote by
$\M_{1}(\ul{L},\ul{D})_0$ (resp. $\M_{1}(\ul{L},\ul{D})_1$) the
$0$-skeleton (resp. the $1$-skeleton) of
$\ol{\M}_{1}(\ul{L},\ul{D})$. We consider the relative singular
homology $H(\ol{\M}_{1}(\ul{L},\ul{D}), \M_{1}(\ul{L},\ul{D})_0,\Q) $
with rational coefficients.  Assuming the cell complexes are finite
for each energy bound, each $1$-cell is oriented by Theorem \ref{main}
fundamental class relative to the $0$-skeleton.  The sum of these
classes for each index two type $\Gamma$ with one vertex and energy
bound $E$ is denoted by $[\ol{\M}_\Gamma^{<E}(\ul{L},\ul{D})_1 ] $.
Let $|\Gamma| = \ul{n} = (n,n_0,n_1)$ denote the number of
semi-infinite edges in $\Gamma$ corresponding to interior markings of
each type and $|\Gamma|! = n!n_0!n_1!$.  Let $ E> 0 $.  The rational
fundamental class of $\ol{\M}^{< E}_{1}(\ul{L},\ul{D})$ is
\begin{multline*}  [\ol{\M}^{< E}_{1}(\ul{L},\ul{D}) ] := \sum_{\Gamma}
(|\Gamma|!)^{-1} [\M_\Gamma^{< E}(\ul{L},\ul{D})_1 ] 
  \in H(\ol{\M}_{1}^{<
    E}(\ul{L},\ul{D}), \M_{1}^{< E}(\ul{L},\ul{D})_0,\Q)
\end{multline*}
where $\Gamma$ ranges over types of expected dimension one.  We write
the zero-dimensional rational fundamental class
$$ [ {\M}_0^{< E}(\ul{L},\ul{D}) ] := \sum_{[u] \in {\M}_1^{< E}(\ul{L},\ul{D}) }
\sigma([u]) \eps([u]) [u] $$
where the coefficient $\sigma([u]) \in \Q$ is equal to
$|\Gamma|!^{-1}$ if $[u]$ is represented by an element of
${\M}^{< E}_\Gamma(\ul{L},\ul{D})$ and $\eps([u])$ is the orientation
sign in Theorem \ref{main}.

The rational fundamental class of the one-dimensional locus defined
above is an element of {\em relative} homology and we investigate its
boundary. Recall that the long exact sequence for relative homology
includes a {\em boundary map}
$$ \delta: \ H_1(\ol{\M}_{1}^{< E}(\ul{L},\ul{D}), \M_{1}^{< E}(\ul{L},\ul{D})_0,\Q) 
\to H_0(\M_{1}^{< E}(\ul{L},\ul{D})_0,\Q) .$$

\begin{theorem} \label{fundclass3} {\rm (Boundary of the rational
fundamental class)} The dimension one component $\ol{\M}_{1}^{< E}(\ul{L},\ul{D})$ 
of $\ol{\M}^{<E}(\ul{L},\ul{D})$ has rational fundamental class relative to the
  $0$-skeleton
$$
 [\ol{\M}_{1}^{< E}(\ul{L},\ul{D})] \in H_1(\ol{\M}_{1}^{< E}(\ul{L},\ul{D}),
 \M_{1}^{< E}(\ul{L},\ul{D})_0,\Q)$$
with image
$$ \delta [\ol{\M}_{1}^{< E}(\ul{L},\ul{D})] \in H_0(\M_{1}^{< E}(\ul{L},\ul{D})_0,\Q)$$
equal to the sum of fundamental classes of true boundary components
$$ \delta [\ol{\M}_{1}^{< E}(\ul{L},\ul{D})] = \sum_{\Gamma \in \T}
[\ol{\M}_{\Gamma}^{< E}(\ul{L},\ul{D})] $$ 
where $\T$ is the set of index two true boundary types.
\end{theorem} 

\begin{proof}
By Proposition \ref{zeroone} the boundary $ \delta [\ol{\M}_{1}^{<
    E}(\ul{L},\ul{D}) ]$ is a sum of contributions from
combinatorial types corresponding to fake and true boundary
components.  Any fake boundary component $\Gamma$ appearing in the
boundary
$$ \delta [\ol{\M}_{1}^{< E}(\ul{L},\ul{D}) ] \in H(\ol{\M}_{1}^{<
  E}(\ul{L},\ul{D}), \M_{1}^{< E}(\ul{L},\ul{D})_0,\Q) $$
is in the boundary of two combinatorial types corresponding to cells
of maximal dimension by the tubular neighborhood part of Theorem
\ref{main}, up to forgetful equivalence.  That is, the fake boundary
strata corresponds to the morphism in which the intersection points
$z_i^b$ with the divisor $D_b$ for some $b \in \{ 0 , 1 \}$ are
forgotten.

Let $\Gamma^+$ (resp. $\Gamma^-$) be the fake boundary type so that
$\Gamma_+$ has $\ul{n} = (n, n_b)$ markings mapping to $D$ and $D_b$
and $\Gamma_-$ has $n$ markings mapping to $D$.  The fiber of the
forgetful morphism ${\M}_{\Gamma^+}(\ul{L},\ul{D}) \to
\M_{\Gamma^-}(\ul{L},\ul{D})$ has order 
$$|\Gamma_+|!/|\Gamma_-|! =
\frac{n!n_b!}{n!} = n_b! ,$$ 
corresponding to ways of ordering the additional marking.  Hence
$$ (|\Gamma^+|!)^{-1} [{\M}_{\Gamma^+}(\ul{L},\ul{D})] -
(|\Gamma^-|!)^{-1} [\M_{\Gamma^-}(\ul{L},\ul{D})] = 0 .$$
So the contributions of these types to $ \delta [\ol{\M}_{1}^{< E}(\ul{L},\ul{D})
]$ cancel.
\end{proof}

Let $\partial \ol{\M}_{1}^{< E}(\ul{L},\ul{D})$ denote the {\em true 
  boundary} given as the union of true boundary types
$$ \partial \ol{\M}_{1}^{< E}(\ul{L},\ul{D}) = \bigcup_{\Gamma \in \T}
\ol{\M}^{< E}_\Gamma(\ul{L},\ul{D}) .$$

\begin{corollary}   
The one-dimensional fundamental class $[ \ol{\M}_{1}^{<
    E}(\ul{L},\ul{D})]$ lifts to an element in the homology relative
to the true boundary $ H_1(\ol{\M}_{1}^{< E}(\ul{L},\ul{D}), \partial
\ol{\M}_{1}^{< E}(\ul{L},\ul{D}),\Q)$.
\end{corollary} 

\begin{proof} 
Consider the long exact sequence in relative homology
\begin{multline} 
 H_1(\ol{\M}_{1}^{< E}(\ul{L},\ul{D}), \partial \ol{\M}_{1}^{< E}(\ul{L},\ul{D}),\Q) \to
 H_1(\ol{\M}_{1}^{< E}(\ul{L},\ul{D}), \M_{1}^{< E}(\ul{L},\ul{D})_0,\Q) \\ \to
 H_0( \M_{1}^{< E}(\ul{L},\ul{D})_0, \partial \ol{\M}_{1}^{< E}(\ul{L},\ul{D}),
 \Q)
 \cong H_0( \M_{1}^{< E}(\ul{L},\ul{D})_0 - \partial \ol{\M}_{1}^{<
   E}(\ul{L},\ul{D}),\Q)  .\end{multline}
The final isomorphism is by excision.  The image of $[ \ol{\M}_{1}^{<
E}(\ul{L},\ul{D})]$ in relative homology
$H_0( \M_{1}^{< E}(\ul{L},\ul{D})_0, \partial \ol{\M}_{1}^{< E}(\ul{L},\ul{D}),
 \Q)$ 
vanishes by Theorem \ref{fundclass3}, hence the corollary.
\end{proof}

\begin{proposition}  \label{product} {\rm (Product axiom for fundamental classes)}  
If $\Gamma$ is the combinatorial type of a true boundary component
corresponding to a strip connecting node of infinite length and
$\Gamma_1,\Gamma_2$ are the trees obtained by cutting the single edge
of $\Gamma$ and the incoming and outgoing markings do not lie on the
same disk component then $ [{\M}_{\Gamma}(\ul{L},\ul{D})]$ is the
image of $[{\M}_{\Gamma_1}(\ul{L},\ul{D})] \times
[{\M}_{\Gamma_2}(\ul{L},\ul{D})] $ under the isomorphism of moduli
spaces ${\M}_{\Gamma}(\ul{L},\ul{D}) \to
{\M}_{\Gamma_1}(\ul{L},\ul{D}) \times_{\cI(L_0,L_1)}
{\M}_{\Gamma_2}(\ul{L},\ul{D}) $ and the restriction map $H_0(
{\M}_{\Gamma_1}(\ul{L},\ul{D}) \times
{\M}_{\Gamma_2}(\ul{L},\ul{D})) \to H_0(
{\M}_{\Gamma_1}(\ul{L},\ul{D}) \times_{\cI(L_0,L_1)}
{\M}_{\Gamma_2}(\ul{L},\ul{D})) $ that sends $0$-cycles not in
${\M}_{\Gamma_1}(\ul{L},\ul{D}) \times_{\cI(L_0,L_1)}
{\M}_{\Gamma_2}(\ul{L},\ul{D})$ to $0$.
\end{proposition} 

\begin{proof} 
By the (Cutting edges) and (Product) axiom for domain-dependent almost
complex structures in Proposition \ref{bmmaps},
${\M}_{\Gamma}(\ul{L},\ul{D}) $ is the fiber product of the moduli
spaces ${\M}_{\Gamma_1}(\ul{L},\ul{D})$ and
${\M}_{\Gamma_2}(\ul{L},\ul{D}) $ over $\cI(L_0,L_1)$.  Orientations
are compatible with the (Cutting edges) morphism by Theorem
\ref{main}, hence the statement of the Proposition.
 \end{proof}

\subsection{Floer operator} 

The Floer operator is defined by a count of elements in the zero-dimensional component
of the moduli space of Floer trajectories:

\begin{definition}  {\rm (Floer operator)} Given $x_+ \in \cI(L_0,L_1)$ define
$\partial \bra{x_+} \in CF(L_0,L_1)$ as the sum
\begin{equation} \label{coboundary} 
 \partial \bra{x_+} =
 \sum_{[u] \in
   {\M}_0(\ul{L},\ul{D},x_+,x_-)} \eps([u]) q^{E([u])}
 \sigma([u]) \bra{x_-} 
\end{equation}
and extend to $\partial : CF(L_0,L_1) \to CF(L_0,L_1)$ by linearity.
\end{definition} 

In general the Floer operator fails to square to zero because of
configurations involving disk bubbles; in Fukaya-Oh-Ohta-Ono
\cite{fooo} the Lagrangian Floer complex is called {\em obstructed} in
this case.  In general the operator $\partial$ is part of an
\ainfty-bimodule structure on $CF(L_0,L_1)$ in \cite{fooo}.  In the
case of admissible Lagrangians in Definition \ref{admissible}, in the
first case we assume that $J_\Gamma$ is $\iota_b$-anti-invariant on
$\ell_b^{-1}(\infty)$ as in Remark \ref{antiinv}, while in the second
we assume that $J_\Gamma$ is spherical on $\ell_b^{-1}(\infty)$ as in
Remark \ref{spheres}.

\begin{theorem} \label{squarezero} Suppose that $L_0,L_1$ are
  admissible Lagrangian branes as in Definition \ref{admissible}.  For
  regular, coherent, stabilizing, ghost-marking-independent
  collections of perturbation data, the Floer coboundary operator
  $\partial$ is well-defined and satisfies $\partial^2 = 0$.
\end{theorem} 

\begin{proof}
Since the zero-dimensional component of the moduli space is a finite
set of points for each energy bound by the first part of Proposition
\ref{zeroone}, the sum in \eqref{coboundary} is well-defined.  The
boundary of the fundamental class on the one-dimensional component of
the moduli space is a sum of points whose sum of coefficients is zero,
since two points with opposite coefficients occur for each one-cell.
The contributions corresponding to fake boundary types cancel by
Theorem \ref{fundclass3}.  It follows that the sum of coefficients of
the true boundary types also vanishes:
\begin{equation} \label{zerosout} 0 = \sum_{\Gamma \in \T} \sum_{ [u] \in
  {\M}_{\Gamma}(\ul{L},\ul{D},x_+,x_-)} \sigma([u]) q^{E([u])}
\eps([u])
\end{equation} 
where $\T$ is the set of true boundary types.  By Theorem \ref{main}
(for fixed point sets of anti-symplectic involutions) and Remark
\ref{spheres2} (for graphs of symplectomorphisms) the contributions
from the boundary components corresponding to disk bubbles at infinity
on $L_0$ or $L_1$ cancel in the right-hand of \eqref{zerosout}, so
that the contributions corresponding to breaking of Floer trajectories
also satisfy \eqref{zerosout}.  Each
$[u] \in {\M}_{\Gamma}(\ul{L},\ul{D},x_+,x_-)$ above can be written as
a concatenation $[u_1] \# [u_2]$ of stable Floer trajectories
$[u_1] \in \M_{\Gamma_1}(\ul{L},\ul{D},x_+,y),[u_2] \in
\M_{\Gamma_2}(\ul{L},\ul{D},y,x_-)$
for some $y \in \cI(L_0,L_1)$, as in Proposition \ref{product} and
\eqref{concat}, where each $\Gamma_k$ has one vertex for $k = 1,2$.
Because of the additional possibilities in re-ordering the markings of
each type, each glued trajectory $u_1 \# u_2$ corresponds to
$$(|\Gamma_1| + |\Gamma_2|)!/ |\Gamma_1|!|\Gamma_2|! = 
\sigma([u_1 \# u_2]) / \sigma([u_1]) \sigma([u_2]) $$ 
 trajectories in $\M_{\Gamma}(\ul{L},\ul{D})$.  As in \ref{product} the
 orientations and energies satisfy 
$$ \eps([u_1 \# u_2]) = \eps([u_1])
 \eps([u_2]), \quad E([u_1 \# u_2]) =
 E([u_1]) + E([u_2]) .$$  
Hence for $x_+ \in \cI(L_0,L_1)$,
\begin{eqnarray*}
 \partial^2 \bra{x_+}
 &=& 
 \sum_{\substack{ [\Gamma_1],[\Gamma_2], x_-,y \in \cI(L_0,L_1) \\ [u_1] \in
     {\M}_{\Gamma_1}(\ul{L},\ul{D},x_+,y) \\ [u_2] \in
     \M_{\Gamma_2}(\ul{L},\ul{D},y,x_-)}} \sigma([u_1]) \sigma([u_2])
 q^{E([u_1]) + E([u_2])} \eps([u_1]) \eps([u_2]) \bra{x_-} \\ 
&=& \sum_{\substack{ [\Gamma] \in \cT,x_- \in \cI(L_0,L_1),\\ [u_1 \# u_2] \in
     \M_\Gamma(\ul{L},\ul{D},x_+,x_-)}} \sigma([u_1] \# [u_2]) q^{E([u_1 \#
   u_2])} \eps([u_1 \# u_2]) \bra{x_-} \ \ = \ 0 
 \end{eqnarray*}
as claimed.
\end{proof} 

\begin{corollary} Let $\phi: X \to X$ be a non-degenerate Hamiltonian diffeomorphism
generated by the time-one flow of a time-dependent Hamiltonian $H \in
C^\infty(\R \times X)$.  The boundary operator $\partial$ for $CF(H)
:= CF(\Delta, (1 \times \phi) \Delta)$ satisfies $\partial^2 = 0$,
hence the Floer cohomology $HF(H) = \on{ker}(\partial)/
\on{im}(\partial)$ is well-defined.
\end{corollary} 

\subsection{Clean intersections} 

Recall that a pair $L_0,L_1 \subset X$ of submanifolds intersect {\em
  cleanly} if $L_0 \cap L_1$ is a smooth manifold and $ T(L_0 \cap
L_1) = TL_0 \cap TL_1 .$ Floer homology for clean intersections was
constructed in Pozniak \cite{po:cl} and also Schm\"aschke
\cite[Section 7]{schmaschke} under certain monotonicity
assumptions. In this section we show how to extend the monotone
results to the case that the union of the cleanly-intersecting
Lagrangians is rational using stabilizing divsiors.  We have in mind
especially the case that the two Lagrangians are rational and equal.
In the particular case of diagonal boundary conditions, we show that
the Floer cohomology is the singular cohomology with Novikov
coefficients.

The definition of Floer cohomology in the clean intersection case is a
cout of configurations of holomorphic strips, disks, spheres, and
Morse trees as in, for example, Biran-Cornea \cite[Section 4]{bc:ql}.
We suppose that $L_0 \cap L_1$ is clean and $L_0 \cup L_1$ is
rational, that is, some power of the line-bundle-with-connection
$\ti{X}$ is trivializable over $L_0 \cup L_1$. For example, if $L_0 =
L_1$ is rational then the intersection is clean and the union is
rational.  Let $F: L_0 \cap L_1 \to \R$ be a Morse function.  By the
Morse lemma, the critical set
$$\cI(L_0,L_1) := \crit(F) = \{l \in L_0 \cap L_1 \ | \d F(l) = 0 
\} $$
is necessarily finite.  Choose a generic metric $G$ on $L_0 \cap L_1$,
and let $\phi_t$ be the time $t$ flow of
$-\grad(F) \in \Vect(L_0 \cap L_1)$.  Denote the stable and unstable
manifolds of $F$:
$$ W_x^\pm = \left\{ l \in L_0 \cap L_1 \ | \lim_{t \to \pm \infty} 
\phi_t(l) = x \right\} .$$
The space of Floer cochains is then as before
$$ CF(L_0,L_1) =  \bigoplus_{x \in \cI(L_0,L_1)} \Lambda \bra{x} .$$

The Floer coboundary counts combinations of holomorphic disks and
gradient segments for the Morse function on the intersection.
We assume that $(F,G)$ is {\em Morse-Smale}, that is, the stable and
unstable manifolds meet transversally
$$ T_l W_x^+ + T_l  W_y^-  = T_l (L_0 \cap L_1), \quad 
\forall l\in W_x^+ \cap W_y^-, \ x,y \in \crit(F) .$$
Choose a compatible almost complex structure $J$ on $X$.  Given a
stable strip $C_0$ with boundary markings $z_-,z_+$ let
$w_1,\ldots,w_k \in C_0$ denote the nodes appearing in any
non-self-crossing path between $z_-$ and $z_+$.  Define a topological
space 
$$C =  C_0 \sqcup \cup_{i = 1}^k [0,\ell(w_i)] / \sim $$ 
by replacing each node $w_i$ by a segment $T_i \cong [0,\ell(w_i)]$ of
length $\ell(w_i)$.  Denote by 
$$T = T_1 \cup \ldots \cup T_k \quad S = \ol{C - T} $$ 
the {\em tree resp. surface part} of $C$.  A Floer trajectory is then
a map from $C = S \cup T$ that is $J$-holomorphic on the surface part
and a $F$-gradient trajectory on each segment in $T$.  See Figure
\ref{treedstrip}.

\begin{figure}[h!] 
\begin{picture}(0,0)%
\includegraphics{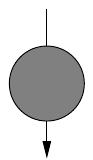}%
\end{picture}%
\setlength{\unitlength}{3947sp}%
\begingroup\makeatletter\ifx\SetFigFont\undefined%
\gdef\SetFigFont#1#2#3#4#5{%
  \reset@font\fontsize{#1}{#2pt}%
  \fontfamily{#3}\fontseries{#4}\fontshape{#5}%
  \selectfont}%
\fi\endgroup%
\begin{picture}(891,1224)(3974,-3373) 
\put(4557,-2414){\makebox(0,0)[lb]{\smash{{\SetFigFont{12}{14.4}{\rmdefault}{\mddefault}{\updefault}{$L_0 \cap L_1$}%
}}}}
\put(4570,-3279){\makebox(0,0)[lb]{\smash{{\SetFigFont{12}{14.4}{\rmdefault}{\mddefault}{\updefault}{$L_0 \cap L_1$}%
}}}}
\put(3989,-2800){\makebox(0,0)[lb]{\smash{{\SetFigFont{12}{14.4}{\rmdefault}{\mddefault}{\updefault}{$L_0$}%
}}}}
\put(4850,-2782){\makebox(0,0)[lb]{\smash{{\SetFigFont{12}{14.4}{\rmdefault}{\mddefault}{\updefault}{$L_1$}%
}}}}
\end{picture}%
\caption{A treed strip with Lagrangian boundary conditions}
\label{treedstrip}
\end{figure}

A perturbation datum for Floer trajectories consists of a perturbation 
of the Morse function and the almost complex structure.  The universal 
treed strip can be written as the union of one-dimensional and 
two-dimensional parts 
$$ \ol{{\U}}_\Gamma = \ol{{\S}}_\Gamma \cup \ol{{\T}}_\Gamma $$ 
so that $\ol{\S}_\Gamma \cap \ol{\T}_\Gamma$ is the set of points on
the boundary of the disks meeting the edges of the tree.  As before,
we fix {\em thin parts} of the universal curves: a neighborhood
$\ol{{\T}}_{\Gamma}^{\thin}$ of the endpoints and a neighborhood
$\ol{{\S}}_{\Gamma}^{\thin}$ of the markings and nodes.  In the
regularity construction, these neighborhoods must be small enough so
that either a given fiber is in a neighborhood of the boundary, where
transversality has already been achieved, or otherwise each segment
and each disk or sphere component in a fiber meets the complement of
the chosen thin parts.  For an integer $l \ge 0$ a {\em
  domain-dependent perturbation} of $F$ of class $C^l$ is a $C^l$ map
\begin{equation} \label{FGam}  F_{\Gamma}: \ol{{\T}}_{\Gamma}
 \times (L_0 \cap L_1) \to \R \end{equation}
equal to the given function $F$ away from the endpoints:
$$ F_\Gamma | \ol{{\T}}_{\Gamma}^{\thin} = 
\pi_2^* F $$
where $\pi_2$ is the projection on the second factor in \eqref{FGam}.
A {\em domain-dependent almost complex structure} of class $C^l$ for
treed disks of type $\Gamma$ is a map from the two-dimensional part
$\ol{{\S}}_{\Gamma} $ of the universal curve $\ol{{\U}}_{\Gamma}$ to $
\J_\tau(X)$ given by a $C^l$ map
$$ J_{\Gamma} : \ \ol{{\S}}_{\Gamma} \times X \to \End(TX) $$
equal to the given $J_D$ away from nodes and boundary:
$$ J_\Gamma | \ol{{\S}}_\Gamma^{\thin} = 
\pi_2^* J_D $$
and equal to the given stabilizing almost complex structure $J_D$ on
the boundary.  A {\em Floer trajectory} for the pair $(L_0,L_1)$
consists of a treed disk $C$ and a map $u: C = S \cup T \to X$ such
that
\vskip .1in
\begin{enumerate} 
\item[] {\rm (Boundary condition)} The Lagrangian boundary condition
  holds $u ( \partial S \cup T) \subset L_0 \cup L_1$.
\vskip .1in
\item[] {\rm (Surface equation)} On the surface part $S$ of $C$ the map
  $u$ is $J$-holomorphic for the given domain-dependent almost complex
  structure: if $j$ denotes the complex structure on $S$ then
$$ J_{\Gamma,u(z),z} \ \d u |_S = \d u |_S \ j. $$ 
\vskip .1in
\item[] {\rm (Boundary tree equation)} On the boundary tree part $T
  \subset C$ the map $u$ is a collection of gradient trajectories:
$$ \dds u |_T = -\grad_{ F_{\Gamma, (s,u(s))} }(u |_T) $$
where $s$ is a local coordinate with unit speed.  Thus for each edge
$e \in \Edge_{f}(\Gamma)$ the length of the trajectory is given by the
length $u |_{e \subset T}$ is equal to $\ell(e)$.
\end{enumerate} 
\noindent Given a stabilizing divisor $D \subset X - (L_0 \cup L_1)$,
one says that a stable trajectory $u:C \to X$ is {\em adapted} if and only if
\begin{enumerate} 
\item[] {\rm (Stable domain property)} $C$ is a stable marked strip;
  and 
%
\vskip .1in
\item[] {\rm (Marking property)} Each interior marking lies in
  $u^{-1}(D)$ and each component of $u^{-1}(D)$ contains an interior
  marking.
\end{enumerate} 
Let $\ol{\M}(\ul{L},D)$ denote the set of isomorphism classes of
stable $D$-adapted Floer trajectories to $X$, and by
$\M_\Gamma(\ul{L},D)$ the subspace of combinatorial type $\Gamma$.
Compactness and transversality properties of the moduli space of Floer
trajectories in the case of clean intersection, including exponential
decay estimates, can be found in \cite{wo:gdisk} and
\cite{schmaschke}. The necessary gluing result can be found in
Schm\"aschke \cite[Section 7]{schmaschke}.  The compactness and
transversality results for Floer trajectories allow the definition of
Floer cohomology by counted treed strips: as before in
\eqref{coboundary}, define
\begin{equation} \label{coboundary3} CF(L_0,L_1) = \bigoplus_{l \in
    \cI(L_0,L_1)} \Lambda l, \ \partial \bra{x_+} = \sum_{[u] \in
    {\M}_0(\ul{L},\ul{D},x_+,x_-)} \eps([u]) q^{E([u])} \sigma([u])
  \bra{x_-}.
\end{equation}

\begin{remark} {\rm (The case of equal Lagrangians)}  
  In the special case that the boundary conditions on the two sides
  are equal, the Floer operator may be identified with a count of
  treed holomorphic disks.  That is, let $L_0 = L_1 = L$.  By removal
  of singularities for pseudoholomorphic maps with Lagrangian boundary
  conditions \cite[Section 4.2]{ms:jh}, any holomorphic strip with
  boundary values in $(L,L)$ extends to a holomorphic disk with
  boundary values in $L$.  Thus the Floer operator counts
  configurations of treed holomorphic disks in $X$ with boundary in
  $L$ together with Morse flow lines in $L$.  The Fukaya algebra for
  an arbitrary such Lagrangian is constructed using stabilizing
  divisors in \cite{fuk}.
\end{remark} 

\begin{remark} {\rm (Floer cohomology for the diagonal)}   
\label{diagonal} Continuing Remark \ref{spheres2},
  consider the case of the diagonal Lagrangian. Let $L_0 = L_1 =
  \Delta$ where $\Delta \subset X = Y^- \times Y$ is the diagonal
  Lagrangian in a product symplectic manifold $Y^- \times Y$.  In this
  case, there are two natural regularization procedures involving
  stabilizing divisors.

  The first regularization procedure involves the interpretation of
  these configurations as holomorphic strips.  In this scheme, one
  chooses a Donaldson hypersurface $D \subset Y^- \times Y$ disjoint
  from the diagonal.  In particular, this means that the intersection
  points with the Donaldson hypersurface always intersect the interior
  of the strip. On the other hand, a strip with values in a product as
  above naturally defines a holomorphic cylinder, hence a holomorphic
  sphere by removal of singularities.  Any sphere in $Y$ may be
  regularized by choosing a Donaldson hypersurface in $Y$ and a
  domain-dependent almost complex structure on $Y$, as in Remark
  \ref{spheres}.  Let $\ti{\U}_\Gamma \to \ti{\M}_\Gamma$ denote the
  universal treed cylinder.  As in Remark \ref{spheres2}, there are
  forgetful maps
 $$ f_\Gamma: \M_\Gamma \to \ti{\M}_\Gamma, \quad \phi_\Gamma:
 \U_\Gamma \to \ti{\U}_\Gamma .$$
  The fiber of the forgetful map $f_\Gamma$ maybe described as
  follows.  Given a disk component $(C_i,\ul{w}_i,\ul{z}_i)$ with
  boundary special points $\ul{w}_i$ and interior special points
  $\ul{z}_i$, the corresponding sphere component
  $(\ti{C}_i, \ul{\ti{z}}_i)$ has interior special points arising from
  both boundary and interior special points on $C_i$.  If $C_i$ has at
  most two boundary points, then the dimension of the moduli space
  containing $(\ti{C}_i, \ul{\ti{z}}_i)$ is strictly greater than that
  containing $(C_i,\ul{w}_i,\ul{z}_i)$, since the difference in
  dimensions of the automorphism groups of the disk and sphere is
  three.  As explained before in Remark \ref{spheres2}, the fibers of
  the forgetful map are always positive dimensional.

Suppose regular perturbations have been chosen making uncrowded moduli
spaces of treed holomorphic cylinders regular.  By pulling back under
the forgetful maps, one obtains perturbations for strips making all
uncrowded moduli spaces of treed holomorphic cylinders regular.  As a
result, treed holomorphic cylinders with at least one cylinder can
never be isolated, because of the fiber $(S^1)^k, k \ge 1$ of the
projection map where $k$ is the number of cylinders.

This argument implies that the Floer operator, for this particular
perturbation scheme, is the same as the Morse operator.  So
$HF(\Delta,\Delta) = H(L;\Lambda)$ is the Morse cohomology with
Novikov coefficients.  In Section \ref{relative} we show that the
Floer cohomology is independent of the perturbation scheme and
Hamiltonian perturbation, as long as the perturbation is chosen so
that the intersection is clean.
\end{remark}

\subsection{Open invariants}

This section describes how the previous stabilization setting can be
used to define open Gromov-Witten invariants without relying on a
virtual perturbation setup as in Solomon \cite{solomon:thesis} or
Georgieva \cite{georgieva:orient}.  The open Gromov-Witten invariants
are defined for Lagrangians that are fixed point sets of
anti-symplectic involutions satisfying certain conditions.

We introduce the following notation for moduli spaces of stable treed
disks.  For integers $n\geq 1$ and $k \geq 0$, let
$\ol{\mathcal{M}}_{k,n}^{\mathrm{disks}}$ be the space of stable
connected $(k,n)$-marked nodal disks with markings
$(\ul{x},\ul{z}) = (x_1, \ldots, x_k, z_1, \ldots, z_n)$. That is, the
markings $z_i$, $1\leq i \leq n$, are interior markings as before, but
there are now $k$ (say ordered) boundary markings $x_i$,
$0\leq i \leq k$.  We enlarge the moduli
$\ol{\mathcal{M}}_{k,n}^{\mathrm{disks}}$ as in Section \ref{bm} by
adding {\em treed disks} and denote the space of stable $(k,n)$-marked
treed disks as $\ol{\mathcal{M}}_{k,n}$. A {\em treed disk} of type
$\Gamma$ is a $(k,n)$-marked disk $(C,\ul{z})$ of combinatorial type
$\Gamma$ together with a metric
$\ell: \mathrm{Edge}_{<\infty,d}(\Gamma) \to [0,\infty]$. We will no
longer see the boundary markings as punctures, so that every disk
component of a marked treed disk is seen as a marked disk instead of a
punctured disk.

Open Gromov-Witten invariants are defined by counting
pseudoholomorphic curves with certain Lagrangian boundary conditions.
Let $(X,\omega)$ be a rational compact symplectic manifold and
$L\subset X$ a rational Lagrangian brane that is the fixed locus of an
anti-symplectic involution $\iota: X \to X$. Moreover, assume that the
minimal Maslov number is divisible by four and $L$ admits an
$\iota$-equivariant relative spin structure as in \cite[Corollary 1.6
(a)]{fooo:anti}.  Let $D\in X - L$ resp. $D_\iota$ be a weakly
stabilizing resp. stabilizing divisor for $L$ as constructed in
Section \ref{stab} with $J_D \in \mathcal{J}_\tau(X,\omega)$ resp.
$J_{D_\iota} \in \mathcal{J}_\iota(X,\omega)$ adapted to $D$
resp. $D_\iota$.  As in Remark \ref{families}, since $D$ and $D_\iota$
can be chosen to come from homotopic trivializing sections over $L$,
one can find a family of symplectically isotopic divisors $D_t \in X -
L$, $t \in [0,1]$, such that $D_1 = D$ and $D_0=D_\iota$ together with
a family of almost-complex structures $J_t$, $t \in [0,1]$, such that
$J_1=J_D$ and $J_0=J_{D_\iota}$.

As before, one can achieve transversality via domain-dependent almost
complex structures.  Consider a coherent family of perturbation data
$\ul{J} = (J_\Gamma)_{\Gamma}$ as in Definition \ref{coherent} that
satisfy the conditions of Remarks \ref{families} and \ref{anti} with
the distances of the components now being measured relative to the
component containing the marking $z_1$ (instead of the strip
components). In particular the $J_\Gamma$ are equal to $J_t$ around
the markings and at the boundary of disks components at distance
$-\mathrm{log}(t)$ from the component containing $z_1$.  We restrict
to the case $k = 0$.  Choose a combinatorial type $\Gamma$ of maps
from $(0,n)$-marked nodal treed disks of a certain combinatorial type
to $(X,L)$. One can consider as in Section \ref{sect_adap}, the spaces
$\mathcal{M}_{\Gamma}(X,L,\ul{J},\ul{D})$ of adapted
$J_\Gamma$-holomorphic disks with boundary on $L$.  Again, the fact
that the divisors $D$ and $D_\iota$ are weakly stabilizing
resp. stabilizing for $L$ ensures that those spaces can be
compactified by considering stable domains.

\begin{theorem}
Suppose that $\Gamma$ is an uncrowded combinatorial type of stable
disk trajectories of index $i(\Gamma)\leq 3$. Suppose that regular
coherent perturbation data $J_{\Gamma'}$ for types $\Gamma'>\Gamma$
are given. Then there exists a comeager subset
$\mathcal{J}_\Gamma^{\mathrm{reg}}(X,\ul{D}) \subset
\mathcal{J}_\Gamma(X,\ul{D}) $ of {\em regular} perturbation data,
compatible with the given types on the boundary strata, such that if
$J_\Gamma \in \mathcal{J}_\Gamma^{\mathrm{reg}}(X,\ul{D})$, then
$\mathcal{M}_{\Gamma}(X,L,\ul{J},\ul{D})$ is a smooth compact manifold
of the expected dimension, that is, it is a finite number of points if
$i(\Gamma)=3$ and it is empty whenever $i(\Gamma)<3$.
\end{theorem}

The proof of the above transversality statement is obtained by the
transversality proof of Theorem \ref{main} for trajectories with index
at most one. The compactness part can be achieved as in Theorem
\ref{compthm}. The open Gromov-Witten invariants are counts of points
of $\mathcal{M}_{\Gamma}(X,L,\ul{J},\ul{D})$ using signs $\eps(u)$
obtained by a fixed reference orientation on $\ol{\mathcal{M}}_{k,n}$.
Let $\beta = \pi_2(\Gamma) \in \Pi(X,L)$ be the sum of the the
$\Pi(X,L)$ labels on the vertices of $\Gamma$ and
$[\partial \beta] \in \pi_1(L)$ be the homotopy class of the boundary
of $\pi_2(\Gamma)$. Since the index $i(\Gamma)$ only depends on
$\beta$, given a class $\beta \in \Pi(X,L)$, we write the resulting
index as $i(\beta)$.

\begin{definition}
 \label{opendef}
 For $\beta \in \Pi(X,L)$ is such that $i(\beta)=3$ and
 $[\partial \beta] \neq 0 \in \pi_1(L)$, set
$$ \tau_{X,L}(\beta) :=\tau_{X,L}(\beta,\ul{J},\ul{D}) := \frac{1}{n(\beta)!}
\sum_{\pi_2(\Gamma) = \beta } \sum_{u \in
  \mathcal{M}_\Gamma(X,L,\ul{J},\ul{D})} \eps(u) \in \Q$$
\end{definition}

We sketch an argument that the open Gromov-Witten invariants defined
above are independent of the choice of regular perturbation data
$\ul{J}$ and the choice of divisors $\ul{D}$.  Suppose first that two
regular choices of coherent perturbations $\ul{J}^{(0)},\ul{D}^{(0)},
$ and $\ul{J}^{(1)},\ul{D}^{(1)}$ are smoothly homotopic through a
$1$-parameter family $\ul{J}^{(s)},\ul{D}^{(s)}$, $s\in [0,1]$. In
particular, the stabilizing divisors have the same degree and are
built from homotopic trivializing sections over $L$.  For
$i(\beta)=3$, consider the moduli spaces
$$\ol{\mathcal{M}}(X,L,\beta,\ul{J}^{(s)},\ul{D}^{(s)}) =
\underset{\substack{\pi_2(\Gamma) = \beta\\ s\in [0,1]}}{\bigcup}
\mathcal{M}_{\Gamma}(X,L,\ul{J}^{(s)},\ul{D}^{(s)}).$$
As in the proof of Theorem \ref{main} regarding dimension one trajectories
and their gluings, for a comeager subset of homotopies
$\ul{J}^{(s)},\ul{J}^{(s)}$,
$\ol{\mathcal{M}}(X,L,\beta,\ul{J}^{(s)},\ul{D}^{(s)})$ has a
$1$-dimensional cell complex structure.  Furthermore, the zero
skeleton $\ol{\mathcal{M}}(X,L,\beta,\ul{J}^{(s)},\ul{D}^{(s)})_0$ of
the non-circle components is the union
$$\ol{\mathcal{M}}(X,L,\beta,\ul{J}^{(k)},\ul{D}^{(k)})= 
\bigcup_{\pi_2(\Gamma) = \beta}
\mathcal{M}_\Gamma(X,L,\ul{J}^{(k)},\ul{D}^{(k)}),\quad
k \in \{ 0, 1 \}$$
and a union of fake boundary strata
$ \cup_{f=1}^{F}
\mathcal{M}_{\Gamma_f}(X,L,\ul{J}^{(s_f)},\ul{D}^{(s_f)}).$
The fake boundary strata correspond to configurations with domains in
a codimension one strata of $\ol{\mathcal{M}}_{0,n(\beta)}$ made of
nodal disks having exactly one node of length $0$ or $\infty$.
Consider first the length $0$ case.  The gluing arguments in the proof
of Theorem \ref{main} imply that the configuration can be glued by
both (Collapsing an edge) or (Making an edge non-zero). An orientation
transfers consistently over such stratum in the sense that it induces
opposite orientations on the two glued families.  On the other hand,
in the length $\infty$ case, as in Proposition \ref{diagonal}, the
anti-symplectic involution acts on a pair of points of
$\bigcup_{f=1}^{F}
\mathcal{M}_{\Gamma_f}(X,L,\ul{J}^{(s_f)},\ul{D}^{(s_f)})$
in such a way that an orientation induces opposite orientations on the
two corresponding glued families.  Combining these arguments gives
that
$$\tau_{X,L}(\beta,\ul{J}^{(0)},\ul{D}^{(0)}) =
\tau_{X,L}(\beta,\ul{J}^{(1)},\ul{D}^{(1)}) .$$
If the data $\ul{J}^{(0)},\ul{D}^{(0)}$ and
$\ul{J}^{(1)},\ul{D}^{(1)}$ are not homotopic, for instance if
$D^{(0)}$ and $D^{(1)}$ have different degrees, one can reduce the
problem to the homotopic case as in \cite[Theorem 1.3]{cm:trans},
using homotopic stabilizing divisors $\ul{D}^{(0)'}$ and
$\ul{D}^{(1)'}$ of the same sufficiently high degree that are
respectively $\ul{J}^{(0)}$ and $\ul{J}^{(1)}$ almost-complex and
$\epsilon$-transverse to $\ul{D}^{(0)}$ and $\ul{D}^{(1)}$.

The open Gromov-Witten invariants can be also defined without the use
of treed disks as follows.  In the above setting, taking $D=D_\iota$,
that is, choosing $D$ to be a $\iota$-invariant stabilizing divisor,
and choosing the perturbations $\ul{J}$ to be $\iota$-anti-invariant
on disk components allows the cancellation of the contribution of the
configurations having a disk bubble.

\section{Relative invariants}
\label{relative} 

In this section we construct maps between Lagrangian Floer cohomology
groups associated to surfaces with strip like ends.  We show that
these maps are independent of all choices; in the case that the
surface is simply an infinite strip, it follows that the Lagrangian
Floer cohomology is independent, up to isomorphism, of Hamiltonian
perturbation and independent of the choice of stabilizing divisor and
perturbation data used to construct it.  We also show that in the case
of the diagonal, the perturbation scheme above constructed from
holomorphic spheres gives the same Floer cohomology, up to isomorphism
as a perturbation constructed from holomorphic strips with diagonal
boundary conditions.  This suffices to show that in case of diagonal
boundary conditions, the Floer cohomology is the singular cohomology
with Novikov coefficients.

\subsection{Relative invariants} 

Relative invariants are maps between Floer cohomology groups defined
by counting perturbed pseudoholomorphic sections of a symplectic
fibrations with strip-like ends.

Recall the basic notations for surfaces with strip-like ends from, for
example, Seidel \cite{se:book}.  A surface with strip like ends
consists of a surface with boundary $\Sigma$ equipped with a complex
structure $j: T \Sigma \to T\Sigma$, and a collection of embeddings
$$ \kappa_e : \pm(0,\infty) \times [0,1] \to \Sigma, \quad i =
0,\ldots, n $$
such that $\kappa_e^* j$ is the standard almost complex structure on
the strip, and the complement of the union of the images of the maps
$\kappa_e$ is compact.  Any such surface has a canonical
compactification $\ol{\Sigma}$ with the structure of a compact surface
with boundary obtained by adding a point at infinity along each strip
like end and taking the local coordinate to be the exponential of
$ \pm 2 \pi i \kappa_e $.

Our surfaces with strip-like ends will be labelled by Lagrangian
boundary conditions and equipped with Hamiltonian perturbations.
Denote the connected components of $\Sigma$ by
$(\partial \Sigma)_i, i = 0,\ldots, m$.  We assume for simplicity that
each boundary component $(\partial \Sigma)_i$ is contractible (rather
than a circle) and that for each $i = 0,\ldots, m$ we have fixed a
Lagrangian brane $L_i$.  For each end $e = 1,\ldots, n$ we denote the
adjacent Lagrangian branes $L_{e,-}, L_{e,+}$.  For each end $e$ we
suppose that $ H_e \in C^\infty([0,1] \times X)$ be time-dependent
Hamiltonian perturbations such that $\phi_{H_e}(L_{e,-}) \cap L_{e,+}$
is a clean intersection and $\phi_{H_e}(L_{e,-}) \cup L_{e,+}$ is
rational for each end. 

\begin{figure}[h!]
\begin{picture}(0,0)%
\includegraphics{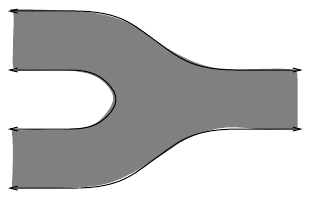}%
\end{picture}%
\setlength{\unitlength}{4144sp}%
\begingroup\makeatletter\ifx\SetFigFont\undefined%
\gdef\SetFigFont#1#2#3#4#5{%
  \reset@font\fontsize{#1}{#2pt}%
  \fontfamily{#3}\fontseries{#4}\fontshape{#5}%
  \selectfont}%
\fi\endgroup%
\begin{picture}(3030,1490)(1072,-400)
\put(4087,339){\makebox(0,0)[lb]{\smash{{\SetFigFont{12}{14.4}{\rmdefault}{\mddefault}{\updefault}{$H_{e_0}$}%
}}}}
\put(1087,859){\makebox(0,0)[lb]{\smash{{\SetFigFont{12}{14.4}{\rmdefault}{\mddefault}{\updefault}{$H_{e_1}$}%
}}}}
\put(1136,-113){\makebox(0,0)[lb]{\smash{{\SetFigFont{12}{14.4}{\rmdefault}{\mddefault}{\updefault}{$H_{e_2}$}%
}}}}
\put(2991,899){\makebox(0,0)[lb]{\smash{{\SetFigFont{12}{14.4}{\rmdefault}{\mddefault}{\updefault}{$L_0$}%
}}}}
\put(2898,-331){\makebox(0,0)[lb]{\smash{{\SetFigFont{12}{14.4}{\rmdefault}{\mddefault}{\updefault}{$L_2$}%
}}}}
\put(1975,335){\makebox(0,0)[lb]{\smash{{\SetFigFont{12}{14.4}{\rmdefault}{\mddefault}{\updefault}{$L_1$}%
}}}}
\end{picture}%
\caption{A surface with strip-like ends}
\label{surfwends}
\end{figure}

The construction of relative invariants begins with the extension of
the Hamiltonian perturbation and almost complex structures on the ends
over the entire surface.  Choose an extension of the Hamiltonian
perturbations at the ends over the strip as a one-form
$$K\in \Omega^1(\Sigma, C^\infty(X)), \quad \forall e \in \mE, \ \kappa_e^* K =
- H_e \d t .$$
Suppose almost complex structures $J_e:[0,1] \to \J(X,\omega) $ for
each end have been chosen compatible with the symplectic form
$\omega$.  Extend the almost complex structures $J_e$ to a compatible
almost complex structure 
$$J: \Sigma \to \J(X,\omega), \quad \kappa_e^* J(s,t) = J_e(t),
\forall e \in \mE, \ t \in [0,1], s \in \pm ( 0,\infty) .$$

The Hamiltonian-perturbed Cauchy-Riemann operator is the usual
Cauchy-Riemann operator for a modified almost complex structure on the
total space of a fibration.  Namely consider
$ E := \Sigma \times X $
as fiber bundle over $\Sigma$ with fiber $X$. Following
\cite[(8.1.3)]{ms:jh}, let $\pi_X: E \to X$ denote the projection on
the fiber. In local coordinates $s,t$ on $\Sigma$ define $K_s,K_t$
by $ K = K_s \d s + K_t \d t$.  Let
$$ \omega_{E} = \pi_X^* \omega - \pi^*_X \d K_s \wedge \d s - \pi_X^* \d
K_t \wedge \d t + (\partial_t K_s - \partial_s K_t) \d s \wedge \d t
.$$
The form $\omega_{E}$ is closed, restricts to the two-form
$\omega$ on any fiber, and defines the structure of a symplectic fiber
bundle on $E$ over $\Sigma$.  Consider the splitting
$TE \cong \pi_X^* TX \oplus (\Sigma \times \R^2) .$
Let $j_\Sigma: T \Sigma \to T\Sigma$ denote the standard complex
structure on $\Sigma$.  Define an almost complex structure on $E$ by
$$ J_E : T E \to TE, \quad (v,w) \mapsto ( ( J \hat{K} - \hat{K} j_\Sigma )w +
Jv, j_\Sigma w ) $$
where, as in Section \ref{floertraj},
$\hat{K} \in \Omega^1(\Sigma,\Vect(X))$ is the
Hamiltonian-vector-field-valued one-form associated to $K $.  See
Figure \ref{surfwends}.  A smooth
map $u: \Sigma \to X$ is $(J,K)$-holomorphic if and only if the
associated section $(\on{id} \times u): \Sigma \to E$ is
$J_E$-holomorphic \cite[Exercise 8.1.5]{ms:jh}.  Let
$ \ti{L}_i = (\partial \Sigma)_i \times L_i$ the fiber-wise Lagrangian
submanifolds of $E$ defined by $L_i$.  Then $u: \Sigma \to X$ has
boundary conditions in $(L_i, i = 1,\ldots, m)$ if and only if
$\on{id} \times u: \Sigma \to E$ has boundary conditions in
$(\ti{L}_i, i = 1,\ldots, m)$.  Thus, we have re-formulated
Hamiltonian-perturbed holomorphic maps with Lagrangian boundary
conditions as holomorphic sections.

The construction of the relative invariants proceeds, similar to
Floer's original work \cite{floer:lag}, \cite{floer:symp}, by
constructing moduli spaces of holomorphic sections of the bundle in
the previous paragraph.  Our regularization uses a divisor in the
total space of the fibration.  Suppose that $J_e \in \J(X,\omega)$ are
compatible almost complex structures stabilizing for
$\phi_{H_e,1}(L_{e,-}) \cup L_{e,+}$.  Let $\phi_{H_e,1-t}^* J_e$
denote the corresponding time-dependent almost complex structures, and
$\sigma_{k,e} : X \to \ti{X}^k$ are asymptotically $J_e$-holomorphic,
uniformly transverse sequences of sections with the property that
$D_{e} = \sigma_{k,e}^{-1}(0)$ are stabilizing for
$\phi_{H_e}(L_{e,-}) \cup L_{e,+}$ for $k$ sufficiently large.  The
pull-backs $\phi_{e,1-t}^* \sigma_{k,e}$ are then
$\phi_{e,1-t}^* J_e$-holomorphic, and any $(J_e,H_e)$-holomorphic
strip with boundary in $(L_{e,-}, L_{e,+})$ meets
$\phi_{e,1-t}^{-1}(\Sigma \times D_{e})$ in at least one point.
Denote by $\ti{E} \to E$ the pull-back of $\ti{X} \to X$ to the
fibration, equipped with the almost complex structure induced by the
given almost complex structure on $E$.

\begin{lemma}  
  Let $\Sigma$ be a surface with strip-like ends, let $L_i \subset X$
  be rational Lagrangians associated to the boundary components
  $(\partial \Sigma)_i$, and suppose that stabilizing divisors $D_e$
  for the ends $e = 1,\ldots, n$ of $\Sigma$ have been chosen as zero
  sets of asymptotically holomorphic sequences of sections
  $\sigma_{e,k}$ for $k$ sufficiently large.  There exists a
  asymptotically $J_E$-holomorphic, uniformly transverse sequence
  $\sigma_k: E \to \ti{E}^k$ with the property that for each end $e$,
  the pull-back $ \kappa_e^* \sigma_k( \cdot + s, \cdot, \cdot)$
  converges in $C^\infty$ uniformly on compact subsets to
  $\phi_{e,1-t}^* \sigma_{e,k}$ as $s \to \pm \infty$.  The zero
  set $D_E = \sigma_k^{-1}(0)$ is approximately holomorphic for $k$
  sufficiently large, asymptotic to
  $(1 \times \phi_{e,1-t})^{-1} ( \R \times [0,1] \times D_e)$ for
  each end $e = 1,\ldots, n$, and the intersection of $D_E$ with each
  boundary fiber $\pi^{-1}(z), z \in \partial \Sigma$ is stabilizing.
\end{lemma}

\begin{proof} We first compactify the fibration as follows.  Let
  $\H = \{ z \in \C \ | \ \on{Im}(z) \ge 0 \}$ denote complex
  half-space.  Let $\ol{E}$ be the fiber bundle over $\ol{\Sigma}$
  with fiber $X$ defined by gluing together
$$ U_0 = \Sigma \times X, \quad U_e= \H \times
X ,e = 1,\dots, n$$
using the transition maps $\kappa_e \times \phi_{H_e,1\pm t}$ on
$\H - \{ 0 \} \cong \R \times [0,1]$ from $U_0$ to $U_e$.  Denote the
projections over $U_e$ to $X$ by $\pi_{X,e}$.  The two-forms
$\omega_{E}$ on $U_0$, and $\pi_{X,e}^* \omega$ on $U_e$ glue together
to a two-form $\omega_E$ on $E$, making $\ol{E}$ into a symplectic
fiber bundle.

The fiber-wise symplectic form above may be adjusted to an honest
symplectic form by adding a pull-back from the base, and furthermore
adjusted so that the boundary conditions have rational union.  For the
first claim, since $\omega_E$ is fiber-wise symplectic, there exists a
symplectic form $\nu \in \Omega^2( \ol{\Sigma})$ with the property
that $\omega_E + \pi^* \nu$ is symplectic, where
$\pi: \ol{E} \to \ol{\Sigma}$ is the projection.  The almost complex
structure $J_E$ is compatible with $\omega_E + \pi^* \nu$, and equal
to the given almost complex structures $J_B \oplus J_e$ on the ends.
For the second claim, let $\ti{\Sigma} \to \ol{\Sigma}$ be a
line-bundle-with-connection whose curvature is
$\nu \in \Omega^2(\ol{\Sigma})$.  Let $\ti{E} \to \ol{\Sigma}$ be a
line-bundle-with-connection whose curvature is $\omega_E + \pi^* \nu$.
Denote by $\ti{L}_i$ the closure of the image of
$(\partial \Sigma)_i \times L_i$ for $i = 0,\ldots,m$.  Fix
trivializations of $\ti{E}$ over
$\ti{L}_{e_-} \cap \ti{L}_{e,+} \cong L_{e,-} \cap L_{e,+}$ for each
end $e$.  By assumption, the line bundle $\ti{E}$ is trivializable
over $L_e$, hence also $\ti{L}_e$ by parallel transport along the
boundary components.  Let $e_k, k = 0,1$ be ends connected by a
connected boundary component labelled by $L_i$.  For any
$p_k \in \phi_{H_{e_k},1}(L_{e_k,-}) \cap L_{e_k,+}$ the parallel
transport $T(p_0,p_1) \in U(1)$ from $p_0$ to $p_1$ is independent of
the choice of path.  Indeed, any two paths differ up to homotopy by a
loop in $L_i$, which has trivial holonomy by assumption.  After
perturbation of the connection and curvature on $\ti{E}$, we may
assume that the parallel transports $T(p_0,p_1)$ are rational for all
choices of $(p_0,p_1)$.  After taking a tensor power of $\ti{E}$, we
may assume that the parallel transports $T(p_0,p_1)$ are trivial,
hence $\ti{E}$ admits a covariant constant section $\tau$ over the
union $\ti{L}_e, e = 1,\ldots, n$.

Donaldson's construction \cite{don:symp} implies the existence of a
symplectic hypersurface in the total space of the fibration.  We show
that the hypersurface may be taken to equal the pullback of one of the
given ones over a neighborhood of the boundary of the base, as
follows.  Let $\sigma_{m,k}: X \to \ti{X}^k$ be an asymptotically
holomorphic sequence of sections concentrating on $L_e$ for
$e = 1,\ldots, n$.  Let $\sigma_{e,k}: X \to \ti{X}^k$ be an
asymptotically holomorphic sections of sections concentrating on
$\phi_{H_e,1}(L_{e,-}) \cup L_{e,+}$ and $\sigma_{i,k}$ an
asymptotically holomorphic sequence of sections concentrating on
$L_i$, both asymptotic to the given trivializations on the Lagrangians
themselves.  For each point $z \in \ol{\Sigma}$, let
$\sigma_{z,k}: \ol{\Sigma} \to \ti{\Sigma}^k$ denote the Gaussian
asymptotically holomorphic sequence of sections of $\ti{\Sigma}^k$
concentrated at $z$ as in \eqref{approxhol}.  We may assume that the
images of the strip-like ends are disjoint.  Let $V_i$ be disjoint
open neighborhoods $(\partial \Sigma)_i - \cup_e \on{Im}(\kappa_e)$ in
$\Sigma$.  For each $p = (z,x) \in \ti{L}_e$, let $ \sigma_{x,k}$ be
either equal to $\sigma_{e,k}$, for $z \in \on{Im}(\phi_e)$ or
otherwise equal to $\sigma_{i,k}$ if $b$ lies in $V_i$.  Let $P_k$ be
a set of points in $\partial \ol{\Sigma}$ such that the balls of
$g_k$-radius $1$ cover $\partial \ol{\Sigma}$ and any two points of
$P_k$ are at least distance $2/3$ from each other, where $g_k$ is the
metric defined by $k \nu$.  The desired asymptotically-holomorphic
sections are obtained by taking products of asymptotically-holomorphic
sections on the two factors: Write
$\theta(z,x) \tau(z,x) = \sigma_{x,k}(x) \boxtimes \sigma_{z,k}(z)$ so
that $\theta(x,z) \in \C$ is the scalar relating the two sections.
Define
\begin{equation} \label{approxhol} \sigma_k = \sum_{p \in P_k}
  \theta(z,x)^{-1} \sigma_{x,k} \boxtimes \sigma_{z,k} .\end{equation}
Then by construction the sections \eqref{approxhol} are asymptotically
holomorphic, since each summand is.

It remains to achieve the uniformly transverse condition.  Recall from
\cite{don:symp} that a sequence $(s_k)_{k \ge 0}$ is {\em uniformly
  transverse} to $0$ if there exists a constant $\eta$ independent of
$k$ such that for any $x \in X$ with $|s_k(x)| < \eta$, the derivative
of $s_k$ is surjective and satisfies $| \nabla s_k(x)| \ge \eta$.  The
sections $\sigma_k$ are asymptotically $J_E$-holomorphic and uniformly
transverse to zero over $\partial \ol{\Sigma} \times X$, since the
sections $\sigma_{i,k}$ and $\sigma_{e,k}$ are uniformly transverse to
the zero section. Hence $\sigma_k$ is also uniformally transverse over
a neighborhood of $\partial \Sigma$.  Pulling back to $E$ one obtains
an asymptotically holomorphic sequence of sections of
$\ti{E} | \Sigma$ that is uniformally transverse in a neighborhood of
infinity, that is, except on a compact subset of $E$.  Donaldson's
construction \cite{don:symp} although stated only for compact
manifolds, applies equally well to non-compact manifolds assuming that
the section to be perturbed is uniformly transverse on the complement
of a compact set.  The resulting sequence $\sigma_{E,k}$ is
uniformally transverse and consists of asymptotically holomorphic
sections asymptotic to the pull-backs of $\sigma_{e,k}$ on the ends.
The divisor $D_E = \sigma_{E,k}^{-1}(0)$ is approximately holomorphic
for $k$ sufficiently large and equal to the given divisors
$ \pm (0,\infty) \times D_e$ on the ends, by construction, and
concentrated at $L_i$, over each boundary component
$(\partial \Sigma)_i$.
\end{proof}

We warn the reader that the Donaldson hypersurface constructed above
does not stabilize bubbles in arbitrary fibers.  Indeed in general the
Donaldson hypersurface in the total space of the fibration constructed
above will not be of product form, and the projection to the base will
have singular values.  In particular, the intersection of $D_E$ with
$\{ z \} \times X$ could, for some $z \in \Sigma$, be singular and
cannot be used to stabilizes sphere bubbles in these fibers.  However,
the genericity assumptions below will guarantee that sphere and disk
bubbles meet $D_E$ transversally and so do not meet these singular
points.

Because of the perturbation required to make the Donaldson
hypersurface almost complex, the projection to the base is no longer
an almost complex map.  In particular, the projection of a component
lying in the small neighborhood of a fiber need no longer be a point.
However, the following holds for the disk components in the perturbed
moduli space: For each perturbed stable strip $u: C \to E$, consider
the projection $\pi \circ u: C \to \Sigma$.  For each disk component
$C_k \subset C$, the long exact sequence of homotopy groups implies
that the homotopy class of $\pi \circ u | C_k$ is classified by its
winding number.  The boundary conditions now imply that any disk with
boundary in $\ti{L}_i$ has homotopically trivial projection, any strip
connecting the positive end of $E$ with itself has homotopically
trivial projection, while any strip $C_k$ connecting the two ends of
$E$ has projection $\pi \circ u | C_k$ homotopically trivial to the
identity.

A perturbation scheme similar to the one for Floer trajectories makes
the moduli spaces transverse.  Choose a tamed almost complex structure
$J_E \in \J_\tau(E, \omega_E + \pi_B^* \nu)$ making $D_E$ holomorphic,
so that $D_E$ contains no holomorphic spheres, each holomorphic sphere
meets $D_E$ in at least three points, and each disk with boundary in
$\ti{L}_{e_-} \cup \ti{L}_{e,+}$ meets $D_E$ in at least one point.
Since $D_E$ is only approximately holomorphic with respect to the
product complex structure, the complex structure $J_E$ will not
necessarily be of split form, nor will the projection to $\Sigma$
necessarily be $(J_E,j_\Sigma)$-holomorphic away from the ends.
Furthermore, choose domain-dependent perturbations $F_{\Gamma}$ of the
Morse functions $F_e$ on $\phi_{H_e,1}(L_{e,-}) \cap L_{e,+}$, so that
$F_\Gamma$ is a perturbation of $F_e$ on the segments that map to
$\phi_{H_e,1}(L_{e,-}) \cap L_{e,+}$.

Domain-dependent perturbations give a regularized moduli space of
stable adapted treed strips.  These are maps to $E$, homotopic to
sections, with the given Lagrangian boundary conditions and mapping
the positive resp. negative end of the strip to the positive
resp. negative end of $E$.  Let $\ol{\M}(\ul{L},\ul{D},\ul{P})$ denote
the moduli space of adapted stable treed strips in $E$ of this type.
For generic domain-dependent perturbations
$P_{\Gamma,E} = (F_{\Gamma,E},J_{\Gamma,E})$ on $E$, the moduli space
$\ol{\M}(\ul{L},\ul{D},\ul{P},(x_e))$ of perturbed maps to $E$ with
boundary in $(L_{E,0}, L_{E,1})$ and limits $(x_e)$ has zero and
one-dimensional components that are compact and smooth (as manifolds
with boundary) with the expected boundary.  In particular the boundary
of the one-dimensional moduli spaces
${\M}_1(\ul{L},\ul{D},\ul{P}, (x_e))$ are $0$-dimensional strata
$\M_\Gamma(\ul{L},\ul{D},\ul{P}, (x_e))$ corresponding to either a
Floer trajectory bubbling off on end, or a disk bubbling off the boundary.

\begin{remark} \label{arearel} {\rm (Area-energy relation)} Suppose
  that $\Sigma$ is a disk with points on the boundary removed and
  $u: \Sigma \to E$ is a smooth map with a continuous extension to a
  map $\ol{\Sigma} \to \ol{E}$, exponential convergence on the strip
  like ends to limits $x_e, e = 1,\ldots, n$ and boundary conditions
  $\ti{L}_i, i = 1,\ldots, m$.  Since $\omega_E$ is only fiber-wise
  symplectic, the area $A(u) = \int_\Sigma u^* \omega_E$ may be
  negative.

  We obtain a lower bound on the area as follows.  If the projection
  is pseudoholomorphic, then the composed map $\pi \circ u$ is
  pseudoholomorphic and one easily obtains the identity
  \begin{equation} \label{bound} A(u) = \int_\Sigma u^* \omega_E =
    \int_\Sigma u^*(\omega_E + \pi^* \nu) - \int_\Sigma \nu \ge -
    \int_\Sigma \nu .\end{equation}
  Indeed because of the boundary conditions $\pi \circ u$ maps
  isomorphically onto $\Sigma$.  In general, since the projection
  $\pi$ is not necessarily almost complex after perturbation, the
  composed map $\pi \circ u$ is not necessarily the identity.
  However, by the boundary condition the map on the boundary
  $\pi \circ u| \partial \Sigma$ is homotopic to the identity.  The
  maps from the disk relative its boundary are classified up to
  homotopy by their winding number, as a special case of the long
  exact sequence for relative homotopy groups.  It follows that the
  composed map $\pi \circ u$ is homotopic to the identity.  Hence
  $ \int u^* \pi^* \nu = \int \nu$.  This identity implies the area
  identity \eqref{bound} for maps $u: B \to E$ with the given boundary
  conditions.  The same identity \eqref{bound} holds for maps from
  stable holomorphic strips $u: C \to E$, since these consist of a
  principal component as above and a number of ``bubble components''
  for which the projection is homotopically trivial.  In particular
  the $\omega_E$-area of a stable holomorphic strip $u: C \to E$
  differs from the energy as a map to the total space by a universal
  constant.
\end{remark}

\begin{definition} {\rm (Chain-level relative invariants)} Given
  $D_E,P_{\Gamma,E} = (J_{\Gamma,E},F_{\Gamma,E})$ and with
  $\ul{D}_\pm$ the collection of stabilizing divisors for the incoming
  resp. outgoing ends, define
  \begin{multline} \label{psi} \psi: \bigotimes_{e \in \mE_-}
    CF(L_{e,-},L_{e,+},;H_e,\ul{D}_-,\ul{P}_-) \to \bigotimes_{e \in
      \mE_+} CF(L_{e,-},L_{e,+}; H_e,\ul{D}_+, \ul{P}_+) , \\
    \otimes_{e \in \mE_-} \bra{x_e} \mapsto \sum_{x_e, e \in \mE_+}
    \sum_{[u] \in {\M}_0^{}(\ul{L},H,\ul{D},\ul{P}, \ul{x})} \eps([u])
    q^{A([u])} \sigma([u]) \otimes_{e \in \mE_+} \bra{x_e}
    .\end{multline}
\end{definition}

\begin{theorem} \label{chainmap2} The chain-level map $\psi$
  associated to the surface with strip-like ends $\Sigma$ is
  well-defined and a chain map: $\psi \partial =
\partial \psi$.
\end{theorem} 

\begin{proof} First we show that the sum \eqref{psi} is well-defined.
  By Remark \ref{arearel}, any area bound $A(u) < E_0$ implies an
  energy bound $E(u) < E_0$ as maps into the total space of the
  fibration $E| \Sigma$, viewed as a symplectic manifold with
  strip-like ends.  Gromov compactness implies that the moduli space
  $\ol{\M}^{< E_0}(\ul{L},H,\ul{D},\ul{P})$ with the given area bound
  is compact, hence has finite zero-dimensional component
  $\ol{\M}^{<E_0}_0(\ul{L},H,\ul{D},\ul{P})$.  It follows that the map
  $\psi$ is well-defined.  

  The proof of the chain relation is similar to the proof of Theorem
  \ref{squarezero}: the boundary of the one-dimensional components of
  the moduli space of perturbed trajectories
  $\ol{\M}_1(\ul{L},H,\ul{D},\ul{P})$ consists of configurations of
  type $\Gamma$ corresponding to a Floer trajectory bubbling off at
  the positive or negative end of the strip, or to a disk bubble
  forming at the boundary.  The latter cancel by Remark \ref{antiinv}
  or in the case of graphs as in Remark \ref{diagonal}, assuming the
  perturbation system is anti-invariant in the first case or spherical
  on the disk components with diagonal boundary conditions at infinite
  distance.
 \end{proof}

 Our main application of the relative invariants is the isomorphism of
 Floer cohomologies defined using different perturbation systems.
 Suppose that $\Sigma = \R \times [0,1]$ is an infinite strip, $L_0$
 and $L_1$ are equal and one of the Hamiltonian perturbations, say
 $H_-$, vanishes while the other $H_+$ is such that
 $\phi_{H_+,1}(L_0) \cap L_1$ is a transverse intersection.  In this
 case, the gradient trajectories on the positive end are constant and
 so may be ignored, while the gradient trajectories on the negative
 end are the Morse trajectories on $L_0 = L_1$.  Thus in this case one
 obtains the Lagrangians Piunikhin-Salamon-Schwarz  maps of
 Albers \cite{albers:lag}.  In particular, for $L_0 = L_1 = \Delta$ we
 obtain the usual Piunikhin-Salamon-Schwarz maps.

\begin{remark} \label{diagonal2} {\rm (Relative invariants for
    diagonal boundary conditions)} Similar relative invariants may be
  used to relate the Lagrangian Floer cohomology $HF(\Delta,\Delta)$
  defined using the {\em sphere} perturbation scheme with
  $HF(\Delta, \Delta)$ defined using the {\em strips} perturbation
  scheme.  Recall from Remark \ref{diagonal} that in the first case,
  the stabilizing divisor is a Donaldson hypersurface $D_Y \subset Y$,
  while in the second the stabilizing divisor
  $D_X \subset X = Y^- \times Y$ lies in the complement of the
  diagonal.  After perturbation we may suppose that $D_X$ is chosen
  transversally to $(D_Y \times Y)$ and $(Y \times D_Y)$ and
  intersects $D_Y \times D_Y$ transversally, hence trivially; see
  Remark \ref{marginal}.  Choose a base almost complex structure $J_D$
  on $X$ which, for disk components $C_i$ at distance greater than $2$
  is of split form $J_D |_{C_i} = J_{D,Y} \times J_{D,Y}$ while on
  components $C_i$ distance less than $1$ from the strip components,
  preserves $D_X, D_Y \times Y$ and $ Y \times D_Y$.  Note that on
  components of the latter type, $J_D$ will not be of split form.  Fix
  perturbations $P_\Gamma = (F_\Gamma,J_\Gamma)$ depending only on the
  markings mapping to $D_Y$, for any cylinder at infinite distance
  from the positive end, and only on the markings mapping to $D_X$ for
  any cylinder at infinite distance from the negative end; that is,
  pulled back under the relevant forgetful maps.  The same set-up as
  before defines a homotopy operator between the complexes defined
  using the two perturbation systems.
\end{remark}

A parametrized version of the moduli spaces leads to independence of
the relative invariants of all choices.  We first deal with
independence of the almost complex structure. Let $J_{E,0}, J_{E,1}$ be
almost complex structures on the total spaces $E$ as above, equal to
the given complex structures $J_e$ on the ends.  Contractibility of
the space of almost complex structures implies the existence of a
homotopy $J_{E,t}$ between $J_{E,0}$ and $J_{E,1}$.  The {\em
  parametrized moduli space} consists of pairs $(t,u: C \to X,\ul{z})$
where $(u: C \to X,\ul{z})$ is a stable marked treed $J_{E,t}$
holomorphic section on the surface part of $C$.  First suppose that
$J_{E,0},J_{E,1}$ both leave a stabilizing divisor $D$ invariant, that
is, $D$ is almost complex with respect to both $J_{E,0}$ and
$J_{E,1}$.  Perturbations $P_\Gamma = (J_\Gamma,F_\Gamma)$ are then
defined as before.  We assume that perturbations $J_{\Gamma,E,t}$ are
chosen so that at $t = 0$ resp. $t = 1$, the perturbation data only
depends on the intersections with $D_{E,0}$ resp. $D_{E,1}$;, that is,
are pulled back via forgetful maps forgetting markings mapping to
$D_{E,1}$ resp. $D_{E,0}$.  Denote by
$\widetilde{\M}(\ul{L},H,\ul{D},\ul{P})$ the parametrized moduli space of
such configurations.  The compactness and transversality properties of
the moduli space are similar to those given before, but now there are
additional boundary components: The boundary of the one-dimensional
component $\widetilde{\M}_1(\ul{L},H,\ul{D},\ul{P})$ of the parametrized
moduli space consists of configurations with either $t =0 $, $t = 1$,
Floer trajectories bubbling off the positive or negative ends of the
strip, or disk bubbles bubbling off the left or right boundary
components.  Define a homotopy operator
\begin{multline} h: \bigotimes_{e \in \mE_-}
  CF(L_{e,-},L_{e,+},;H_e,\ul{D}_-,\ul{P}_-) \to \bigotimes_{e \in
    \mE_+} CF(L_{e,-},L_{e,+}; H_e,\ul{D}_+, \ul{P}_+) , \\
  \otimes_{e \in \mE_-} \bra{x_e} \mapsto \sum_{x_e, e \in \mE_+}
  \sum_{[u] \in \widetilde{\M}_0(\ul{L},H,\ul{D},\ul{P}, \ul{x})}
  \eps([u]) q^{A([u])} \sigma([u]) \otimes_{e \in \mE_+} \bra{x_e}
  .\end{multline}

A modification of this construction allows the divisor to vary.  Given
two choices $D_{E,0}$ and $D_{E,1}$ that are sections of the same
bundle $\ti{E}$, Lemma \ref{auroux} \eqref{uniq} provides a family
$D_{E,t}$ such that $D_{E,t}$ is approximately $J_{E,t}$-holomorphic.
One then requires the markings for $(t,u)$ to map to $D_{E,t}$ and
obtains moduli spaces and a map as before.

A final modification deals with the case that the stabilizing divisors
are defined by sections of different bundles.  As in Remark
\ref{marginal}, there exist divisors $D_{E}', k = 0,1$ intersecting
$D_{E}$ $\eps$-transversally, such that $D_{E,0}'$ and $D_{E,1}'$
are sections of the same line bundle.  By Auroux's result
\cite[Theorem 2]{auroux:asym} $D_{E,0}', D_{E,1}'$ may be connected by
a family $D_{E,t}$.  Choosing perturbations depending only on the
intersections with $D_E$ at $t = k, k \in \{ 0, 1 \}$ map $h$ as in
the previous case.

\begin{theorem} Suppose that $L_i \subset X$ are admissible rational
  Lagrangian branes for $i = 0,\ldots, m$, and $\psi_0,\psi_1$
  relative maps defined using perturbations and divisors
  $\ul{P}_b, D_{E,b}$ over $\Sigma$ for $b \in \{ 0,1 \}$.  Then in each of
  the cases above the map $h$ is well-defined and a homotopy from
  $\psi_0$ to $\psi_1$:
  \begin{equation} \label{hop} 
\psi_1 - \psi_0 =  \partial_1 h +
    h \partial_0  .\end{equation}
As a result, the induced map in cohomology
$$[\psi]: \otimes_{e \in \mE_-} HF(L_{e,-},L_{e,+};\ul{D}_e,\ul{P}_e)
\to \otimes_{e \in \mE_+} HF(L_{e,-} ,L_{e,+}; \ul{D}_e, \ul{P}_e)$$
is independent of the all choices.  In case that
$\Sigma = \R \times [0,1]$ is the strip, the divisors
$\ul{D}_{e_-} = \ul{D}_{e_+}$ for the incoming and outgoing ends are
  equal and $\ul{P}_{e_-} = \ul{P}_{e_+}$, the induced map in
  cohomology is the identity map.
\end{theorem} 

\begin{proof} 
  The proof that $h$ is well-defined is a Gromov compactness argument
  similar to that in Theorem \ref{chainmap2}.  Only the contributions
  from boundary components corresponding to bubbling off Floer
  trajectories or the boundaries at $t = 0$ or $t = 1$ contribute,
  leading to the relation \eqref{hop}.  In the case
  $\Sigma = \R \times [0,1]$, $\ul{D}_- = \ul{D}_+$ and
  $\ul{P}_- = \ul{P}_+$, one may take for the divisor and perturbation
  data used to define the relative invariant the pull-back of the
  divisor and perturbation data used to define the Floer operator.  It
  follows that the only elements in the dimension-zero moduli space
  $\M(\ul{L},H,\ul{D},\ul{P})$ are constant configurations linking two
  equal intersection points $x_+ = x_-$.  Thus the chain level map
  $\psi: CF(L_0,L_1;\ul{D}_-,\ul{P}_-) \to CF(L_0,L_1; \ul{D}_+,
  \ul{P}_+) $ is the identity.

  The map for the first two cases shows that the invariants are
  independent of the choice of almost complex structure, as long as
  the stabilizing divisors are taken to be defined by sections of the
  same line bundle.  The last case shows that for a fixed almost
  complex structure, the maps are independent of the choice of line
  bundle used to define the stabilizing divisors.  Combining these
  observations gives independence of all choices.
\end{proof}

\subsection{Gluing laws and homotopies} 

The relative invariants satisfy a gluing law as in topological quantum
field theory.  Suppose that $\Sigma_{01},\Sigma_{12}$ are surfaces
with strip-like ends such that $\Sigma_{01}$ has the same number of
outgoing ends as $\Sigma_{12}$ has incoming ends.  Let $\Sigma_{02}$
denote a surface obtained by gluing together strip-like ends, as in
Figure \ref{gluefig}, by removing the intervals
$\pm (2T,\infty) \times [0,1]$ on each end for some real number
$T > 0 $ and gluing together the intervals $\pm (T,2T) \times [0,1]$.

\begin{figure}[h!]
\begin{picture}(0,0)%
\includegraphics{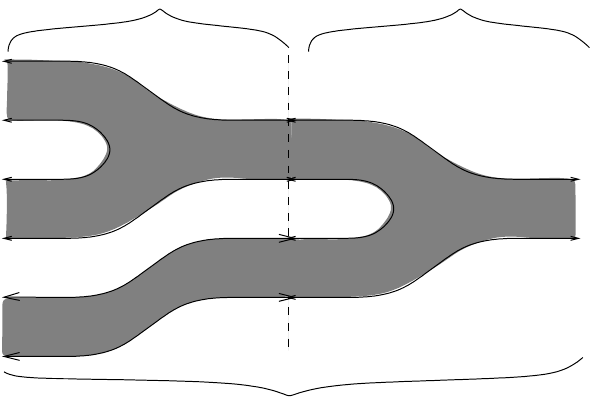}%
\end{picture}%
\setlength{\unitlength}{4144sp}%
\begingroup\makeatletter\ifx\SetFigFont\undefined%
\gdef\SetFigFont#1#2#3#4#5{%
  \reset@font\fontsize{#1}{#2pt}%
  \fontfamily{#3}\fontseries{#4}\fontshape{#5}%
  \selectfont}%
\fi\endgroup%
\begin{picture}(4503,3584)(1789,-1846)
\put(2710,1537){\makebox(0,0)[lb]{\smash{{\SetFigFont{12}{14.4}{\rmdefault}{\mddefault}{\updefault}{$\Sigma_{01}$}%
}}}}
\put(4920,1567){\makebox(0,0)[lb]{\smash{{\SetFigFont{12}{14.4}{\rmdefault}{\mddefault}{\updefault}{$\Sigma_{12}$}%
}}}}
\put(3630,-1773){\makebox(0,0)[lb]{\smash{{\SetFigFont{12}{14.4}{\rmdefault}{\mddefault}{\updefault}{$\Sigma_{02}$}%
}}}}
\end{picture}%
\caption{Gluing surfaces with strip-like ends}
\label{gluefig}
\end{figure}

\begin{theorem}
  Suppose that $H_e$ are Hamiltonian perturbations such that
  $\phi_{H_e}(L_{e,-}) \cap L_{e,+}$ is clean and
  $\phi_{H_e}(L_{e,-}) \cup L_{e,+}$ is rational for each end of
  $\Sigma_-$ and $\Sigma_+$.  Let $D_e$ be stabilizing divisors for
  $\phi_{H_e}(L_{e_-}) \cup L_{e_+}$ and $\ul{P}_e$ regular
  perturbations.  Then the relative invariants
$$ \psi_{ij}: \bigotimes_{e \in \cE_i}
CF(L_{e,-},L_{e,+},;H_e,\ul{D}_e,\ul{P}_e) \to \bigotimes_{e \in
  \cE_j} CF(L_{e,-},L_{e,+},;H_e,\ul{D}_e,\ul{P}_e) $$
where $\cE_i$ is the set of incoming ends of $\Sigma_{01}$ for
$i = 0$, the set of outgoing ends for $\Sigma_{01}$ for $i = 1$ and
the set of outgoing ends for $\Sigma_{12}$ for $i = 2$, satisfy the
gluing law $ [\psi_{12}] \circ [\psi_{01}] = [\psi_{02}] .$
\end{theorem}

The argument depends on the construction of another family of moduli
spaces which interpolate between the fiber products
$$ \M(\ul{L},H_{01},\ul{D}_{01},\ul{P}_{01} ) 
\times_{\prod_{e \in \cE_1} \cI(\phi_{H_e}(L_{e,-}),L_{e,+})}
\M(\ul{L},H_{12},\ul{D}_{12},\ul{P}_{12}) \quad \text{and} \quad
\M(\ul{L},H_{02},\ul{D}_{02},\ul{P}_{02} ) .$$
Namely consider a family of surfaces with strip-like ends $\Sigma_t$
for $t \in [0,1]$ obtained from $\Sigma_{01}$ and $\Sigma_{12}$ by
repeating the same gluing as before but gluing in necks of lengths
$1/t$ for each glued end.  If $\mE_0,\mE_1,\mE_2$ are the sets of ends
then $\Sigma_t$ has incoming ends $\mE_0$ and outgoing ends $\mE_2$.
Associate to the family $\Sigma_{012} = \cup_t \Sigma_t$ is a family
of fiber bundles
$$E_{012} =
\cup_{t \in [0,1]} (E_t := \Sigma_t \times X).$$
Let $J_{012} = (J_{012,t})$ denote a collection of almost complex
structures on $E_t$ connecting the almost complex structures obtained
from $J_{01}, J_{12}$ on the one hand and $J_{02}$ on the other.  Let
$\eta_{ij}$ denote the two-forms on $\Sigma_{ij}$ used to construct
the stabilizing divisors $D_{ij}$.  We may assume (by independence of
the choice of divisor) that $D_{02}$ is the divisor obtained by gluing
together $D_{01}$ and $D_{12}$.  A treed holomorphic section of
$E_{012}$ consists of a treed curve $C$, a map $u: C \to X$ that is a
$J_{012}$-holomorphic section on the surface part, is a gradient
trajectory on each edge.  The conditions for interior markings are
similar to those before: each interior markings is required to map to
$D_{012}$, and each connected component of the inverse image of
$D_{012}$ is required to contain at least one interior marking.  Let
$\ol{\M}(\ul{L}; H_{012}, \ul{D}_{012}, \ul{P}_{012})$ denote the
space of perturbed adapted stable treed sections of $E_{012}$ with
domain-dependent perturbations $\ul{P}_{012}$.  For generic choices of
perturbations, the zero and one-dimensional components of
$\ol{\M}(\ul{L}; H_{012}, \ul{D}_{012}, \ul{P}_{012})$ are compact
with the expected boundary.  As before, counting elements in the
zero-dimensional moduli space defines a homotopy operator
\begin{multline}
  h: \otimes_{e \in \mE_0} CF(L_{e,-},L_{e,+};H_e,\ul{D}_e,\ul{P}_e)
  \to \otimes_{e \in \mE_2} CF(L_{e,-},L_{e,+};H_e,\ul{D}_e,\ul{P}_e)  \\
  \bra{x_0} \mapsto \sum_{[u] \in
    \widetilde{\M}_0(\ul{L},H_{012},\ul{D}_{012},\ul{P}_{012},
    x_0,x_2)} \eps([u]) q^{A([u])} \sigma([u])
  \bra{x_2} \end{multline}
satisfying 
$ \psi_{12} \circ \psi_{01} - \psi_{02} = \partial_2 h +
 h \partial_0  .$
The Theorem follows.

\begin{remark}   {\rm (Products in Lagrangian Floer cohomology)}  
  If $L_0,L_1,L_2$ are admissible Lagrangian branes then the relative
  invariant for the half-pair-of-pants (triangle) gives rise to a
  relative invariant
  $HF(L_0,L_1) \otimes HF(L_1,L_2) \to HF(L_0,L_2)$.  A well-known
  argument for a surface with four strip-like ends, split in two
  different ways, implies the associativity relation for these
  products.  The category one obtains it the cohomology of the Fukaya
  category with objects restricted to admissible branes. Presumably
  one could use a similar argument to compare the perturbations
  constructed here with those of Fukaya-Oh-Ohta-Ono \cite{fooo}; the
  details of such a construction would be at least as lengthy as their
  construction.
\end{remark}

\begin{remark} {\rm (Periodic Floer cohomology)}  
  Continuing Remark \ref{diagonal} on the case that both Lagrangians
  are the diagonal correspondence, combining the gluing laws with
  invariance under perturbation implies that
$HF(\Delta, (1 \times
  \phi_H) \Delta) \cong HF(\Delta,\Delta)  $
  is isomorphic to the singular cohomology $H(Y,\Lambda) $ for any
  choice of Hamiltonian such that the intersection is
  $\Delta \cap (1 \times \phi_H) \Delta$ is transverse.  This
  completes the proof of Theorem \ref{introthm}.
\end{remark} 

\begin{corollary} For any compact symplectic manifold $(X,\omega)$ and
  time-dependent Hamiltonian $H \in C^\infty(X \times [0,1])$ whose
  time-one flow $\phi_1: X \to X$ has non-degenerate fixed point set
  $\Fix(\phi_1)$, the following inequality holds:
$$  \# \Fix(\phi_1) \ge \on{rank}(H(X,\Q)) $$
\end{corollary} 

\begin{proof} The proof is by a perturbation argument to reduce to the
  rational case.  Since $H^2(X,\Q)$ is dense in $H^2(X,\R)$ and the
  set of symplectic forms is open, there exists a small perturbation
  $\omega'$ of $\omega$ with rational cohomology class
  $[\omega'] \in H^2(X,\Q)$.  The time-one-flow $\phi_1': X \to X$ of
  $H$ with respect to $\omega'$ is a perturbation of $\phi_1$, and so
  still has non-degenerate fixed point set for $\omega'$ sufficiently
  close to $\omega$ and admits a canonical bijection
  $\Fix(\phi_1) \cong \Fix(\phi_1')$.  Indeed, since the set of
  two-forms $\omega'$ for which the fixed point set is non-degenerate
  is open, for $\omega'$ sufficiently close to $\omega$ there exists a
  path $\omega_t$ from $\omega'$ to $\omega$ so that the fixed point
  set $\Fix(\phi_1,\omega_t)$ is non-degenerate for each
  $t \in [0,1]$.  The union $\cup_{t \in [0,1]} \Fix(\phi_1,\omega_t)$
  is a one-manifold projecting submersively onto $[0,1]$, and so has
  the structure of a fiber bundle with discrete fibers.  Choosing a
  connection defines the bijection.  Thus as claimed we have 
$$ \# \Fix(\phi_1)  = \on{rank}(CF( (1 \times \phi_1)\Delta,\Delta)) \ge
\on{rank}(HF(\Delta,\Delta)) = \on{rank}H(X,\Q) .$$
\end{proof}

\def\polhk#1{\setbox0=\hbox{#1}{\ooalign{\hidewidth
      \lower1.5ex\hbox{`}\hidewidth\crcr\unhbox0}}} \def\cprime{$'$}


\begin{thebibliography}{10}


\bibitem{ab:ex} M.~Abouzaid.  \newblock Framed bordism and
  {L}agrangian embeddings of exotic spheres.  \newblock {\em Ann. of
    Math.} 175: 71--185, 2012.

\bibitem{albers:lag} P.~Albers.  \newblock A {L}agrangian 
  {P}iunikhin-{S}alamon-{S}chwarz Morphism and Two Comparison 
  Homomorphisms in {F}loer Homology.  Int. Math. Res. Notices (2008) 
  rnm134. 

\bibitem{afw} P.~Albers, B.~Filippenko, J.~Fish and K.~Wehrheim.
  \newblock Polyfold Proof of the Weak Arnold Conjecture.  \newblock
  Talk at the AMS Joint Meetings. 
\href{https://jointmathematicsmeetings.org/amsmtgs/2181_abstracts/1116-57-814.pdf}
{https://jointmathematicsmeetings.org/amsmtgs/2181$\_$abstracts/1116-57-814.pdf}

\bibitem{ar:gac2}
E.~Arbarello, M.~Cornalba, and P.~A. Griffiths.
\newblock {\em Geometry of algebraic curves. {V}olume {II}}, volume 268 of {\em
  Grundlehren der Mathematischen Wissenschaften}.
\newblock Springer, Heidelberg, 2011.
\newblock With a contribution by Joseph Daniel Harris.

\bibitem{auroux:asym}
D.~Auroux.
\newblock Asymptotically holomorphic families of symplectic submanifolds.
\newblock {\em Geom. Funct. Anal.}, 7(6):971--995, 1997.

\bibitem{auroux:remark} D.~Auroux.  \newblock A remark about
  Donaldson's construction of symplectic submanifolds.  \newblock {\em
    J. Symplectic Geom.} , 1:647--658, 2002.


\bibitem{auroux:complement}
D.~Auroux, D.~Gayet, and J.-P.~Mohsen.
\newblock Symplectic hypersurfaces in the complement of an isotropic
  submanifold.
\newblock {\em Math. Ann.}, 321(4):739--754, 2001.

\bibitem{be:gw}
K.~Behrend.
\newblock Gromov-{W}itten invariants in algebraic geometry.
\newblock {\em Invent. Math.}, 127(3):601--617, 1997.

\bibitem{bf:in}
K.~Behrend and B.~Fantechi.
\newblock The intrinsic normal cone.
\newblock {\em Invent. Math.}, 128(1):45--88, 1997.

\bibitem{bm:gw}
K.~Behrend and Yu. Manin.
\newblock Stacks of stable maps and {G}romov-{W}itten invariants.
\newblock {\em Duke Math. J.}, 85(1):1--60, 1996.

\bibitem{bc:ql}
P.~Biran and O.~Cornea.
\newblock Quantum structures for {L}agrangian submanifolds.
\newblock 
\newblock \href{http://www.arxiv.org/abs/0708.4221}{arxiv:0708.4221}.

\bibitem{bo:le}
D.~Borthwick, T.~Paul, and A.~Uribe.
\newblock Legendrian distributions with applications to relative {P}oincar\'e
  series.
\newblock {\em Invent. Math.}, 122(2):359--402, 1995.

\bibitem{charest:clust}
F.~{Charest}.
\newblock {Source Spaces and Perturbations for Cluster Complexes}.
\newblock 
\newblock \href{http://www.arxiv.org/abs/1212.2923}{arxiv:1212.2923}.

\bibitem{fuk}
F.~{Charest} and C. Woodward
\newblock {Fukaya algebras and stabilizing divisors}
\newblock 
\newblock \href{http://www.arxiv.org/abs/1505.08146}{arxiv:1505.08146}.



\bibitem{cm:trans}
K.~Cieliebak and K.~Mohnke.
\newblock Symplectic hypersurfaces and transversality in {G}romov-{W}itten
  theory.
\newblock {\em J. Symplectic Geom.}, 5(3):281--356, 2007.

\bibitem{cl:curves}
H.~Clemens.
\newblock Curves on generic hypersurfaces.
\newblock {\em Ann. Sci. \'Ecole Norm. Sup. (4)}, 19(4):629--636, 1986.

\bibitem{cl:clusters}
O.~Cornea and F. Lalonde.
\newblock Cluster homology: An overview of the construction and results.
\newblock {\em Electron. Res. Announc. Amer. Math. Soc.}, (12):1-12, 2006.

\bibitem{don:symp}
S.~K. Donaldson.
\newblock Symplectic submanifolds and almost-complex geometry.
\newblock {\em J. Differential Geom.}, 44(4):666--705, 1996.

\bibitem{don:floer} S.~K. Donaldson.  \newblock {\em {F}loer homology
  groups in {Y}ang-{M}ills theory}, volume 147 of {\em Cambridge
  Tracts in Mathematics}.  \newblock Cambridge University Press,
  Cambridge, 2002.  \newblock With the assistance of M. Furuta and
  D. Kotschick.

\bibitem{dragnev:fred}
D.~L.~Dragnev.
\newblock Fredholm theory and transversality for noncompact pseudoholomorphic
  maps in symplectizations.
\newblock {\em Comm. Pure Appl. Math.}, 57(6):726--763, 2004.

\bibitem{fhs:tr}
A.~{F}loer, H.~Hofer, and D.~Salamon.
\newblock Transversality in elliptic {M}orse theory for the symplectic action.
\newblock {\em Duke Math. J.}, 80(1):251--292, 1995.

\bibitem{floer:lag}
A.~Floer. 
\newblock Morse theory for Lagrangian intersections. 
\newblock {\em J. Differential Geom.} 28 (3): 513--547, 1988. 

\bibitem{floer:symp}
A.~Floer.
\newblock Symplectic fixed points and holomorphic spheres.
\newblock {\em Comm. Math. Phys.}, 120(4):575--611, 1989.



\bibitem{fooo}
K.~Fukaya, Y.-G.~Oh, H.~Ohta, and K.~Ono.
\newblock {\em Lagrangian intersection {F}loer theory: anomaly and
  obstruction.}, volume~46 of {\em AMS/IP Studies in Advanced Mathematics}.
\newblock American Mathematical Society, Providence, RI, 2009.

\bibitem{fooo:tech}
K.~{Fukaya}, Y.-G. {Oh}, H.~{Ohta}, and K.~{Ono}.
\newblock {Technical details on Kuranishi structure and virtual fundamental
  chain}.
\newblock \href{http://www.arxiv.org/abs/1209.4410}{arxiv:1209.4410}.

\bibitem{fooo:anti}
K.~Fukaya, Y.-G.~Oh, H.~Ohta, and K.~Ono. 
\newblock  Anti-symplectic involution and Floer cohomology. 
\newblock \href{http://www.arxiv.org/abs/0912.2646}{arxiv:0912.2646}.

\bibitem{fooo:toric1}
K.~Fukaya, Y.-G.~Oh, H.~Ohta, and K.~Ono.
\newblock Lagrangian {F}loer theory on compact toric manifolds. {I}.
\newblock {\em Duke Math. J.}, 151(1):23--174, 2010.

\bibitem{fo:arnold}
K.~Fukaya and K.~Ono.
\newblock Arnold conjecture and {G}romov-{W}itten invariant.
\newblock {\em Topology}, 38(5):933--1048, 1999.

\bibitem{gayet:re}
D.~Gayet.
\newblock Hypersurfaces symplectiques r\'eelles et pinceaux de {L}efschetz
  r\'eels.
\newblock {\em J. Symplectic Geom.}, 6(3):247--266, 2008.

\bibitem{georgieva:orient}
P.~Georgieva.
\newblock Open Gromov-Witten disk invariants in the presence of an anti-symplectic involution.
\newblock \href{http://www.arxiv.org/abs/1306.5019}{arxiv:1306.5019}.

\bibitem{gersten:trans}
A.~{Gerstenberger}.
\newblock {Geometric transversality in higher genus Gromov-Witten theory}.
\newblock \href{http://www.arxiv.org/abs/1309.1426}{arxiv:1309.1426}.

\bibitem{gs:iso}
Y.~{Groman}, J.~Solomon.
\newblock {A reverse isoperimetric inequality for $J$-holomorphic curves}.
\newblock \href{http://www.arxiv.org/abs/1210.4001}{arxiv:1210.4001}.


\bibitem{ha:al}
R.~Hartshorne.
\newblock {\em Algebraic Geometry}, volume~52 of {\em Graduate Texts in
  Mathematics}.
\newblock Springer-Verlag, Berlin-Heidelberg-New York, 1977.

\bibitem{ho:sc}
H.~Hofer, K.~Wysocki, and E.~Zehnder.
\newblock $sc$-smoothness, retractions and new models for smooth spaces.
\newblock {\em Discrete and Continuous Dyn. Systems}, 28:665--788, 2010.


\bibitem{ionel:vfc}
E.-N. {Ionel} and T.~H. {Parker}.
\newblock {A natural Gromov-Witten virtual fundamental class}.
\newblock \href{http://www.arxiv.org/abs/1302.3472}{arxiv:1302.3472}.



\bibitem{kn:proj2}
F.~F. Knudsen.
\newblock The projectivity of the moduli space of stable curves. {II}. {T}he
  stacks {$M_{g,n}$}.
\newblock {\em Math. Scand.}, 52(2):161--199, 1983.

\bibitem{kresch:can}
A.~Kresch.
\newblock Canonical rational equivalence of intersections of divisors.
\newblock {\em Invent. Math.}, 136(3):483--496, 1999.

\bibitem{litian:vir}
J.~Li and G.~Tian.
\newblock Virtual moduli cycles and {G}romov-{W}itten invariants of general
  symplectic manifolds.
\newblock In {\em Topics in symplectic $4$-manifolds (Irvine, CA, 1996)}, First
  Int. Press Lect. Ser., I, pages 47--83. Internat. Press, Cambridge, MA, 1998.

\bibitem{tian:floer}
G.~Liu and G.~Tian.
\newblock Floer homology and {A}rnold conjecture.
\newblock {\em J. Differential Geom.}, 49(1):1--74, 1998.

\bibitem{loc:ell} R.~B. Lockhart and R.~C. McOwen.  \newblock Elliptic
  differential operators on noncompact manifolds.  \newblock {\em
    Ann. Scuola Norm. Sup. Pisa Cl. Sci. (4)}, 12(3):409--447, 1985.

\bibitem{mau:gluing}
S.~Ma'u.
\newblock Gluing pseudoholomorphic quilted disks.
\newblock \href{http://www.arxiv.org/abs/0909.3339}{arxiv:0909.3339}.


\bibitem{mazja:est} V.~G. Maz'ja and B.~A. Plamenevskii.
  \newblock Estimates in {$L_{p}$} and in {H}\"older classes, and the
            {M}iranda-{A}gmon maximum principle for the solutions of
            elliptic boundary value problems in domains with singular
            points on the boundary.  \newblock {\em Math. Nachr.},
            81:25--82, 1978.

\bibitem{mcduff:kur}
D.~{McDuff} and K.~{Wehrheim}.
\newblock {Smooth Kuranishi atlases with trivial isotropy}.
\newblock 1208.1340.


\bibitem{ms:jh}
D.~McDuff and D.~Salamon.
\newblock {\em {$J$}-holomorphic curves and symplectic topology}, volume~52 of
  {\em American Mathematical Society Colloquium Publications}.
\newblock American Mathematical Society, Providence, RI, 2004.

\bibitem{oh:fl1}
Y.-G. Oh.
\newblock Floer cohomology of {L}agrangian intersections and pseudo-holomorphic
  disks. {I}.
\newblock {\em Comm. Pure Appl. Math.}, 46(7):949--993, 1993.

\bibitem{pardon:alg} J.~{Pardon}.  \newblock {An algebraic approach to
  virtual fundamental cycles on moduli spaces of J-holomorphic
  curves}.  
\newblock \href{http://www.arxiv.org/abs/1309.2370}{arxiv:1309.2370}.



\bibitem{pss}
S.~Piunikhin, D.~Salamon, and M.~Schwarz.
\newblock Symplectic {F}loer-{D}onaldson theory and quantum cohomology.
\newblock In {\em Contact and symplectic geometry (Cambridge, 1994)}, volume~8
  of {\em Publ. Newton Inst.}, pages 171--200. Cambridge Univ. Press,
  Cambridge, 1996.

\bibitem{royden} {H.L.~Royden}. 
\newblock {\it Real Analysis}.
\newblock Macmillan Publishing Company, New York, 1988.


\bibitem{po:cl}
M.~Po{\'z}niak. 
\newblock Floer homology, {N}ovikov rings and clean intersections. 
\newblock In {\em Northern {C}alifornia {S}ymplectic {G}eometry {S}eminar},
  volume 196 of {\em Amer. Math. Soc. Transl. Ser. 2}, pages 119--181. Amer. 
  Math. Soc., Providence, RI, 1999. 



\bibitem{schmaschke} F.~Schm{\"a}schke.  Floer homology of Lagrangians
  in clean intersection.  \newblock
  \href{http://www.arxiv.org/abs/1606.05327}{arXiv:1606.05327}

\bibitem{sch:coh} M.~Schwarz.  \newblock {\em Cohomology Operations
from $S^1$-Cobordisms in {F}loer Homology}.  \newblock \newblock
\href{http://www.math.uni-leipzig.de/~schwarz/}{PhD thesis}, ETH
Zurich, 1995 

\bibitem{se:book}
P.~Seidel.
\newblock Fukaya categories and Picard-Lefschetz theory.
\newblock {\em Zurich Lectures in Advanced Mathematics, EMS}, 2008.

\bibitem{se:gr}
P.~Seidel.
\newblock Graded {L}agrangian submanifolds.
\newblock {\em Bull. Soc. Math. France}, 128(1):103--149, 2000.

\bibitem{seidel:genustwo}
P.~Seidel.
\newblock Homological mirror symmetry for the genus two curve.
\newblock {\em J. Algebraic Geom.} 20:727--769, 2011.

\bibitem{solomon:thesis}
J.~Solomon.
\newblock Intersection theory on the moduli space of holomorphic curves with Lagrangian boundary conditions.
\newblock {\em MIT Thesis}.  \href{http://arxiv.org/abs/math.SG/0606429}{arXiv:math/0606429}

\bibitem{orient}
K.~Wehrheim and C.T. Woodward.
\newblock Orientations for pseudoholomorphic quilts.
\newblock 2012 \href{http://www.math.rutgers.edu/~ctw/orient.pdf}{preprint}.

\bibitem{wel:inv} J.~Y.~Welschinger.  
\newblock Invariants of real
  rational symplectic 4-manifolds and lower bounds in real enumerative
  geometry.  \newblock {\em C. R. Math. Acad. Sci. Paris}
  336:341--344, 2003.

\bibitem{wendl:hyp}
C.~{Wendl}.
\newblock {Contact Hypersurfaces in Uniruled Symplectic Manifolds Always
  Separate}.
\newblock \href{http://www.arxiv.org/abs/1202.4685}{arxiv:1202.4685}



\bibitem{wo:gdisk}
C.T.~Woodward.
\newblock Gauged {F}loer theory of toric moment fibers.
\newblock {\em Geom. and Func. Anal.}, 21:680--749, 2011.




\bibitem{zilt:dipl}
F.~{Ziltener}.
\newblock {Floer-Gromov-compactness and stable connecting orbits}.
\newblock
\href{http://www.staff.science.uu.nl/~zilte001/publications_preprints/Ziltener_diploma_thesis.pdf}{ETH Diploma Thesis}, 2002. 

\end{thebibliography}
\end{document}